\input ./style/arxiv-general.cfg
\documentclass[aap,MSNbibl,dvips]{arximspdf}
\makeatletter
   \@ifpackageloaded{graphicx}{}{\usepackage{graphicx}}
\makeatother

\doi{10.1214/15-AAP1097}
\volume{26}
\issue{1}
\pubyear{2016}
\firstpage{462}
\lastpage{506}
\docsubty{FLA}

\makeatletter

\newcommand{\rrvert}{\vert}

\newcommand{\llvert}{\vert}
\newcommand{\eqref}[1]{(\ref{#1})}
\newtheorem{theorem}{Theorem}
\newtheorem{condition}[theorem]{Condition}
\newtheorem{corollary}[theorem]{Corollary}
\newproclaim{definition}[theorem]{Definition}
\newproclaim{example}[theorem]{Example}
\newtheorem{lemma}[theorem]{Lemma}
\newtheorem{proposition}[theorem]{Proposition}
\newproclaim{remark}[theorem]{Remark}
\makeatother

\begin{document}
\begin{frontmatter}

\title{High-frequency asymptotics for Lipschitz--Killing curvatures of
excursion sets on the sphere}
\runtitle{Lipschitz--Killing curvatures of excursion sets on the sphere}

\begin{aug}
\author[A]{\fnms{Domenico}~\snm{Marinucci}\corref{}\ead[label=e1]{marinucc@mat.uniroma2.it}\thanksref{T1}}
\and
\author[B]{\fnms{Sreekar}~\snm{Vadlamani}}
\runauthor{D. Marinucci and S. Vadlamani}
\thankstext{T1}{Supported by the ERC Grants n. 277742 \emph{Pascal},
``Probabilistic
and Statistical Techniques for Cosmological Applications.''}
\affiliation{University of Rome Tor Vergata and TIFR-CAM}
\address[A]{Department of Mathematics\\
University of Rome Tor Vergata\\
Rome\\
Italy\\
\printead{e1}}
\address[B]{Tata Institute of Fundamental Research\\
Centre For Applicable Mathematics\\
Sharada Nagar, Chikkabommsandra\\
Bangalore 560065\\
Karnataka, India}
\end{aug}

%
\received{\smonth{2} \syear{2014}}
%
\revised{\smonth{11} \syear{2014}}

%
\begin{abstract}
In this paper, we shall be concerned with geometric functionals and
excursion probabilities for some nonlinear transforms evaluated on Fourier
components of spherical random fields. In particular, we consider both
random spherical harmonics and their smoothed averages, which can be viewed
as random wavelet coefficients in the continuous case. For such fields, we
consider smoothed polynomial transforms; we focus on the geometry of their
excursion sets, and we study their asymptotic behaviour, in the
high-frequency sense. We focus on the analysis of Euler--Poincar\'{e}
characteristics, which can be exploited to derive extremely accurate
estimates for excursion probabilities. The present analysis is
motivated by
the investigation of asymmetries and anisotropies in cosmological data. The
statistics we focus on are also suitable to deal with spherical random
fields which can only be partially observed, the canonical example being
provided by the masking effect of the Milky Way on Cosmic Microwave
Background (CMB) radiation data.
\end{abstract}

%
\end{frontmatter}

\section{Introduction}

\subsection{Motivations and general framework}

In this paper, we shall be concerned with geometric functionals and
excursion probabilities for some nonlinear transforms evaluated on Fourier
components of spherical random fields. More precisely, let $ \{
T(x), x\in S^{2} \} $ denote a Gaussian, zero-mean\break isotropic
spherical random field, that is, for some probability space $(\Omega
,\Im,P)$
the application $T(x,\omega)\rightarrow\mathbb{R}$ is $ \{
\mathcal{B}%
(S^{2})\times\Im \} $ measurable, where $\mathcal{B}(S^{2})$ denotes
the Borel $\sigma$-algebra on the sphere, and by isotropy we mean as usual
that for all rotation $g\in \operatorname{SO}(3)$, the field $ \{ T(x) \} $ has
the same law as $ \{ T^{g}(x):=T(gx) \} $. It is well known that
the following representation holds in the mean square sense (see, e.g., \cite{leonenko2,marpecbook,mal}):
%
\begin{equation}
T(x)=\sum_{\mathbb{\ell}m}a_{\mathbb{\ell}m}Y_{\mathbb{\ell
}m}(x)=
\sum_{%
\mathbb{\ell}}T_{\mathbb{\ell}}(x),\qquad T_{\mathbb{\ell}}(x)=\sum
_{m=-%
\mathbb{\ell}}^{\mathbb{\ell}}a_{\mathbb{\ell}m}Y_{\mathbb{\ell
}m}(x),
\label{norcia}
\end{equation}
where $ \{ Y_{\mathbb{\ell}m}(\cdot) \} $ denotes the family of
spherical harmonics, and $ \{ a_{\mathbb{\ell}m} \} $ the
array of
random spherical harmonic coefficients, which satisfy $\mathbb
{E}a_{\mathbb{%
\ell}m}\overline{a}_{\mathbb{\ell}^{\prime}m^{\prime}}=C_{\mathbb
{\ell}%
}\delta_{\mathbb{\ell}}^{\mathbb{\ell}^{\prime}}\delta
_{m}^{m^{\prime
}}$; here, $\delta_{a}^{b}$ is the Kronecker delta function, and the
sequence $ \{ C_{\mathbb{\ell}} \} $ represents the angular power
spectrum of the field. As pointed out in \cite{mp2012}, under isotropy the
sequence $C_{\mathbb{\ell}}$ necessarily satisfies $\sum_{\mathbb{\ell
}}%
\frac{(2\mathbb{\ell}+1)}{4\pi}C_{\mathbb{\ell}}=\mathbb
{E}T^{2}<\infty$
and the random field $T(x)$ is mean square continuous. Under the slightly
stronger assumption $\sum_{\mathbb{\ell}\geq L}(2\mathbb{\ell}+1)C_{%
\mathbb{\ell}}\leq O(\log^{-2}L)$, the field can be shown to be a.s.
continuous, an assumption that we shall exploit heavily below.

Our attention will be focused on the Fourier components $ \{
T_{\mathbb{%
\ell}}(x) \} $, which represent random eigenfunctions of the spherical
Laplacian:
\[
\Delta_{S^{2}}T_{\mathbb{\ell}}=-\mathbb{\ell}(\mathbb{
\ell}+1)T_{%
\mathbb{\ell}},\qquad \mathbb{\ell}=1,2,\ldots.
\]
A lot of recent work has been focused on the characterization of geometric
features for $ \{ T_{\mathbb{\ell}} \} $, under Gaussianity
assumptions; for instance, \cite{Wig1,Wig2} studied the asymptotic
behaviour of the nodal domains, proving an earlier conjecture by Berry on
the variance of {(functionals of)} the zero sets of $T_{\mathbb{\ell
}}$. In
an earlier contribution, \cite{BGS} had focused on the \textit{Defect} or
signed area, that is, the difference between the positive and negative regions;
a central limit theorem for these statistics and more general nonlinear
transforms of Fourier components was recently established by~\cite{MaWi3}.
These studies have been motivated, for instance, by the analysis of
so-called Quantum Chaos {(see again \cite{BGS})}, where the behaviour of
random eigenfunctions is taken as an approximation for the asymptotics in
deterministic case, under complex boundary conditions. More often, spherical
eigenfunctions emerge naturally from the analysis of the Fourier components
of spherical random fields, as in (\ref{norcia}). In the latter
circumstances, several functionals of $T_{\mathbb{\ell}}$ assume a great
practical importance: to mention a couple, the squared norm of
$T_{\mathbb{%
\ell}}$ provides an unbiased sample estimate for the angular power spectrum
$C_{\mathbb{\ell}}$,
\[
\mathbb{E} \biggl\{ \int_{S^{2}}T_{\mathbb{\ell}}^{2}(x)\,dx
\biggr\} =(2\mathbb{%
\ell}+1)C_{\mathbb{\ell}},
\]
while higher-order power lead to estimates of the so-called polyspectra,
which have a great importance in the analysis of non-Gaussianity (see, e.g.,
\cite{marpecbook}).

The previous discussion shows that the analysis of nonlinear
functionals of $%
 \{ T_{\mathbb{\ell}} \} $ may have a great importance for
statistical applications, especially in the framework of cosmological data
analysis. In this area, a number of papers have searched for deviations of
geometric functionals from the expected behaviour under Gaussianity. For
instance, the so-called Minkowski functionals have been widely used as tools
to probe non-Gaussianity of the field $T(x)$; see \cite{matsubara} and the
references therein. On the sphere, Minkowski functionals correspond to the
area, the boundary length and the Euler--Poincar\'{e} characteristic of
excursion sets, and up to constants they correspond to the Lipschitz--Killing
curvatures we shall consider in this paper; see \cite{RFG}, page 144. Many
other works have also focused on local deviations from the Gaussianity
assumption, mainly exploiting the properties of integrated higher order
moments (polyspectra); see \cite{pietrobon1,rudjord2}.

In general, the works aimed at the analysis of local phenomena are often
based upon wavelets-like constructions, rather than standard Fourier
analysis. The astrophysical literature on these issues is vast; see, for
instance, \cite{Wiaux,Starck} and the references therein.
Indeed, the
double localization properties of wavelets (in real and harmonic domain)
turn out usually to be extremely useful when handling real data.

In this paper, we shall focus on sequence of spherical random fields which
can be viewed as averaged forms of the spherical eigenfunctions, for example,
\[
\beta_{j}(x)=\sum_{\mathbb{\ell}}b \biggl(
\frac{\mathbb{\ell}}{B^{j}}%
 \biggr)T_{\mathbb{\ell}}(x),\qquad j=1,2,3\ldots
\]
for $b(\cdot)$ a weight function whose properties we shall discuss immediately.
The fields $ \{ \beta_{j}(x) \} $ can indeed be viewed as a
representation of the coefficients from a continuous wavelet transform
from $%
T(x)$, at scale $j$. More precisely, consider the kernel
\begin{eqnarray*}
\Psi_{j}\bigl( \langle x,y \rangle\bigr) &:=&\sum
_{\mathbb{\ell
}}b \biggl(%
\frac{\mathbb{\ell}}{B^{j}} \biggr)
\frac{2\mathbb{\ell}+1}{4\pi
}P_{\mathbb{%
\ell}}\bigl( \langle x,y \rangle\bigr)
\\
&=&\sum_{\mathbb{\ell}}b \biggl(\frac{\mathbb{\ell}}{B^{j}} \biggr)\sum
_{m=-%
\mathbb{\ell}}^{\mathbb{\ell}}Y_{\mathbb{\ell}m}(x)\overline
{Y}_{\mathbb{%
\ell}m}(y).
\end{eqnarray*}
Assuming that $b(\cdot)$ is smooth (e.g., $C^{\infty})$, compactly
supported in $%
[B^{-1},B]$, and satisfying the partition of unity property $\sum_{j}b^{2}(%
\frac{\mathbb{\ell}}{B^{j}})=1$, for all $\ell>B$, where $B$ is a fixed
``bandwidth'' parameter s.t. $B>1$. Then $\Psi
_{j}( \langle x,y \rangle)$ can be viewed as a continuous version
of the needlet transform, which was introduced by Narcowich et al. in
\cite%
{npw1}, and considered from the point of view of statistics and cosmological
data analysis by many subsequent authors, starting from \cite{bkmpAoS,mpbb08,pbm06}. In this framework, the following localization
property is now well known (see, e.g., \cite{npw1}, Theorem~3.5., \cite%
{gm1}, Lemma~4.1 or \cite{marpecbook}, Proposition~10.5): for all $M\in
\mathbb{N}$, there exists a constant $C_{M}$ (independent of $j$) such
that
%
\begin{equation}
\bigl\llvert \Psi_{j}\bigl( \langle x,y \rangle\bigr)\bigr\rrvert
\leq \frac{%
C_{M}B^{2j}}{ \{ 1+B^{j}d(x,y) \} ^{M}}, \label{localineq}
\end{equation}
where $d(x,y)=\arccos( \langle x,y \rangle)$ is the usual
geodesic distance on the sphere. Hence, the needlet field
\begin{eqnarray}\label{ptrf}
\beta_{j}(x) &=&\int_{S^{2}}\Psi_{j}
\bigl( \langle x,y \rangle\bigr)T(y)\,dy\nonumber
\\
&=&\int_{S^{2}}\sum_{\mathbb{\ell}m}b
\biggl(\frac{\mathbb{\ell}}{B^{j}} 
 \biggr)Y_{\mathbb{\ell}m}(x)
\overline{Y}_{\mathbb{\ell}m}(y)\sum_{\ell
^{\prime}m^{\prime}}a_{\ell^{\prime}m^{\prime}}Y_{\ell^{\prime
}m^{\prime}}(y)\,dy
\\
&=&\sum_{\mathbb{\ell}m}\sum_{\ell^{\prime}m^{\prime}}b
\biggl(\frac
{\mathbb{%
\ell}}{B^{j}} \biggr)a_{\ell^{\prime}m^{\prime}}Y_{\mathbb{\ell}%
m}(x)
\delta_{\ell}^{\ell^{\prime}}\delta_{m}^{m^{\prime}}=\sum
_{%
\mathbb{\ell}}b \biggl(\frac{\mathbb{\ell}}{B^{j}}
\biggr)T_{\mathbb
{\ell}%
}(x)\nonumber
\end{eqnarray}
is then only locally determined, that is, for $B^{j}$ large enough its value
depends only on the behaviour of $T(y)$ in a neighbourhood of $x$.
This is
a very important property, for instance, when dealing with spherical random
fields which can only be partially observed, the canonical example being
provided by the masking effect of the Milky Way on Cosmic Microwave
Background (CMB) radiation.

It is hence very natural to produce out of $ \{ \beta_{j}(x)
\} $
nonlinear statistics of great practical relevance. To provide a concrete
example, a widely disputed theme in CMB data analysis concerns the existence
of asymmetries in the angular power spectrum; it has been indeed often
suggested that the angular power $ \{ C_{\mathbb{\ell}} \} $ may
exhibit different behaviour for different subsets of the sky, at least over
some multipole range; see, for instance, \cite{hansen2009,pietrobon1}.
It is readily seen that
\[
\mathbb{E} \bigl\{ \beta_{j}^{2}(x) \bigr\} =\sum
_{\mathbb{\ell}}b \biggl(%
\frac{\mathbb{\ell}}{B^{j}} \biggr)
\frac{2\mathbb{\ell}+1}{4\pi
}C_{\mathbb{%
\ell}},
\]
which hence suggests a natural ``local''
estimator for a binned form of the angular power spectrum (note that the
right-hand side does not depend on $x$, as a consequence of isotropy). More
precisely, it is natural to consider some form of averaging and introduce
the process
%
\begin{equation}
\int_{S^{2}}K\bigl( \langle z,x \rangle\bigr)
\beta_{j}^{2}(x)\,dx, \qquad z\in S^{2}, \label{considef}
\end{equation}
where $K( \langle\cdot,\cdot \rangle)$ is some kernel function whose
properties we will discuss below; for instance, should we consider the
behaviour of the angular power spectrum on the northern and southern
hemisphere, we might focus on $z=N,S$, where $N,S$ denote, respectively, the
North and South Poles (compare \cite{hansen2009,pietrobon1,bennett2012,planckIS} and the references therein). In the rest of this
paper, we shall be concerned with centred and normalized versions of
(\ref%
{considef}), that is, processes of the form
%
\begin{equation}
g_{j;q}(z):=\int_{S^{2}}K\bigl( \langle z,x \rangle
\bigr)H_{q} \biggl(\frac{%
\beta_{j}(x)}{\sqrt{\mathbb{E}\beta_{j}^{2}(x)}} \biggr)\,dx, \label{mainproc}
\end{equation}
where $H_{q}(\cdot)$ is the Hermite polynomial of $q$th order; for instance,
for $q=3$ these processes could be exploited to investigate local variation
in Gaussian and non-Gaussian features (see \cite{rudjord2} and below for
more discussion and details).

\subsection{Main result}

The purpose of this paper is to study the asymptotic behaviour for the
expected value of the Euler characteristic and other geometric functionals
for the excursion regions of sequences of fields such as $ \{
g_{j;q}(\cdot) \} $, and to exploit these results to obtain excursion
probabilities in non-Gaussian circumstances. Indeed, on one hand these
geometric functionals are of interest by themselves, as they provide the
basis for implementing goodness-of-fit tests (compare~\cite
{matsubara}); on
the other hand, they provide the clue for approximations of the excursion
probabilities for $ \{ g_{j;q}(\cdot) \} $, by means of some weak
convergence results we shall establish, in combination with some now
classical arguments described in detail in the monograph \cite{RFG}.

It is important to stress that our results are obtained under a setting
which is quite different from usual. In particular, the asymptotic
theory is
investigated in the high frequency sense, for example, assuming that a single
realization of a spherical random field is observed at higher and higher
resolution as more and more refined experiments are implemented. This
is the
setting adopted in \cite{marpecbook}; see also \cite{anderes,steinm}
for the related framework of fixed-domain asymptotics.

Due to the nature of high-frequency asymptotics, we cannot expect the
finite-dimensional distributions of the processes we focus on to converge.
This will require a more general notion of weak convergence, as developed,
for instance, by \cite{davydov,dudley}. By means of this, we shall
indeed show how to evaluate asymptotically valid excursion probabilities,
which provide a natural solution for hypothesis testing problems. Indeed,
the main result of the paper, Theorem~\ref{thm:exc:prob}, provides a very
explicit bound for the excursion probabilities of non-Gaussian fields
such as
(\ref{mainproc}), for example,
%
\begin{eqnarray}\label{mainresult}
&&\limsup_{j\rightarrow\infty}\Bigl\llvert \Pr \Bigl\{ \sup
_{x\in
S^{2}}\tilde{%
g}_{j;q}(x)>u \Bigr\} - \bigl
\{ 2\bigl(1-\Phi(u)\bigr)+2u\phi(u)\lambda _{j;q} \bigr\} \Bigr\rrvert
\nonumber
\\[-8pt]
\\[-8pt]
\nonumber
&&\qquad\leq\exp \biggl( -\frac{\alpha
u^{2}}{2} \biggr),
\end{eqnarray}
where $\tilde{g}_{j;q}(x)$ has been normalized to have unit variance,
$\phi
(\cdot),\Phi(\cdot)$ denote standard Gaussian density and distribution
function, $%
\alpha>1$ is some constant and the parameters $\lambda_{j;q}$ have
analytic expressions in terms of generalized convolutions of angular power
spectra; see (\ref{porteaperte}), (\ref{porteaperte2}). See also \cite%
{NardiSiegYakir} for some related results on the distribution of maxima of
approximate Gaussian random fields; note, however, that our approach is
quite different from theirs and the tools we use allow us to get much
stronger results in terms of the uniform estimates.

\subsection{Plan of the paper}

The plan of the paper is as follows: In Section~\ref{sec:background}, we
review some background results on random fields and geometry, mainly
referring to the now classical monograph \cite{RFG}. Section~\ref%
{sec:spherical:field} specializes these results to spherical random fields,
for which some background theory is also provided, and provides some simple
evaluations for Lipschitz--Killing curvatures related to excursion sets for
harmonic components of such fields. More interesting Gaussian subordinated
fields are considered in Section~\ref{sec:gaussian:subordinate}, where some
detailed computations for covariances in general Gaussian subordinated
circumstances are also provided. Section~\ref{sec:weak:cgs} provides the
main convergence results, that is, shows how the distribution of these random
elements are asymptotically proximal (in the sense of \cite{davydov}) to
those of a Gaussian sequence with the same covariances. This result is then
exploited in Section~\ref{sec:application}, to provide the proof of %
\eqref{mainresult}. A number of possible applications on real cosmological
data sets are discussed throughout the paper.

\section{Background: Random fields and geometry}
\label{sec:background}

This section is devoted to recall basic integral geometric concepts, to
state the Gaussian kinematic fundamental formula, and to discuss its
application in evaluating the excursion probabilities. This theory has been
developed in a series of fundamental papers by R.~J. Adler, J.~E. Taylor and
coauthors (see \cite{adleraap,worsley,tayloradler2003,tayloradler2009,taylortakemuraadler,adleradvapp}), and it
is summarized in the monographs \cite{RFG,adlerstflour} which are
our main references in this Section (see also \cite{azaiswschebor,azaisbook} for a different approach,
and \cite{taylorvadlamani,chengxiao,adlersamo,adlerblanchet} for some further
developments in this area; applications to the sphere have also been
considered very recently by \cite{chengschwar,chengxiao2}).

\subsection{Lipschitz--Killing curvatures and Gaussian Minkowski functionals}

There are a number of ways to define Lipschitz--Killing curvatures, but
perhaps the easiest is via the so-called tube formula, which, in its
original form is due to Hotelling~\cite{Hotelling} and Weyl \cite
{Wey39}. To
state the tube formula, let $M$ be an $m$-dimensional smooth subset of $
\mathbb{R}^{n}$ such that $\partial M$ is a $C^{2}$ manifold endowed with
the canonical Riemannian structure on $\mathbb{R}^{n}$. The tube of
radius $%
\rho$ around $M$ is defined as
%
\begin{equation}
\operatorname{Tube}(M,\rho) = \bigl\{ x\in\mathbb{R}^{n}\dvtx d(x,M)\leq\rho
\bigr\} ,
\end{equation}
where
%
\begin{equation}
d(x,M)=\inf_{y\in M}\Vert x-y\Vert.
\end{equation}
Then according to Weyl's tube formula {(see \cite{RFG})}, the Lebesgue
volume of this constructed tube, for small enough $\rho$, is given by
%
\begin{equation}
\lambda_{n}\bigl(\operatorname{Tube}(M,\rho)\bigr)= \sum
_{j=0}^{m}\rho^{n-j}\omega
_{n-j}%
\mathcal{L}_{j}(M), \label{tube:formula}
\end{equation}
where $\omega_{j}$ is the volume of the $j$-dimensional unit ball and $
\mathcal{L}_{j}(M)$ is the $j$th-Lipschitz--Killing curvature (LKC)
of $M$%
. A little more analysis shows that $\mathcal{L}_{m}(M)=\mathcal{H}_{m}(M)$,
the $m$-dimensional Hausdorff measure of $M$, and that $\mathcal{L}_{0}(M)$
is the Euler--Poincar\'{e} characteristic of $M$. Although the
remaining LKCs
have less transparent interpretations, it is easy to see that they satisfy
simple scaling relationships, in that $\mathcal{L}_{j}(\alpha M)=\alpha
^{j}%
\mathcal{L}_{j}(M)$ for all $1\leq j\leq m$, where $\alpha M=\{x\in
\mathbb{R%
}^{n}\dvtx x=\alpha y\mbox{ for some }y\in M\}$. Furthermore, despite the fact that
defining the $\mathcal{L}_{j}$ via (\ref{tube:formula}) involves the
embedding of $M$ in $\mathbb{R}^{n}$, the $\mathcal{L}_{j}(M)$ are actually
intrinsic, and so are independent of the ambient space.

Apart from their appearance in the tube formula (\ref{tube:formula}), there
are a number of other ways in which to define the LKCs. One such
(nonintrinsic) way which signifies the dependence of the LKCs on the
Riemannian metric is through the shape operator. Let $M$ be an
$m$-dimensional $C^{2}$ manifold embedded in $\mathbb{R}^{n}$; then
%
\begin{eqnarray}\label{LK:equation}\quad
&&\mathcal{L}_{k}(M)
\nonumber
\\[-8pt]
\\[-8pt]
\nonumber
&&\qquad=K_{n,m,k}\int_{M}\int
_{S(N_{x}M)}\operatorname{Tr}\bigl(S_{\nu
}^{(m-k)}
\bigr)1_{N_{x}M}(-\nu) \mathcal{H}_{n-m-1}(d\nu)
\mathcal{H}_{m-1}(dx),
\end{eqnarray}
where, $K_{n,m,k}= \frac{1}{(m-k)!}\frac{\Gamma ( {(n-k)}/{2} )%
}{(2\pi)^{(n-k)/2}}$, and $S(N_{x}M)$ denotes a sphere in the normal
space $%
N_{x}M$ of $M$ at the point $x\in M$.

Closely related to the LKCs are set functionals called the Gaussian\break 
Minkowski functionals (GMFs), which are defined via a Gaussian tube formula.
Consider the Gaussian measure, $\gamma_{n}(dx)=(2\pi)^{-n/2}e^{-\Vert
x\Vert^{2}/2}\,dx$, instead of the standard Lebesgue measure in (\ref%
{tube:formula}); the Gaussian tube formula is then given by
%
\begin{equation}
\gamma_{n} \bigl( (M,\rho) \bigr) =\sum_{k\geq0}
\frac{\rho^{k}}{k!}%
\mathcal{M}_{k}^{\gamma_{n}}(M),
\label{Gaussian:tube:formula}
\end{equation}
where the coefficients $\mathcal{M}_{k}^{\gamma_{n}}(M)$'s are the GMFs
(for technical details, we refer the reader to \cite{RFG}). We note that
these set functionals, like their counterparts in (\ref{tube:formula}) can
be expressed as integrals over the manifold and its normal space (cf.
\cite%
{RFG}).

\subsection{Excursion probabilities and the Gaussian kinematic fundamental
formula}

A classical problem in stochastic processes is to compute the excursion
probability or the suprema probability
\[
P \Bigl( \sup_{x\in M}f(x)\geq u \Bigr) ,
\]
where, as before, $f$ is a random field defined on the parameter space $M$.
In the case when $f$ happens to be a centered Gaussian field with constant
variance $\sigma^{2}$ defined on $M$, a piecewise smooth manifold,
then by
the arguments set forth in Chapter~14 of \cite{RFG}, we have that
%
\begin{equation}
\Bigl|P \Bigl\{ \sup_{x\in M}f(x)\leq u \Bigr\} -\mathbb{E} \bigl\{
\mathcal{L}%
_{0}\bigl(A_{u}(f;M)\bigr) \bigr\}
\Bigr|<O \biggl( \exp \biggl(-\frac{\alpha u^{2}}{%
2\sigma^{2}} \biggr) \biggr) , \label{eqn:sup:EC}
\end{equation}
where $\mathcal{L}_{0}(A_{u}(f;M))$ is, as defined earlier, the
Euler--Poincar%
\'{e} characteristic of the excursion set $A_{u}(f;M)=\{x\in M\dvtx f(x)\geq
u\}$%
, and $\alpha>1$ is a constant, which depends on the field $f$ and can be
determined (see Theorem~14.3.3 of \cite{RFG}).

At first sight, from (\ref{eqn:sup:EC}) it may appear that we may have to
deal with a hard task, for example, that of evaluating $\mathbb{E}
\{ \mathcal{L}%
_{0}(A_{u}(f;M)) \} $. This task, however, is greatly simplified
due to
the \textit{Gaussian kinematic fundamental formula} (Gaussian-KFF) (see
Theorems 15.9.4--15.9.5 in \cite{RFG}), which states that, for a smooth
$%
M\subset\mathbb{R}^{N}$
\begin{eqnarray*}
&&\mathbb{E}\bigl(\mathcal{L}_{i}^{f}\bigl(A_{u}(f,M)
\bigr)\bigr)
\\
&&\qquad=\sum_{\mathbb{\ell}=0}^{\dim(M)-i}\pmatrix{ i+\mathbb{\ell}
\cr
\mathbb{\ell}}\frac{\Gamma ( {i}/{2}+1 ) \Gamma ( {\mathbb{%
\ell}}/{2}+1 ) }{\Gamma ( {(i+\mathbb{\ell})}/{2}+1 )
}%
(2
\pi)^{-\mathbb{\ell}/2}\mathcal{L}_{i+\mathbb{\ell}}^{f}(M)\mathcal
{M}_{%
\mathbb{\ell}}^{\gamma}\bigl([u,\infty)\bigr),
\end{eqnarray*}
for example, in the special case of the Euler characteristic ($i=0$)
%
\begin{equation}
\mathbb{E} \bigl\{ \mathcal{L}_{0}^{f}
\bigl(A_{u}(f;M)\bigr) \bigr\} =\sum_{j=0}^{%
\operatorname{dim}(M)}(2
\pi)^{-j/2}\mathcal{L}_{j}^{f}(M)\mathcal
{M}_{j}^{\gamma
} \bigl( [u,\infty) \bigr) , \label{eqn:GKF}
\end{equation}
where $\mathcal{L}_{j}^{f}(M)$ is the $j$th LKC of $M$ with respect to the
induced metric $g^{f}$ given by
\[
g_{x}^{f}(Y_{x},Z_{x})=\mathbb{E}
\bigl\{ Yf(x)\cdot Zf(x) \bigr\} ,
\]
for $X_{x},Y_{x}\in T_{x}M$, the tangent space at $x\in M$. The Gaussian
kinematic fundamental formula will play a crucial role in all the
developments to follow in the subsequent sections.

\section{Spherical Gaussian fields}
\label{sec:spherical:field}

In this section, we shall start from some simple results on the
evaluation of
the expected values of Lipschitz--Killing curvatures for sequences of
spherical Gaussian processes. These results will be rather straightforward
applications of the Gaussian kinematic fundamental formula (\ref{eqn:GKF}),
and are collected here for completeness and as a bridge toward the more
complicated case of {nonlocal transforms of} Gaussian subordinated
processes, to be considered later.

Note first that for {a unit variance} Gaussian field on the sphere $%
f\dvtx S^{2}\rightarrow\mathbb{R}$, the expected value of the
Euler--Poincar%
\'{e} characteristic of the excursion set $A_{u}(f;S^{2})=\{x\in
S^{2}\dvtx f(x)\geq u\}$ is given by
\begin{eqnarray*}
&&\mathbb{E} \bigl\{ \mathcal{L}_{0}\bigl(A_{u}
\bigl(f,S^{2}\bigr)\bigr) \bigr\}
\\
&&\qquad=\mathcal{L}_{0}^{f}\bigl(S^{2}\bigr)
\mathcal{M}_{0}^{\gamma}\bigl([u,\infty)\bigr)+(2\pi
)^{-1/2}\mathcal{L}_{1}^{f}\bigl(S^{2}
\bigr)\mathcal{M}_{1}^{\gamma}\bigl([u,\infty )\bigr)\\
&&\qquad\quad{}+(2
\pi)^{-1}\mathcal{L}_{2}^{f}\bigl(S^{2}
\bigr)\mathcal{M}_{2}^{\gamma
}\bigl([u,\infty )\bigr),
\end{eqnarray*}
for
\[
\mathcal{M}_{0}^{\gamma}\bigl([u,\infty)\bigr)=\int
_{u}^{\infty}\phi(x)\,dx,\qquad \mathcal{M}_{j}^{\gamma}
\bigl([u,\infty)\bigr)=H_{j-1}(u)\phi(u),
\]
where $\phi(\cdot)$ denotes the density of a real valued standard normal
random variable, and $H_{j}(u)$ denotes the Hermite polynomials,
\[
H_{j}(u)=(-1)^{j} \bigl( \phi(u) \bigr) ^{-1}
\frac{d^{j}}{du^{j}}\phi (u)\quad %
\mbox{and}\quad H_{-1}(u)=1-\Phi(u),
\]
{while} $\mathcal{L}_{k}^{f}(S^{2})$ are the usual Lipschitz--Killing
curvatures, under the induced Gaussian metric, that is,
\[
\mathcal{L}_{k}^{f}\bigl(S^{2}\bigr):={
\frac{(-2\pi)^{-(2-k)/2}}{2}}%
\int_{S^{2}}\operatorname{Tr}
\bigl(R^{(N-k)/2}\bigr)\operatorname{Vol}_{g^{f}};
\]
here, $R$ is the Riemannian curvature tensor and $\operatorname{Vol}_{g^{f}}$ is the volume
form, under the induced Gaussian metric, given by
\[
g^{f}(X,Y):=\mathbb{E} \{ Xf\cdot Yf \} =XY\mathbb {E}
\bigl(f^{2}\bigr).
\]

We recall that $\mathcal{L}_{0}(M)$ is a topological invariant and does not
depend on the metric; in particular, $\mathcal{L}_{0}(S^{2})\equiv2$.
Moreover, because the sphere is an (even-) $2$-dimensional manifold, $%
\mathcal{L}_{1}^{f}(S^{2})$ is identically zero.

As mentioned before, we start from some very simple result on the Fourier
components and wavelets transforms of Gaussian fields, for example, the expected
value of the Euler--Poincar\'{e} characteristic for two forms of harmonic
components, namely
\[
T_{\ell}(x)=\sum_{m=-\ell}^{\ell}a_{\ell m}Y_{\ell m}(x)\quad
\mbox{and}\quad \beta _{j}(x)=\sum_{\mathbb{\ell}}b\biggl(
\frac{\ell}{B^{j}}\biggr)T_{\ell}(x),
\]
the first representing a Fourier component at the multipole $\ell$, the
second a field of continuous needlet/wavelet coefficients at scale $j$. We
normalize these processes to unit variance by taking
\begin{eqnarray*}
\widetilde{T}_{\ell}(x)&=&\frac{T_{\ell}(x)}{\sqrt{({(2\ell+1)}/{(4\pi)})%
C_{\ell}}}\quad\mbox{and}\\
\widetilde{
\beta}_{j}(x)&=&\frac{\beta_{j}(x)}{
\sqrt{\sum_{\mathbb{\ell}}b^{2}({\ell}/{B^{j}})({(2\ell
+1)}/{(4\pi)})%
C_{\ell}}}.
\end{eqnarray*}
We start reporting some simple results on Lipschitz--Killing curvatures {of
excursion sets generated by} spherical Gaussian fields (see \cite{matsubara}
and the references therein for related expressions on $\mathbb{R}^{2}$ from
an astrophysical point of view). These results are straightforward
consequences of equation (\ref{eqn:GKF}).

\begin{lemma}
We have
\[
\mathcal{L}_{2}^{\widetilde{\beta}_{j}}\bigl(S^{2}\bigr) =4\pi
\frac{\sum_{\ell
}b^{2}({\ell}/{B^{j}})(2\ell+1)C_{\ell}({\ell(\ell+1)}/{2})}{
\sum_{\ell}b^{2}({\ell}/{B^{j}})(2\ell+1)C_{\ell}}.
\]
\end{lemma}

\begin{pf}
Recall first that, in standard spherical coordinates,
\[
P_{\mathbb{\ell}}\bigl( \langle x,y \rangle\bigr)=P_{\mathbb{\ell
}}\bigl(\sin
\vartheta_{x}\sin\vartheta_{y}\cos(\phi_{x}-
\phi_{y})+\cos \vartheta _{x}\cos\vartheta_{y}
\bigr).
\]
Some simple algebra then yields
\[
\frac{\partial^{2}}{\partial\vartheta_{x}\,\partial\vartheta_{y}}%
P_{\mathbb{\ell}}\bigl( \langle x,y \rangle\bigr)
\bigg\vert_{x=y}= \frac{\partial^{2}}{\sin\vartheta_{x}\sin\vartheta_{y}\partial\phi
_{x}\,\partial\phi_{y}}P_{\mathbb{\ell}}\bigl( \langle x,y \rangle
\bigr)\bigg\vert_{x=y}=P_{\mathbb{\ell}}^{\prime}(1)
\]
and
\[
\frac{\partial^{2}}{\sin\vartheta_{x}\partial\vartheta
_{y}\,\partial\phi_{x}}P_{\mathbb{\ell}}\bigl( \langle x,y \rangle \bigr)
\bigg\vert_{x=y}=0.
\]
The geometric meaning of the latter result is that the process is still
isotropic {under the new transformation,} whence the derivatives along the
two directions are still independent. As a consequence, writing $\mathbb
{E}%
 \{ \widetilde{\beta}_{j}(x)\widetilde{\beta}_{j}(y) \}
=:\Gamma
_{j}(x,y)$ we have
\begin{eqnarray*}
\frac{\partial^{2}\Gamma_{j}(x,y)}{\partial\vartheta_{x}\,\partial
\vartheta_{y}}\bigg\vert_{x=y}&= &\frac{\partial^{2}\Gamma_{j}(x,y)%
}{\sin\vartheta_{x}\sin\vartheta_{y}\partial\phi_{x}\,\partial\phi
_{y}}%
\bigg\vert_{x=y}\\
&=&\frac{\sum_{\mathbb{\ell}}b^{2}({\ell}/{B^{j}})C_{%
\mathbb{\ell}}({(2\mathbb{\ell}+1)}/{(4\pi)})P_{\mathbb{\ell
}}^{\prime}(1)%
}{\sum_{\mathbb{\ell}}b^{2}({\ell}/{B^{j}})({(2\mathbb{\ell
}+1)}/{%
(4\pi)})C_{\mathbb{\ell}}}
\end{eqnarray*}
and
\[
\frac{\partial^{2}\Gamma_{j}(x,y)}{\sin\vartheta_{x}\partial
\vartheta_{y}\,\partial\phi_{x}}\bigg\vert_{x=y}=0.
\]
We thus have that
{\fontsize{10.7}{12.7}{\selectfont
\begin{eqnarray*}
&&\hspace*{-4pt}\mathcal{L}_{2}^{\widetilde{\beta}_{j}}\bigl(S^{2}\bigr)\\
&&\hspace*{-6pt}\qquad=\int
_{S^{2}}\left\{ \det%
\left[ \matrix{\displaystyle \frac{\partial^{2}\Gamma_{j}(x,y)}{\partial\vartheta_{x}\,\partial
\vartheta_{y}}\bigg\vert_{x=y} & \displaystyle\frac{\partial^{2}\Gamma
_{j}(x,y)}{\sin\vartheta_{x}\partial\phi_{x}\,\partial\vartheta_{y}}
\bigg\vert_{x=y}
\vspace*{2pt}\cr
\displaystyle\frac{\partial^{2}\Gamma_{j}(x,y)}{\sin\vartheta_{y}\partial\phi
_{y}\,\partial\vartheta_{x}}\bigg\vert_{x=y} & \displaystyle\frac{\partial
^{2}\Gamma_{j}(x,y)}{\sin\vartheta_{x}\sin\vartheta_{y}\partial
\phi
_{x}\,\partial\phi_{y}}
\bigg\vert_{x=y}}
 \right] \right\} ^{1/2}\sin\vartheta \,d\vartheta \,d\phi
\\
&&\hspace*{-6pt}\qquad=4\pi\frac{\sum_{\mathbb{\ell}}b^{2}({\ell}/{B^{j}})
({(2\mathbb{%
\ell}+1)}/{(4\pi)})C_{\mathbb{\ell}}P_{\mathbb{\ell}}^{\prime}(1)}{\sum_{%
\mathbb{\ell}}b^{2}({\ell}/{B^{j}})({(2\mathbb{\ell
}+1)}/{(4\pi)})C_{%
\mathbb{\ell}}}.
\end{eqnarray*}}}
\hspace*{-3pt}Now recall that $P_{\mathbb{\ell}}^{\prime}(1)=\frac{\mathbb{\ell}(%
\mathbb{\ell}+1)}{2}$, whence the claim is established.
\end{pf}

\begin{remark}
Note that since the random field $\beta_{j}$ is an isotropic Gaussian
random field, the Lipschitz--Killing curvatures of $S^{2}$ under the metric
induced by the field $\beta_{j}$ are given by
\[
\mathcal{L}_{i}^{\widetilde{\beta}_{j}}\bigl(S^{2}\bigr)=\lambda
_{j}^{i/2}\mathcal{L}%
_{i}
\bigl(S^{2}\bigr),
\]
where $\mathcal{L}_{i}(S^{2})$ is the $i$th LKC under the usual Euclidean
metric, and $\lambda_{j}$ is the second spectral moment of $\widetilde{
\beta}_{j}$ (cf. \cite{RFG}). This result is true for all isotropic and
unit variance Gaussian random fields.
\end{remark}

The second auxiliary result that we shall need follows immediately from
Theorem~13.2.1 in \cite{RFG}, specialized to isotropic spherical random
fields with unit variance. Analogous expressions have been given (among many
other results) in the two recent papers \cite{chengschwar,chengxiao2}. The computations are straightforward and we report them only for
completeness.

\begin{lemma}
For the Gaussian isotropic field $\widetilde{\beta
}_{j}\dvtx S^{2}\rightarrow
\mathbb{R}$, such that $\mathbb{E}\widetilde{\beta}_{j}=0$, $\mathbb
{E}%
\widetilde{\beta}_{j}^{2}=1$, $\widetilde{\beta}_{j}\in C^{2}(S^{2})$
almost surely, we have that
%
\begin{eqnarray}\label{feb1}\quad
&&\mathbb{E} \bigl\{ \mathcal{L}_{0}\bigl(A_{u}\bigl(
\widetilde{\beta}%
_{j}(x),S^{2}\bigr)\bigr) \bigr\}
\nonumber
\\[-8pt]
\\[-8pt]
\nonumber
&&\qquad=2 \bigl\{ 1-\Phi(u) \bigr\} +4\pi \biggl\{ \frac
{%
\sum_{\mathbb{\ell}}b^{2}({\ell}/{B^{j}})C_{\mathbb{\ell}}({(2
\mathbb{\ell}+1)}/{(4\pi)})P_{\mathbb{\ell}}^{\prime}(1)}{\sum_{\mathbb
{\ell}%
}b^{2}({\ell}/{B^{j}})C_{\mathbb{\ell}}({(2\mathbb{\ell}+1)}/{(4\pi)})}%
 \biggr
\} \frac{ue^{-u^{2}/2}}{\sqrt{(2\pi)^{3}}},
\\
\label{feb2}&&\mathbb{E} \bigl\{ \mathcal{L}_{1}\bigl(A_{u}\bigl(
\widetilde{\beta}%
_{j}(x),S^{2}\bigr)\bigr) \bigr\}
\nonumber
\\[-8pt]
\\[-8pt]
\nonumber
&&\qquad=\pi \biggl\{ \frac{\sum_{\mathbb{\ell
}}b^{2}({%
\ell}/{B^{j}})C_{\mathbb{\ell}}({(2\mathbb{\ell}+1)}/{(4\pi)})P_{\mathbb{%
\ell}}^{\prime}(1)}{\sum_{\mathbb{\ell}}b^{2}({\ell}/{B^{j}})C_{
\mathbb{\ell}}({(2\mathbb{\ell}+1)}/{(4\pi)})} \biggr\} ^{1/2}e^{-u^{2}/2},
\end{eqnarray}
and finally
%
\begin{equation}
\mathbb{E} \bigl\{ \mathcal{L}_{2}\bigl(A_{u}\bigl(
\widetilde{\beta}%
_{j}(x),S^{2}\bigr)\bigr) \bigr\}
= \bigl\{ 1-\Phi(u) \bigr\} 4\pi. \label{feb3}
\end{equation}
\end{lemma}

\begin{pf}
We start by recalling that, from Theorem~13.2.1 in \cite{RFG},
\[
\mathbb{E} \bigl\{ \mathcal{L}_{i}\bigl(A_{u}\bigl(
\widetilde{\beta}%
_{j}(x),S^{2}\bigr)\bigr) \bigr\}
=\sum_{\mathbb{\ell}=0}^{\dim(S^{2})-i}\left[ \matrix{i+\mathbb{\ell}
\vspace*{2pt}\cr
\mathbb{\ell}}
 \right] \lambda^{\mathbb{\ell}/2}
\rho_{\mathbb{\ell}}(u)\mathcal {L}_{i+%
\mathbb{\ell}}\bigl(S^{2}\bigr),
\]
where
\begin{eqnarray*}
\left[ \matrix{ i+\mathbb{\ell}
\vspace*{2pt}\cr
\mathbb{\ell}}
 \right]&:=&\pmatrix{ i+\mathbb{\ell}
\vspace*{2pt}\cr
\mathbb{\ell}%
} \frac{\omega_{i+\mathbb{\ell}}}{\omega_{i}\omega_{\mathbb
{\ell}}%
},\qquad
\omega_{i}=\frac{\pi^{i/2}}{\Gamma({i}/{2}+1)},
\\
\rho_{\mathbb{\ell}}(u)&=&(2\pi)^{-\mathbb{\ell}/2}\mathcal {M}_{\mathbb{%
\ell}}^{\gamma}
\bigl([u,\infty)\bigr)=(2\pi)^{-(\mathbb{\ell}+1)/2}H_{\mathbb
{%
\ell}-1}(u)e^{-u^{2}/2},
\end{eqnarray*}
so that
\begin{eqnarray*}
\rho_{0}(u)&=&(2\pi)^{-1/2}\sqrt{2\pi}\bigl(1-\Phi (u)
\bigr)e^{u^{2}/2}e^{-u^{2}/2}=\bigl(1-\Phi(u)\bigr),
\\
\rho_{1}(u)&=&\frac{1}{2\pi}e^{-u^{2}/2},\qquad\rho
_{2}(u)=\frac{%
1}{\sqrt{(2\pi)^{3}}}ue^{-u^{2}/2}.
\end{eqnarray*}
Here,
\begin{eqnarray*}
\lambda&=&\mathbb{E}\beta_{j;\vartheta}^{2}=\mathbb{E}
\beta_{j;\phi
}^{2}%
,\qquad\beta_{j;\vartheta}=
\frac{\partial}{\partial\vartheta
}\beta _{j}(\vartheta,\phi),\\
\beta_{j;\phi}&=&
\frac{\partial}{\sin
\vartheta\partial\phi}\beta_{j}(\vartheta,\phi),
\\
\mathbb{E} \bigl\{ \widetilde{\beta}_{j;\vartheta}^{2} \bigr\} &=&
\frac{%
\partial^{2}}{\partial\vartheta^{2}}\mathbb{E} \bigl\{ \widetilde {\beta}%
_{j}^{2}
\bigr\} =\frac{\sum_{\mathbb{\ell}}b^{2}({\ell
}/{B^{j}})C_{%
\mathbb{\ell}}({(2\mathbb{\ell}+1)}/{(4\pi)})P_{\mathbb{\ell
}}^{\prime}(1)%
}{\sum_{\mathbb{\ell}}b^{2}({\ell}/{B^{j}})C_{\mathbb{\ell}}
({(2\mathbb{\ell}+1)}/{(4\pi)})},
\end{eqnarray*}
whence
\begin{eqnarray*}
&&\mathbb{E} \bigl\{ \mathcal{L}_{0}\bigl(A_{u}\bigl(
\widetilde{\beta}%
_{j}(x),S^{2}\bigr)\bigr) \bigr\}
\\
&&\qquad=2 \bigl\{ 1-\Phi(u) \bigr\} +4\pi \biggl\{ \frac{%
\sum_{\mathbb{\ell}}b^{2}({\ell}/{B^{j}})C_{\mathbb{\ell}}
({(2\mathbb{\ell}+1)}/{(4\pi)})P_{\mathbb{\ell}}^{\prime}(1)}{\sum_{\mathbb
{\ell}%
}b^{2}({\ell}/{B^{j}})C_{\mathbb{\ell}}({(2\mathbb{\ell}+1)}/{(4\pi)})}%
 \biggr\} \frac{ue^{-u^{2}/2}}{\sqrt{(2\pi)^{3}}}.
\end{eqnarray*}
Also,
\[
\mathbb{E} \bigl\{ \mathcal{L}_{1}\bigl(A_{u}\bigl(
\widetilde{\beta}%
_{j}(x),S^{2}\bigr)\bigr) \bigr\}
=\pi \biggl\{ \frac{\sum_{\mathbb{\ell
}}b^{2}({%
\ell}/{B^{j}})C_{\mathbb{\ell}}({(2\mathbb{\ell}+1)}/{(4\pi)})P_{\mathbb{%
\ell}}^{\prime}(1)}{\sum_{\mathbb{\ell}}b^{2}({\ell}/{B^{j}})C_{
\mathbb{\ell}}({(2\mathbb{\ell}+1)}/{(4\pi)})} \biggr\} ^{1/2}e^{-u^{2}/2}.
\]
Finally,
\[
\mathbb{E} \bigl\{ \mathcal{L}_{2}\bigl(A_{u}\bigl(
\widetilde{\beta}%
_{j}(x),S^{2}\bigr)\bigr) \bigr\}
=\rho_{0}(u)\mathcal{L}_{2}\bigl(S^{2}\bigr)=
\bigl\{ 1-\Phi (u) \bigr\} 4\pi,
\]
which completes the proof.
\end{pf}


In the case of spherical eigenfunctions, the previous lemma takes the
following simpler form; the proof is entirely analogous, and hence omitted.

\begin{corollary}
For the field $ \{ T_{\mathbb{\ell}}(\cdot) \} $, we have that
%
\begin{eqnarray}\label{sreekar}
\mathbb{E} \bigl\{ \mathcal{L}_{0}\bigl(A_{u}\bigl(
\widetilde{T}_{\mathbb{\ell}%
}(\cdot),S^{2}\bigr)\bigr) \bigr\} &=&2 \bigl\{ 1-
\Phi(u) \bigr\} +\frac{\mathbb{\ell}(
\mathbb{\ell}+1)}{2}\frac{ue^{-u^{2}/2}}{\sqrt{(2\pi)^{3}}}4\pi,
\nonumber
\\[-8pt]
\\[-8pt]
\nonumber
\mathbb{E} \bigl\{ \mathcal{L}_{1}\bigl(A_{u}\bigl(
\widetilde{T}_{\mathbb{\ell}%
}(\cdot),S^{2}\bigr)\bigr) \bigr\} &=&\pi \biggl\{
\frac{\mathbb{\ell}(\mathbb{\ell
}+1)}{2}%
 \biggr\} ^{1/2}e^{-u^{2}/2}
\end{eqnarray}
and
\[
\mathbb{E} \bigl\{ \mathcal{L}_{2}\bigl(A_{u}\bigl(
\widetilde{T}_{\mathbb{\ell}%
}(\cdot),S^{2}\bigr)\bigr) \bigr\} =4\pi\times
\bigl\{ 1-\Phi(u) \bigr\}.
\]
\end{corollary}

\begin{remark}
{Using the differential geometric definition of the Lipschitz--Killing
curvatures, it is easy to observe that}
\[
2\mathbb{E} \bigl\{ \mathcal{L}_{1}(A_{u}\bigl(
\widetilde{T}_{\mathbb{\ell}%
}(\cdot),S^{2}\bigr)) \bigr\} =\mathbb{E} \bigl\{
\operatorname{len}(\partial A_{u}\bigl(\widetilde {T}_{%
\mathbb{\ell}}(\cdot),S^{2}
\bigr)) \bigr\} ,
\]
where $\operatorname{len}(\partial A_{u}(\widetilde{T}_{\mathbb{\ell}}(\cdot),S^{2}))$
is the
usual length of the boundary region of the excursion set, in the usual
Hausdorff sense, {which can also be expressed as $\mathcal
{L}_{1}(\partial
A_{u}(T_{\mathbb{\ell}}(\cdot),S^{2}))$}. Hence,
\[
\mathbb{E} \bigl\{ \operatorname{len}\bigl(\partial A_{u}
\bigl(\widetilde{T}_{\mathbb{\ell}%
}(\cdot),S^{2}
\bigr)\bigr) \bigr\} =2\pi \biggl\{ \frac{\mathbb{\ell}(\mathbb{\ell
}+1)}{2}%
 \biggr\}
^{1/2}e^{-u^{2}/2},
\]
which for $u=0$ fits with well-known results on the expected value of nodal
lines for random spherical eigenfunctions (see \cite{Wig2} and the
references therein). Likewise
%
\begin{eqnarray}\label{latter}
&&\mathbb{E} \bigl\{ \operatorname{len}\bigl(\partial A_{u}\bigl(\widetilde{\beta
}_{j}(\cdot),S^{2}\bigr)\bigr)
\bigr\}
\nonumber
\\[-8pt]
\\[-8pt]
\nonumber
&&\qquad =2\pi \biggl\{
\frac{\sum_{\mathbb{\ell}}b^{2}({\ell
}/{B^{j}})C_{%
\mathbb{\ell}}({(2\mathbb{\ell}+1)}/{(4\pi)})P_{\mathbb{\ell
}}^{\prime}(1)%
}{\sum_{\mathbb{\ell}}b^{2}({\ell}/{B^{j}})C_{\mathbb{\ell}}
({(2\mathbb{\ell}+1)}/{(4\pi)})} \biggr\} ^{1/2}e^{-u^{2}/2}.
\end{eqnarray}
\end{remark}

These formulae can be made more explicit by setting a specific form for the
behaviour of the angular power spectrum $ \{ C_{\mathbb{\ell
}} \} $
and the weighting kernel $b(\cdot)$, see~\cite{fantaye2014} for numerical
results under conditions of astrophysical interest{.}

\section{Gaussian subordinated fields}
\label{sec:gaussian:subordinate}

\subsection{Local transforms of \texorpdfstring{$\beta_{j}(\cdot)$}
{betaj(.)}}

For statistical applications, it is often more interesting to consider
nonlinear transforms of random fields. For instance, in a CMB related
environment a lot of efforts have been spent to investigate local
fluctuations of angular power spectra; to this aim, moving averages of
squared wavelet/needlet coefficients are usually computed; see, for instance,
\cite{pietrobon1} and the references therein. Our purpose here is to derive
some rigorous results on the behaviour of these statistics.

To this aim, let us consider first the simple squared field
\begin{eqnarray*}
H_{2j}(x) &:=&H_{2}\bigl(\widetilde{\beta}_{j}(x)
\bigr)=\frac{\beta_{j}^{2}(x)}{
\sigma_{\beta_{j}}^{2}}-1,
\\
\sigma_{\beta_{j}}^{2} &:=&\sum_{\ell}b^{2}
\biggl(\frac{\ell}{B^{j}}\biggr)C_{%
\mathbb{\ell}}\frac{2\mathbb{\ell}+1}{4\pi}=\mathbb{E}
\beta_{j}^{2}(x).
\end{eqnarray*}
The expected value of Lipschitz--Killing curvatures for the excursion regions
of such fields is easily derived, indeed by the general Gaussian kinematic
formula we have, for $u\geq-1$
\begin{eqnarray*}
&&\mathbb{E} \bigl\{ \mathcal{L}_{0}^{\widetilde{\beta}%
_{j}}
\bigl(A_{u}\bigl(H_{2};S^{2}\bigr)\bigr) \bigr\}
\\
&&\qquad=\sum_{k=0}^{2}(2\pi )^{-k/2}
\mathcal{L}%
_{k}^{\widetilde{\beta}_{j}}\bigl(S^{2}\bigr)
\mathcal{M}_{k}^{\mathcal
{N}}\bigl((-\infty,-%
\sqrt{u+1})
\cup(\sqrt{u+1},\infty)\bigr)
\\
&&\qquad=\sum_{k=0}^{2}(2\pi)^{-k/2}
\mathcal{L}_{k}^{\widetilde{\beta}%
_{j}}\bigl(S^{2}\bigr)2
\mathcal{M}_{k}^{\mathcal{N}}\bigl((\sqrt{u+1},\infty)\bigr)
\\
&&\qquad=4\bigl(1-\Phi(\sqrt{u+1})\bigr)\\
&&\qquad\quad{}+\frac{1}{2\pi}\frac{\sum_{\mathbb{\ell}}b^{2}(
{\mathbb{\ell}}/{B^{j}})({(2\mathbb{\ell}+1)}/{(4\pi)})C_{\mathbb
{\ell}%
}P_{\mathbb{\ell}}^{\prime}(1)}{\sum_{\mathbb{\ell}}b^{2}(
{\mathbb{%
\ell}}/{B^{j}})({(2\mathbb{\ell}+1)}/{(4\pi)})C_{\mathbb{\ell
}}}
\mathcal{L}%
_{2}\bigl(S^{2}\bigr)\frac{e^{-(u+1)/2}}{\sqrt{2\pi}}2
\sqrt{u+1}.
\end{eqnarray*}
Likewise
\begin{eqnarray*}
&&\mathbb{E} \bigl\{ \mathcal{L}_{1}^{\widetilde{\beta}%
_{j}}
\bigl(A_{u}\bigl(H_{2};S^{2}\bigr)\bigr) \bigr\}
\\
&&\qquad=\sum_{k=0}^{1}(2\pi)^{-k/2}
\left[\matrix{ k+1
\vspace*{2pt}\cr
k}
 \right] \mathcal{L}_{k+1}^{\widetilde{\beta}_{j}}
\bigl(S^{2}\bigr)\mathcal {M}_{k}^{%
\mathcal{N}}\bigl((-
\infty,-\sqrt{u+1})\cup(\sqrt{u+1},\infty)\bigr)
\\
&&\qquad=\mathcal{L}_{1}^{\widetilde{\beta}_{j}}\bigl(S^{2}\bigr)\mathcal
{M}_{0}^{\mathcal{N%
}}\bigl((-\infty,-\sqrt{u+1})\cup(\sqrt{u+1},
\infty)\bigr)
\\
&&\qquad\quad{}+(2\pi)^{-1/2}\frac{\pi}{2}\mathcal{L}_{2}^{\widetilde{\beta}%
_{j}}
\bigl(S^{2}\bigr)\mathcal{M}_{1}^{\mathcal{N}}\bigl((-
\infty,-\sqrt{u+1})\cup (\sqrt{%
u+1},\infty)\bigr)
\\
&&\qquad=(2\pi)^{-1/2}\frac{\pi}{2}\biggl(4\pi\times\frac{\sum_{\mathbb{\ell
}}b^{2}(%
{\mathbb{\ell}}/{B^{j}})({(2\mathbb{\ell}+1)}/{(4\pi)})C_{\mathbb
{\ell}%
}P_{\mathbb{\ell}}^{\prime}(1)}{\sum_{\mathbb{\ell}}b^{2}(
{\mathbb{%
\ell}}/{B^{j}})({(2\mathbb{\ell}+1)}/{(4\pi)})C_{\mathbb{\ell}}}
\biggr)2\frac
{%
e^{-(u+1)/2}}{\sqrt{2\pi}}
\\
&&\qquad=2\pi\biggl(\frac{\sum_{\mathbb{\ell}}b^{2}({\mathbb{\ell
}}/{B^{j}})({(2\mathbb{\ell}+1)}/{(4\pi)})C_{\mathbb{\ell}}P_{\mathbb{\ell}}^{\prime
}(1)}{%
\sum_{\mathbb{\ell}}b^{2}({\mathbb{\ell}}/{B^{j}})({(2\mathbb {\ell}+1)}/{(4\pi)})C_{\mathbb{\ell}}}\biggr)e^{-(u+1)/2},
\end{eqnarray*}
which implies for the Euclidean LKC
\[
\mathbb{E} \bigl\{ \mathcal{L}_{1}\bigl(A_{u}
\bigl(H_{2};S^{2}\bigr)\bigr) \bigr\} =2\pi \biggl\{
\frac{\sum_{\mathbb{\ell}}b^{2}({\mathbb{\ell}}/{B^{j}})
({(2\mathbb{\ell}+1)}/{(4\pi)})C_{\mathbb{\ell}}P_{\mathbb{\ell}}^{\prime}(1)}{\sum_{%
\mathbb{\ell}}b^{2}({\mathbb{\ell}}/{B^{j}})({(2\mathbb{\ell}+1)}/{(4\pi)})C_{\mathbb{\ell}}} \biggr\} ^{1/2}e^{-(u+1)/2}
\]
and, therefore,
\[
\mathbb{E} \bigl\{ \mathcal{L}_{1}\bigl(\partial A_{u}
\bigl(H_{2};S^{2}\bigr)\bigr) \bigr\} =4\pi \biggl\{
\frac{\sum_{\mathbb{\ell}}b^{2}({\mathbb{\ell
}}/{B^{j}})({(2%
\mathbb{\ell}+1)}/{(4\pi)})C_{\mathbb{\ell}}P_{\mathbb{\ell}}^{\prime
}(1)}{%
\sum_{\mathbb{\ell}}b^{2}({\mathbb{\ell}}/{B^{j}})({(2\mathbb
{\ell}%
+1)}/{(4\pi)})C_{\mathbb{\ell}}} \biggr\} ^{1/2}e^{-(u+1)/2}.
\]
Finally,
\begin{eqnarray*}
&&\mathbb{E} \bigl\{ \mathcal{L}_{2}^{\widetilde{\beta}%
_{j}}\bigl(A_{u}
\bigl(H_{2};S^{2}\bigr)\bigr) \bigr\} \\
&&\qquad=\mathcal{L}_{2}^{\widetilde{\beta}
_{j}}
\bigl(S^{2}\bigr)\mathcal{M}_{0}^{\mathcal{N}}\bigl((-
\infty,-\sqrt{u+1})\cup (\sqrt{%
u+1},\infty)\bigr)
\\
&&\qquad=4\pi \biggl\{ \frac{\sum_{\mathbb{\ell}}b^{2}({\mathbb{\ell
}}/{B^{j}}%
)({(2\mathbb{\ell}+1)}/{(4\pi)})C_{\mathbb{\ell}}P_{\mathbb{\ell
}}^{\prime
}(1)}{\sum_{\mathbb{\ell}}b^{2}({\mathbb{\ell}}/{B^{j}})
({(2\mathbb{%
\ell}+1)}/{(4\pi)})C_{\mathbb{\ell}}} \biggr\} 2\bigl(1-\Phi(\sqrt{u+1})\bigr)
\end{eqnarray*}
entailing a Euclidean LKC
\[
\mathbb{E} \bigl\{ \mathcal{L}_{2}\bigl(A_{u}
\bigl(H_{2};S^{2}\bigr)\bigr) \bigr\} =4\pi \times 2\bigl(1-
\Phi(\sqrt{u+1})\bigr).
\]
It should be noted that the tail decay for the Euler characteristic and the
boundary length is much slower than in the Gaussian case. This is consistent
with the elementary fact that polynomial transforms shift angular power
spectra at higher frequencies, hence yielding a rougher path behaviour.
Likewise, for cubic transforms we have
\begin{eqnarray*}
&&\mathbb{E} \bigl\{ \mathcal{L}_{0}^{\widetilde{\beta
}_{j}}
\bigl(A_{u}\bigl(\widetilde{%
\beta}_{j}^{3}(x);S^{2}
\bigr)\bigr) \bigr\}\\
&&\qquad =2\bigl(1-\Phi\bigl(\sqrt[3]{u}\bigr)\bigr)\\
&&\qquad\quad{}+
\frac
{1}{2\pi}%
\frac{\sum_{\ell}b^{2}({\ell}/{B^{j}})({(2\ell+1)}/{(4\pi)
})C_{\ell
}P_{\ell}^{\prime}(1)}{\sum_{\ell}b^{2}({\ell}/{B^{j}})
({(2\ell+1)%
}/{(4\pi)})C_{\ell}}\mathcal{L}_{2}
\bigl(S^{2}\bigr)\frac{e^{-(\sqrt[3]{u})^{2}/2}}{
\sqrt{2\pi}}\sqrt[3]{u},
\\
&&\mathbb{E} \bigl\{ \mathcal{L}_{1}^{\widetilde{\beta
}_{j}}
\bigl(A_{u}\bigl(\widetilde{%
\beta}_{j}^{3}(x);S^{2}
\bigr)\bigr) \bigr\}
\\
&&\qquad =\pi\biggl(\frac{\sum_{\ell}b^{2}(
{\ell
}/{B^{j}})({(2\ell+1)}/{(4\pi)})C_{\ell}P_{\ell}^{\prime}(1)}{\sum_{\ell
}b^{2}({\ell}/{B^{j}})({(2\ell+1)}/{(4\pi)})C_{\ell}}\biggr)e^{-(\sqrt
[3]{u}%
)^{2}/2},
\\
&&\mathbb{E} \bigl\{ \mathcal{L}_{1}\bigl(\partial A_{u}
\bigl(\widetilde{\beta}%
_{j}^{3}(x);S^{2}
\bigr)\bigr) \bigr\} \\
&&\qquad=2\pi \biggl\{ \frac{\sum_{\ell}b^{2}(
{\ell
}/{B^{j}})({(2\ell+1)}/{(4\pi)})C_{\ell}P_{\ell}^{\prime}(1)}{\sum_{%
\mathbb{\ell}}b^{2}({\ell}/{B^{j}})({(2\ell+1)}/{(4\pi)})C_{\ell
}}%
 \biggr\}
^{1/2}e^{-(\sqrt[3]{u})^{2}/2},
\end{eqnarray*}
and finally
\begin{eqnarray*}
&&\mathbb{E} \bigl\{ \mathcal{L}_{2}^{\widetilde{\beta
}_{j}}\bigl(A_{u}
\bigl(\widetilde{%
\beta}_{j}^{3}(x);S^{2}
\bigr)\bigr) \bigr\} \\
&&\qquad=4\pi \biggl\{ \frac{\sum_{\ell
}b^{2}(%
{\ell}/{B^{j}})({(2\ell+1)}/{(4\pi)})C_{\ell}P_{\ell}^{\prime
}(1)}{%
\sum_{\ell}b^{2}({\ell}/{B^{j}})({(2\ell+1)}/{(4\pi)})C_{\ell}}%
 \biggr\} 2
\bigl(1-\Phi\bigl(\sqrt[3]{u}\bigr)\bigr)
\end{eqnarray*}
entailing an expected value for the excursion area given by
\[
\mathbb{E} \bigl\{ \mathcal{L}_{2}\bigl(A_{u}\bigl(
\widetilde{\beta}%
_{j}^{3}(x);S^{2}
\bigr)\bigr) \bigr\} =4\pi\bigl(1-\Phi\bigl(\sqrt[3]{u}\bigr)\bigr).
\]

Similar results could be easily derived for higher order polynomial
transforms; numerical evidence and astrophysical applications can be found
in \cite{fantaye2014}. However, as motivated above we believe it is much
more important to focus on transforms that entail some form of local
averaging, as these are likely to be more relevant for practitioners. To
this issue, we devote the rest of this section and a large part of the paper.

\subsection{Nonlocal transforms of \texorpdfstring{$\beta_{j}(\cdot)$}
{betaj(.)}}

We now consider the case of smoothed nonlinear functionals. We are
interested, for instance, in studying the LKCs for local estimates of the
angular power spectrum, which as mentioned before have already found many
important applications in a CMB related framework. To this aim, we
introduce, for every $x\in S^{2}$,
%
\begin{equation}
g_{j;q}(x):=\int_{S^{2}}K\bigl( \langle x,y \rangle
\bigr)H_{q}\bigl(\widetilde{%
\beta}_{j}(y)
\bigr)\,dy; \label{filipo}
\end{equation}
throughout the sequel, we shall assume that the following finite-order
expansion holds:
%
\begin{equation}
K(u)=\sum_{\ell=1}^{L_{K}}\frac{2\ell+1}{4\pi}
\kappa(\ell)P_{\ell
}(u)\qquad
\mbox{some fixed }L_{K}\in
\mathbb{N},u\in{}[ -1,1]. \label{kernexp}
\end{equation}
Here, as before we write $H_{q}(\cdot)$ for the Hermite polynomials. For $q=1$,
we just get the smoothed Gaussian process
%
\begin{equation}
g_{j}(x):=g_{j;1}(x)=\int_{S^{2}}K\bigl(
\langle x,y \rangle \bigr)\widetilde{%
\beta}_{j}(y)\,dy.
\label{kernel}
\end{equation}
The practical importance of the analysis of fields such as $g_{j;q}(\cdot)$ can
be motivated as follows. A crucial topic when dealing with cosmological data
is the analysis of isotropy properties. For instance, in a CMB related
framework a large amount of work has focused on the possible existence of
asymmetries in the behaviour of angular power spectra or bispectra across
different hemispheres (see, e.g., \cite{pietrobon1,rudjord2}).
In these papers, powers of wavelet coefficients at some frequencies $j$ are
averaged over different hemispheres to investigate the existence of
asymmetries/anisotropies in the CMB distribution; some evidence has been
reported, for instance, for power asymmetries with respect to the Milky Way
plane for frequencies corresponding to angular scales of a few degrees (such
effects are related in the cosmological literature to widely debated
anomalies known as \emph{the Cold Spot }and \emph{the Axis of Evil}; see
\cite{bennett2012,planckIS} and the references therein). To investigate
these anomalies, statistics which can be viewed as discretized versions
of $%
\sup_{x\in S^{2}}g_{j;q}(x)$ have been evaluated; their significance is
typically tested against Monte Carlo simulations, under the null of
isotropy. Our results below will provide the first rigorous derivation of
asymptotic properties in this settings.

Our first lemma is an immediate application of spherical Fourier analysis
techniques.

\begin{lemma}
The field $g_{j}(x)$ is zero-mean, finite variance and isotropic, with
covariance function
\[
\mathbb{E} \bigl\{ g_{j}(x_{1})g_{j}(x_{2})
\bigr\} =\frac{1}{\sigma
_{\beta
_{j}}^{2}}\sum_{\mathbb{\ell}}b^{2}
\biggl(\frac{\ell}{B^{j}}\biggr)\kappa ^{2}(\ell)%
\frac{2\ell+1}{4\pi}C_{\ell}P_{\ell}\bigl( \langle
x_{1},x_{2} \rangle\bigr).
\]
\end{lemma}

\begin{pf}
Note first that
\begin{eqnarray*}
&&\mathbb{E} \bigl\{ g_{j}(x_{1})g_{j}(x_{2})
\bigr\}\\
&&\qquad=\frac{1}{\sigma
_{\beta_{j}}^{2}} \biggl\{ \int_{S^{2}\times S^{2}}K\bigl(
\langle x_{1},y_{1} \rangle\bigr)K\bigl( \langle
x_{2},y_{2} \rangle \bigr)\mathbb{E}%
 \bigl\{
\beta_{j}(y_{1})\beta_{j}(y_{2}) \bigr
\} \,dy_{1}\,dy_{2} \biggr\}
\\
&&\qquad=\frac{1}{\sigma_{\beta_{j}}^{2}}\int_{S^{2}\times S^{2}}K\bigl( \langle
x_{1},y_{1} \rangle\bigr)K\bigl( \langle
x_{2},y_{2} \rangle \bigr)\sum_{\ell}b^{2}
\biggl(\frac{\ell}{B^{j}}\biggr)\frac{2\ell+1}{4\pi}C_{\ell
}P_{\ell}
\bigl( \langle y_{1},y_{2} \rangle\bigr).
\end{eqnarray*}
Recall the reproducing kernel formula (see, e.g., \cite{marpecbook},
pages~248--249)
\begin{eqnarray*}
\int_{S^{2}}P_{\ell}\bigl( \langle
x_{1},y_{1} \rangle\bigr)P_{\ell
}\bigl( \langle
y_{1},y_{2} \rangle\bigr)\,dy_{1} &=&
\frac{4\pi}{2\ell
+1}%
P_{\ell}\bigl( \langle x_{1},y_{2}
\rangle\bigr),
\\
\int_{S^{2}}P_{\ell_{1}}\bigl( \langle
x_{1},y_{1} \rangle\bigr)P_{\ell
_{2}}\bigl( \langle
y_{1},y_{2} \rangle\bigr)\,dy_{1} &=&0,\qquad\ell
_{1}\neq\ell_{2},
\end{eqnarray*}
whence
\begin{eqnarray*}
&&\int_{S^{2}\times S^{2}}K\bigl( \langle x_{1},y_{1}
\rangle \bigr)K\bigl( \langle x_{2},y_{2} \rangle\bigr)\sum
_{\ell}b^{2}\biggl(\frac{\ell
}{%
B^{j}}\biggr)
\frac{2\ell+1}{4\pi}C_{\ell}P_{\ell}\bigl( \langle
y_{1},y_{2} \rangle\bigr)
\\
&&\qquad=\int_{S^{2}\times S^{2}}\sum_{\ell_{1}}
\frac{2\ell_{1}+1}{4\pi
}\kappa (\ell_{1})P_{\ell_{1}}\bigl( \langle
x_{1},y_{1} \rangle\bigr)\sum_{\ell
_{2}}
\frac{2\ell_{2}+1}{4\pi}\kappa(\ell_{2})P_{\ell_{2}}\bigl( \langle
x_{2},y_{2} \rangle\bigr)
\\
&&\hspace*{28pt}\qquad\quad{}\times\sum_{\ell}b^{2}\biggl(
\frac{\ell}{B^{j}}\biggr)\frac{2\ell+1}{4\pi}C_{
\mathbb{\ell}}P_{\mathbb{\ell}}\bigl(
\langle y_{1},y_{2} \rangle \bigr)\,dy_{1}\,dy_{2}
\\
&&\qquad=\int_{S^{2}}\sum_{\ell}b^{2}
\biggl(\frac{\ell}{B^{j}}\biggr)\frac{2\ell
+1}{4\pi}%
C_{\ell}\sum
_{\ell_{1}}\kappa(\ell_{1})\sum
_{\ell_{2}}\frac{2\ell
_{2}+1%
}{4\pi}\kappa(\ell_{2})P_{\ell_{2}}
\bigl( \langle x_{2},y_{2} \rangle\bigr)
\\
&&\hspace*{12pt}\qquad\quad{}\times\int_{S^{2}}\frac{2\ell_{1}+1}{4\pi}P_{\ell_{1}}\bigl(
\langle x_{1},y_{1} \rangle\bigr)P_{\ell}\bigl(
\langle y_{1},y_{2} \rangle \bigr)\,dy_{1}\,dy_{2}
\\
&&\qquad=\sum_{\ell}b^{2}\biggl(\frac{\ell}{B^{j}}
\biggr)\kappa(\ell)\frac{2\ell
+1}{4\pi}%
C_{\ell}\sum
_{\ell_{2}}\frac{2\ell_{2}+1}{4\pi}\kappa(\ell _{2})\\
&&\hspace*{23pt}\qquad{}\times \int
_{S^{2}}P_{\ell_{2}}\bigl( \langle x_{2},y_{2}
\rangle \bigr)P_{\ell}\bigl( \langle x_{1},y_{2}
\rangle\bigr)\,dy_{2}
\\
&&\qquad=\sum_{\ell}b^{2}\biggl(\frac{\ell}{B^{j}}
\biggr)\kappa^{2}(\ell)\frac{2\ell
+1}{%
4\pi}C_{\ell}P_{\ell}
\bigl( \langle x_{1},x_{2} \rangle\bigr),
\end{eqnarray*}
as claimed.
\end{pf}

The derivation of analogous results in the case of $q\geq2$ requires more
work and extra notation. In particular, we shall need the Wigner's $3j$
coefficients, which are defined by [for $m_{1}+m_{2}+m_{3}=0$, see \cite
{VMK}, expression (8.2.1.5)]
\begin{eqnarray*}
&&\pmatrix{ \mathbb{\ell}_{1} & \mathbb{
\ell}_{2} & \mathbb{\ell}_{3}
\vspace*{2pt}\cr
m_{1} & m_{2} & m_{3}} \\
&&\qquad :=(-1)^{\mathbb{\ell}_{1}+m_{1}}\sqrt{2\mathbb{\ell
}_{3}+1} \biggl[ \frac{(\mathbb{\ell}_{1}+\mathbb{\ell}_{2}-\mathbb{\ell
}_{3})!(\mathbb{%
\ell}_{1}-\mathbb{\ell}_{2}+\mathbb{\ell}_{3})!(\mathbb{\ell}_{1}-%
\mathbb{\ell}_{2}+\mathbb{\ell}_{3})!}{(\mathbb{\ell}_{1}+\mathbb
{\ell}%
_{2}+\mathbb{\ell}_{3}+1)!} \biggr] ^{1/2}
\\
&&\qquad\quad{} \times \biggl[ \frac{(\mathbb{\ell}_{3}+m_{3})!(\mathbb{\ell
}_{3}-m_{3})!%
}{(\mathbb{\ell}_{1}+m_{1})!(\mathbb{\ell}_{1}-m_{1})!(\mathbb{\ell}%
_{2}+m_{2})!(\mathbb{\ell}_{2}-m_{2})!} \biggr] ^{1/2}
\\
&&\qquad\quad{} \times\sum_{z}\frac{(-1)^{z}(\mathbb{\ell}_{2}+\mathbb{\ell}%
_{3}+m_{1}-z)!(\mathbb{\ell}_{1}-m_{1}+z)!}{z!(\mathbb{\ell
}_{2}+\mathbb{%
\ell}_{3}-\mathbb{\ell}_{1}-z)!(\mathbb{\ell}_{3}+m_{3}-z)!(\mathbb
{\ell}%
_{1}-\mathbb{\ell}_{2}-m_{3}+z)!},
\end{eqnarray*}
where the summation runs over all $z$'s such that the factorials are
nonnegative. This expression becomes somewhat neater for $%
m_{1}=m_{2}=m_{3}=0$, where we have
\begin{eqnarray}\label{appe}
&&\pmatrix
{ \mathbb{\ell}_{1} & \mathbb{
\ell}_{2} & \mathbb{\ell}_{3}
\vspace*{2pt}\cr
0 & 0 & 0}
\nonumber
\\[-8pt]
\\[-8pt]
\nonumber
&&\qquad=
\cases{ %
0,\qquad\mbox{for }\mathbb{\ell}_{1}+\mathbb{\ell}_{2}+\mathbb{\ell }_{3}
\mbox{ odd},
\vspace*{2pt}\cr
\displaystyle (-1)^{{(\mathbb{\ell}_{1}+\mathbb{\ell}_{2}-\mathbb{\ell
}_{3})}/{2}}%
\vspace*{2pt}\cr
\quad{}\times\displaystyle\frac{ [ (\mathbb{\ell}_{1}+\mathbb{\ell}_{2}+\mathbb{\ell
}_{3})/2%
 ] !}{ [ (\mathbb{\ell}_{1}+\mathbb{\ell}_{2}-\mathbb{\ell}%
_{3})/2 ] ! [ (\mathbb{\ell}_{1}-\mathbb{\ell}_{2}+\mathbb
{\ell}%
_{3})/2 ] ! [ (-\mathbb{\ell}_{1}+\mathbb{\ell}_{2}+\mathbb
{\ell}%
_{3})/2 ] !}\vspace*{2pt}\cr
\quad{}\times\displaystyle \biggl\{ \frac{(\mathbb{\ell}_{1}+\mathbb{\ell}_{2}-%
\mathbb{\ell}_{3})!(\mathbb{\ell}_{1}-\mathbb{\ell}_{2}+\mathbb{\ell
}%
_{3})!(-\mathbb{\ell}_{1}+\mathbb{\ell}_{2}+\mathbb{\ell
}_{3})!}{(\mathbb{%
\ell}_{1}+\mathbb{\ell}_{2}+\mathbb{\ell}_{3}+1)!}
\biggr\} ^{1/2},
\vspace*{2pt}\cr
\qquad\hspace*{11pt}\mbox{for }\mathbb{\ell}_{1}+\mathbb{\ell}_{2}+\mathbb{
\ell}_{3}\mbox{ even}.}\nonumber
\end{eqnarray}
It is occasionally more convenient to focus on Clebsch--Gordan coefficients,
which are related to the Wigner's by a simple change of normalization,
for example,
%
\begin{equation}
C_{\ell_{1}m_{1}\ell_{2}m_{2}}^{\ell_{3}m_{3}}:=\frac{(-1)^{\ell
_{3}-m_{3}}}{\sqrt{2\ell_{3}+1}}\pmatrix{ \mathbb{\ell}_{1} & \mathbb{\ell}_{2} &
\mathbb{\ell}_{3}
\vspace*{2pt}\cr
m_{1} & m_{2} & -m_{3}}. \label{cgdef}
\end{equation}
Wigner's $3j$ coefficients are elements of unitary matrices which intertwine
alternative reducible representations of the group of rotations
$\operatorname{SO}(3)$, and
because of this emerge naturally in the evaluation of multiple
integrals of
spherical harmonics (see Section~3.5.2 of \cite{marpecbook}) . As a
consequence, they also appear in the covariances of nonlinear transforms;
for $q=2$, we have indeed

\begin{lemma}
The field $g_{j;2}(x)$ is zero-mean, finite variance and isotropic, with
covariance function
\begin{eqnarray*}
&&\mathbb{E} \bigl\{ g_{j;2}(x_{1})g_{j;2}(x_{2})
\bigr\} \\
&&\qquad=
\frac{2}{\sigma_{\beta_{j}}^{4}}\sum_{\mathbb{\ell}}\kappa ^{2}(
\mathbb{%
\ell})\frac{2\mathbb{\ell}+1}{4\pi}\sum_{\mathbb{\ell}_{1}\mathbb
{\ell}%
_{2}}b^{2}
\biggl(\frac{\mathbb{\ell}_{1}}{B^{j}}\biggr)b^{2}\biggl(\frac{\mathbb{\ell
}_{2}}{%
B^{j}}\biggr)
\\
&&\hspace*{105pt}\qquad\quad{}\times\frac{(2\mathbb{\ell}_{1}+1)(2\mathbb{\ell}_{2}+1)}{4\pi}\\
&&\hspace*{105pt}\qquad\quad{}\times C_{
\mathbb{\ell}_{1}}C_{\mathbb{\ell}_{2}}\pmatrix{ \mathbb{\ell} & \mathbb{\ell}_{1} & \mathbb{
\ell}_{2}
\vspace*{2pt}\cr
0 & 0 & 0} ^{2}P_{\mathbb{\ell}}
\bigl( \langle x_{1},x_{2} \rangle \bigr).
\end{eqnarray*}
\end{lemma}

\begin{pf}
Note first that
\begin{eqnarray*}
&&\mathbb{E} \bigl\{ g_{j;2}(x_{1})g_{j;2}(x_{2})
\bigr\}\\
&&\qquad =\mathbb{E} \biggl\{ \int_{S^{2}}K\bigl( \langle
x_{1},y_{1} \rangle\bigr)H_{2}\bigl(\widetilde {
\beta}%
_{j}(y_{1})\bigr)\,dy_{1}\int
_{S^{2}}K\bigl( \langle x_{2},y_{2} \rangle
\bigr)H_{2}\bigl(%
\widetilde{\beta}_{j}(y_{2})
\bigr)\,dy_{2} \biggr\}
\\
&&\qquad=\int_{S^{2}\times S^{2}}K\bigl( \langle x_{1},y_{1}
\rangle \bigr)K\bigl( \langle x_{2},y_{2} \rangle\bigr)
\mathbb{E} \bigl\{ H_{2}\bigl(\widetilde{%
\beta}_{j}(y_{1})\bigr)H_{2}\bigl(\widetilde{
\beta}_{j}(y_{2})\bigr) \bigr\} \,dy_{1}\,dy_{2}
\\
&&\qquad=\frac{2}{\sigma_{\beta_{j}}^{4}}\int_{S^{2}\times S^{2}}K\bigl( \langle
x_{1},y_{1} \rangle\bigr)K\bigl( \langle
x_{2},y_{2} \rangle\bigr) \\
&&\hspace*{82pt}{}\times\biggl\{ \sum
_{\mathbb{\ell}}b^{2}\biggl(\frac{\mathbb{\ell}}{B^{j}}\biggr)
\frac{2\mathbb
{\ell}%
+1}{4\pi}C_{\mathbb{\ell}}P_{\mathbb{\ell}}\bigl( \langle
y_{1},y_{2} \rangle\bigr) \biggr\} ^{2}\,dy_{1}\,dy_{2}
\\
&&\qquad=\frac{2}{\sigma_{\beta_{j}}^{4}}\int_{S^{2}\times S^{2}}\sum
_{\mathbb{%
\ell}_{1}}\frac{2\mathbb{\ell}_{1}+1}{4\pi}\kappa(\mathbb{\ell
}_{1})P_{%
\mathbb{\ell}_{1}}\bigl( \langle x_{1},y_{1}
\rangle\bigr)\sum_{\mathbb
{%
\ell}_{2}}\frac{2\mathbb{\ell}_{2}+1}{4\pi}\kappa(
\mathbb{\ell }_{2})P_{%
\mathbb{\ell}_{2}}\bigl( \langle x_{2},y_{2}
\rangle\bigr)
\\
&&\hspace*{48pt}\qquad\quad{}\times\sum_{\mathbb{\ell}_{3}\mathbb{\ell}_{4}}b^{2}\biggl(
\frac{\mathbb
{\ell
}_{3}}{B^{j}}\biggr)b^{2}\biggl(\frac{\mathbb{\ell}_{4}}{B^{j}}\biggr)
\frac{2\mathbb
{\ell}%
_{3}+1}{4\pi}\frac{2\mathbb{\ell}_{4}+1}{4\pi}\\
&&\hspace*{76pt}\qquad\quad{}\times C_{\mathbb{\ell
}_{3}}C_{%
\mathbb{\ell}_{4}}P_{\mathbb{\ell}_{3}}
\bigl( \langle y_{1},y_{2} \rangle\bigr)P_{\mathbb{\ell}_{4}}
\bigl( \langle y_{1},y_{2} \rangle\bigr)\,dy_{1}\,dy_{2},
\end{eqnarray*}
where in the third step we have used the covariance formula for Hermite
polynomials in zero-mean, unit variance Gaussian variables (see, e.g.,
\cite{marpecbook}, Remark~4.10)
%
\begin{equation}
\mathbb{E} \bigl\{ H_{q}(X)H_{q^{\prime}}(Y) \bigr\} =\delta
_{q}^{q^{\prime
}}q! \{ \mathbb{E}XY \} ^{q},
\label{covhermite}
\end{equation}
which in this case yields
\[
\mathbb{E} \bigl\{ H_{2}\bigl(\widetilde{\beta}_{j}(y_{1})
\bigr)H_{2}\bigl(\widetilde {\beta }_{j}(y_{2})
\bigr) \bigr\} =\frac{2}{\sigma_{\beta_{j}}^{4}} \biggl\{ \sum_{%
\mathbb{\ell}}b^{2}
\biggl(\frac{\mathbb{\ell}}{B^{j}}\biggr)\frac{2\mathbb{\ell
}+1}{%
4\pi}C_{\mathbb{\ell}}P_{\mathbb{\ell}}
\bigl( \langle y_{1},y_{2} \rangle\bigr) \biggr\}
^{2}.
\]
Now recall that\vspace*{-1pt}
\begin{eqnarray*}
&&\int_{S^{2}}P_{\mathbb{\ell}_{1}}\bigl( \langle
x_{1},y_{1} \rangle\bigr)P_{%
\mathbb{\ell}_{3}}\bigl( \langle
y_{1},y_{2} \rangle\bigr)P_{\mathbb
{\ell}%
_{4}}\bigl( \langle
y_{1},y_{2} \rangle\bigr)\,dy_{1}
\\[-2pt]
&&\qquad=\frac{(4\pi)^{3}}{(2\mathbb{\ell}_{1}+1)(2\mathbb{\ell
}_{3}+1)(2\mathbb{%
\ell}_{4}+1)}
\\[-1pt]
&&\qquad\quad{}\times\int_{S^{2}}\sum_{m_{1}m_{3}m_{4}}Y_{\mathbb{\ell
}_{1}m_{1}}(y_{1})%
\overline{Y}_{\mathbb{\ell}_{1}m_{1}}(x_{1})Y_{\mathbb{\ell}%
_{3}m_{3}}(y_{1})
\overline{Y}_{\mathbb{\ell}_{3}m_{3}}(y_{2})\\
&&\hspace*{93pt}{}\times Y_{\mathbb
{%
\ell}_{4}m_{4}}(y_{1})
\overline{Y}_{\mathbb{\ell}_{4}m_{4}}(y_{2})\,dy_{1}
\\[-1pt]
&&\qquad=\biggl(\frac{(4\pi)^{5}}{(2\mathbb{\ell}_{1}+1)(2\mathbb{\ell
}_{3}+1)(2\mathbb{%
\ell}_{4}+1)}\biggr)^{1/2}
\\[-2pt]
&&\qquad\quad{}\times\sum_{m_{1}m_{3}m_{4}}\pmatrix{ \mathbb{\ell}_{1} & \mathbb{\ell}_{3} & \mathbb{
\ell}_{4}
\vspace*{2pt}\cr
m_{1} & m_{3} & m_{4}%
} \pmatrix{ \mathbb{
\ell}_{1} & \mathbb{\ell}_{3} & \mathbb{\ell}_{4}
\vspace*{2pt}\cr
0 & 0 & 0}\\[-1pt]
&&\qquad\quad{}\times \overline{Y}_{\mathbb{\ell}_{1}m_{1}}(x_{1})
\overline {Y}_{\mathbb{%
\ell}_{3}m_{3}}(y_{2})\overline{Y}_{\mathbb{\ell
}_{4}m_{4}}(y_{2}).\vspace*{-2pt}
\end{eqnarray*}
Likewise\vspace*{-2pt}
\begin{eqnarray*}
&&\int_{S^{2}}P_{\mathbb{\ell}_{2}}\bigl( \langle
x_{2},y_{2} \rangle\bigr)%
\overline{Y}_{\mathbb{\ell}_{3}m_{3}}(y_{2})
\overline{Y}_{\mathbb{\ell
}%
_{4}m_{4}}(y_{2})\,dy_{2}
\\[-1pt]
&&\qquad=\frac{4\pi}{2\mathbb{\ell}_{2}+1}\int_{S^{2}}\sum
_{m_{2}}\overline {Y}_{%
\mathbb{\ell}_{2}m_{2}}(y_{2})Y_{\mathbb{\ell
}_{2}m_{2}}(x_{2})
\overline{Y}%
_{\mathbb{\ell}_{3}m_{3}}(y_{2})\overline{Y}_{\mathbb{\ell}%
_{4}m_{4}}(y_{2})\,dy_{2}
\\[-1pt]
&&\qquad=\sqrt{\frac{(4\pi)(2\mathbb{\ell}_{3}+1)(2\mathbb{\ell
}_{4}+1)}{2\mathbb{%
\ell}_{2}+1}}\\[-1pt]
&&\qquad\quad{}\times \sum_{m_{2}}\pmatrix{ \mathbb{\ell}_{2} & \mathbb{
\ell}_{3} & \mathbb{\ell}_{4}
\vspace*{2pt}\cr
m_{2} & m_{3} & m_{4}}\pmatrix{\mathbb{
\ell}_{2} & \mathbb{\ell}_{3} & \mathbb{\ell}_{4}
\vspace*{2pt}\cr
0 & 0 & 0%
} Y_{\mathbb{\ell}_{2}m_{2}}(x_{2}).\vspace*{-1pt}
\end{eqnarray*}
Using the orthonormality properties of Wigner's $3j$ coefficients (see again
\cite{marpecbook}, Chapter~3.5), we have\vspace*{-1pt}
\[
\sum_{m_{3}m_{4}}\pmatrix{
\mathbb{\ell}_{1} & \mathbb{\ell}_{3} & \mathbb{
\ell}_{4}
\vspace*{2pt}\cr
m_{1} & m_{3} & m_{4}}\pmatrix{ \mathbb{
\ell}_{2} & \mathbb{\ell}_{3} & \mathbb{\ell}_{4}
\vspace*{2pt}\cr
m_{2} & m_{3} & m_{4}} =\frac{\delta_{m_{1}}^{m_{2}}\delta_{\mathbb{\ell
}_{1}}^{\mathbb{%
\ell}_{2}}}{(2\mathbb{\ell}_{1}+1)},\vspace*{-1pt}
\]
whence we get\vspace*{-1pt}
\begin{eqnarray*}
&&\mathbb{E} \bigl\{ g_{j;2}(x_{1})g_{j;2}(x_{2})
\bigr\}\\[-1pt]
&&\qquad =\frac{2}{\sigma_{\beta_{j}}^{4}}\sum_{\mathbb{\ell}}\kappa ^{2}(
\mathbb{%
\ell})\frac{2\mathbb{\ell}+1}{4\pi}\\[-1pt]
&&\hspace*{32pt}\qquad\quad{}\times\sum_{\mathbb{\ell}_{1}\mathbb
{\ell}%
_{2}}b^{2}
\biggl(\frac{\mathbb{\ell}_{1}}{B^{j}}\biggr)b^{2}\biggl(\frac{\mathbb{\ell
}_{2}}{%
B^{j}}\biggr)
\\[-1pt]
&&\hspace*{60pt}\qquad\quad{}\times\frac{(2\mathbb{\ell}_{1}+1)(2\mathbb{\ell}_{2}+1)}{4\pi}C_{
\mathbb{\ell}_{1}}C_{\mathbb{\ell}_{2}}\pmatrix{ \mathbb{\ell} & \mathbb{\ell}_{1} & \mathbb{
\ell}_{2}
\vspace*{2pt}\cr
0 & 0 & 0} ^{2}P_{\mathbb{\ell}}
\bigl( \langle x_{1},x_{2} \rangle \bigr),
\end{eqnarray*}
as claimed. As a special case, the variance is provided by
\begin{eqnarray*}
\mathbb{E}g_{j;2}^{2}(x) &=&\frac{2}{\sigma_{\beta_{j}}^{4}}\sum
_{\mathbb{%
\ell}}\kappa^{2}(\mathbb{\ell})\frac{2\mathbb{\ell}+1}{4\pi}
\\
&&\hspace*{32pt}{}\times\sum_{%
\mathbb{\ell}_{1}\mathbb{\ell}_{2}}b^{2}\biggl(\frac{\mathbb{\ell
}_{1}}{B^{j}}%
\biggr)b^{2}\biggl(\frac{\mathbb{\ell}_{2}}{B^{j}}\biggr)
\\
&&\hspace*{60pt}{}\times\frac{(2\mathbb{\ell}_{1}+1)(2\mathbb{\ell}_{2}+1)}{4\pi}C_{
\mathbb{\ell}_{1}}C_{\mathbb{\ell}_{2}}\pmatrix{ \mathbb{\ell} & \mathbb{\ell}_{1} & \mathbb{
\ell}_{2}
\vspace*{2pt}\cr
0 & 0 & 0} ^{2}.
\end{eqnarray*}
\upqed\end{pf}

\begin{remark}
Since the field $ \{ g_{j;2}(\cdot) \} $ has finite-variance and
it is
isotropic, it admits itself a spectral representation. Indeed, it is a
simple computation to show that the corresponding angular power
spectrum is
provided by
%
\begin{eqnarray}\label{colizzi}
C_{\mathbb{\ell};j,2}&:=&\frac{2}{\sigma_{\beta_{j}}^{4}}\kappa^{2}(%
\mathbb{
\ell})\sum_{\mathbb{\ell}_{1}\mathbb{\ell}_{2}}b^{2}\biggl(
\frac{%
\mathbb{\ell}_{1}}{B^{j}}\biggr)b^{2}\biggl(\frac{\mathbb{\ell}_{2}}{B^{j}}\biggr)
\frac
{(2%
\mathbb{\ell}_{1}+1)(2\mathbb{\ell}_{2}+1)}{4\pi}
\nonumber
\\[-8pt]
\\[-8pt]
\nonumber
&&\hspace*{58pt}{}\times C_{\mathbb{\ell
}_{1}}C_{%
\mathbb{\ell}_{2}}\pmatrix{\mathbb{\ell} & \mathbb{\ell}_{1} & \mathbb{
\ell}_{2}
\vspace*{2pt}\cr
0 & 0 & 0} ^{2},
\end{eqnarray}
for $\mathbb{\ell}=1,2,\ldots.$ This result will have a great relevance
for the
practical implementation of the findings in the next sections.
\end{remark}

\subsubsection{Higher-order transforms}

The general case of nonlinear transforms with $q\geq3$ can be dealt with
analogous lines; the main difference being the appearance of multiple
integrals of spherical harmonics of order greater than 3, and hence
so-called higher order Gaunt integrals and convolutions of Clebsch--Gordan
coefficients. For brevity's sake, we provide only the basic details; we
refer to \cite{marpecbook} for a more detailed discussion on nonlinear
transforms of Gaussian spherical harmonics. Here, we simply recall the
definition of the multiple Gaunt integral (see~\cite{marpecbook}, Remark~6.30 and Theorem~6.31), which is given by
\[
\mathcal{G}(\ell_{1},m_{1};\ldots \ell_{q},m_{q};
\ell,m):=\int_{S^{2}}Y_{\ell
_{1}m_{1}}(x)\cdots  Y_{\ell_{q}m_{q}}(x)Y_{\ell m}(x)\,d
\sigma(x),
\]
where the coefficients $\mathcal{G}(\ell_{1},m_{1};\ldots \ell
_{q},m_{q};\ell
,m)$ can be expressed as multiple convolution of Wigner/Clebsch--Gordan terms
(see \ref{cgdef}),
\begin{eqnarray*}
&&\mathcal{G}(\ell_{1},m_{1};\ldots \ell_{q},m_{q};
\ell,m)\\
&&\qquad=(-1)^{m}\sqrt {\frac{%
(2\ell_{1}+1)\cdots (2\ell_{q}+1)}{(4\pi)^{q-1}(2\ell+1)}}
\\
&&\qquad\quad{}\times\sum_{\lambda_{1}\cdots \lambda_{q-2}}C_{\ell_{1}0\ell
_{2}0}^{\lambda
_{1}0}C_{\lambda_{1}0\ell_{3}0}^{\lambda_{2}0}\cdots  C_{\lambda
_{q-2}0\ell
_{q}0}^{\ell0}
\\
&&\qquad\quad{}\times\sum_{\mu_{1}\cdots \mu_{q-2}}C_{\ell_{1}m_{1}\ell
_{2}m_{2}}^{\lambda_{1}\mu_{1}}C_{\lambda_{1}\mu_{1}\ell
_{3}m_{3}}^{\lambda_{2}\mu_{2}}\cdots  C_{\lambda_{q-2}\mu_{q-2}\ell
_{q}m_{q}}^{\ell m}.
\end{eqnarray*}
Following also \cite{marpecbook}, equation (6.40), let us introduce the
shorthand
notation
%
\begin{eqnarray} \label{colizzi2}
C_{\ell_{1}0\ell_{2}0\cdots \ell_{q}0}^{\lambda_{1}\cdots \lambda_{q-2}\ell
0}&:=&C_{\ell_{1}0\ell_{2}0}^{\lambda_{1}0}C_{\lambda_{1}0\ell
_{3}0}^{\lambda_{2}0}\cdots  C_{\lambda_{q-2}0\ell_{q}0}^{\ell0},
\nonumber
\\[-8pt]
\\[-8pt]
\nonumber
\mathcal{C}(\ell_{1},\ldots,\ell_{q},\ell)
&:=&\sum
_{\lambda_{1}\cdots \lambda
_{q-2}} \bigl\{ C_{\ell_{1}0\ell_{2}0\cdots \ell_{q}0}^{\lambda
_{1}\cdots \lambda
_{q-2}\ell0} \bigr\}
^{2}.
\end{eqnarray}
It should be noted that, from the unitary properties of Clebsch--Gordan
coefficients
\begin{eqnarray*}
&&\sum_{\ell}\mathcal{C}(\ell_{1},\ldots,
\ell_{q},\ell)\\
&&\qquad=\sum_{\lambda
_{1}\cdots \lambda_{q-2}} \bigl\{
C_{\ell_{1}0\ell_{2}0}^{\lambda
_{1}0} \bigr\} ^{2}\cdots \sum
_{\ell} \bigl\{ C_{\lambda_{q-2}0\ell
_{q}0}^{\ell0} \bigr\}
^{2}=\cdots =1.
\end{eqnarray*}

\begin{lemma}
For general $q\geq3$, the field $g_{j;q}(x)$ is zero-mean, finite variance
and isotropic, with covariance function
\begin{eqnarray*}
&&\mathbb{E} \bigl\{ g_{j;q}(x_{1})g_{j;q}(x_{2})
\bigr\}
\\
&&\qquad=\frac{q!}{\sigma_{\beta_{j}}^{2q}}\sum_{\mathbb{\ell}}\kappa^{2}(%
\mathbb{\ell})\sum_{\mathbb{\ell}_{1}\cdots \mathbb{\ell}_{q}}\mathcal{C} 
(
\ell_{1},\ldots,\ell_{q},\ell)\\
&&\hspace*{116pt}{}\times \Biggl[ \prod
_{k=1}^{q}b^{2}\biggl(\frac
{\mathbb{%
\ell}_{k}}{B^{j}}
\biggr)\frac{2\mathbb{\ell}_{k}+1}{4\pi}C_{\mathbb{\ell
}_{k}}%
 \Biggr] P_{\mathbb{\ell}}
\bigl( \langle x_{1},x_{2} \rangle\bigr).
\end{eqnarray*}
\end{lemma}

\begin{pf}
We have
\begin{eqnarray*}
\mathbb{E}g_{j;q}^{2}(x) &=&\mathbb{E} \biggl\{ \int
_{S^{2}}\int_{S^{2}}K\bigl( \langle
x,y_{1} \rangle\bigr)K\bigl( \langle x,y_{2} \rangle
\bigr)H_{q}\bigl(\widetilde{\beta }_{j}(y_{1})
\bigr)H_{q}\bigl(\widetilde{%
\beta}_{j}(y_{2})
\bigr)\,dy_{1}\,dy_{2} \biggr\}
\\
&=&\frac{q!}{\sigma_{\beta_{j}}^{2q}}\int_{S^{2}}\int_{S^{2}}K
\bigl(%
\langle x,y_{1} \rangle\bigr)K\bigl( \langle
x,y_{2} \rangle\bigr) \\
&&\hspace*{48pt}{}\times\biggl\{ \sum_{\mathbb{\ell}}b^{2}
\biggl(\frac{\ell}{B^{j}}\biggr)\frac{2\ell+1}{4\pi}P_{
\mathbb{\ell}}\bigl( \langle
y_{1},y_{2} \rangle\bigr) \biggr\} ^{q}\,dy_{1}\,dy_{2},
\end{eqnarray*}
where we have used the covariance formula for Hermite polynomials (\ref%
{covhermite}). It is convenient here to view $T_{\mathbb{\ell
}}(x),\beta
_{j}(x)$ as isonormal processes of the form
\begin{eqnarray*}
T_{\mathbb{\ell}}(x) &=&\int_{S^{2}}\sqrt{\frac{2\ell+1}{4\pi
}C_{\mathbb{%
\ell}}}P_{\mathbb{\ell}}
\bigl( \langle x,y \rangle\bigr)\,dW(y),
\\
\beta_{j}(x) &=&\frac{1}{\sigma_{\beta_{j}}}\int_{S^{2}}\sum
_{\mathbb
{%
\ell}}b\biggl(\frac{\ell}{B^{j}}\biggr)\sqrt{
\frac{2\ell+1}{4\pi}C_{\mathbb
{\ell}}}%
P_{\mathbb{\ell}}\bigl( \langle x,y
\rangle\bigr)\,dW(y),
\end{eqnarray*}
where $dW(y)$ denotes a Gaussian white noise measure on the sphere,
whence
\begin{eqnarray*}
&&H_{q}\bigl(\beta_{j}(x)\bigr)
\\
&&\qquad=\frac{1}{\sigma_{\beta_{j}}^{q}}\sum_{\mathbb{\ell}_{1}\cdots \mathbb
{\ell
}_{q}}b\biggl(
\frac{\ell_{1}}{B^{j}}\biggr)\cdots  b\biggl(\frac{\ell_{q}}{B^{j}}\biggr)\sqrt
{%
\prod_{i=1}^{q} \biggl
\{ \frac{2\ell_{i}+1}{4\pi}C_{\mathbb{\ell}%
_{i}} \biggr\} }
\\
&&\hspace*{40pt}\qquad\quad{}\times\int_{ \{ S^{2}\times\cdots \times S^{2} \} ^{\prime
}}P_{%
\mathbb{\ell}_{1}}\bigl( \langle
x,y_{1} \rangle\bigr)\cdots  P_{\mathbb
{\ell}%
_{q}}\bigl( \langle
x,y_{q} \rangle\bigr)\,dW(y_{1})\cdots  dW(y_{q}).
\end{eqnarray*}
Here, the domain of integration excludes the ``diagonals,'' that is,
\[
\bigl\{ S^{2}\times\cdots \times S^{2} \bigr\}
^{\prime}:= \bigl\{ (x_{1},\ldots,x_{q})\in
S^{2}\times\cdots \times S^{2}\dvtx x_{i}\neq
x_{j}\mbox{ for all }i\neq j \bigr\} ,
\]
and we are using the characterization of Hermite polynomials as multiple
Wiener--It\^{o} integrals; see, for instance, Theorem~2.7.7 in \cite
{noupebook}. We
are thus led to
\begin{eqnarray*}
g_{j;q}(z) &=&\frac{1}{\sigma_{\beta_{j}}^{q}}\int_{S^{2}}\sum
_{\ell
}\kappa(\ell)\frac{2\ell+1}{4\pi}P_{\ell}
\bigl( \langle z,x \rangle \bigr)\\
&&\hspace*{31pt}{}\times \sum_{\mathbb{\ell}_{1}\cdots \mathbb{\ell}_{q}}b\biggl(
\frac{\ell_{1}}{B^{j}} 
\biggr)\cdots  b\biggl(\frac{\ell_{q}}{B^{j}}\biggr)\sqrt
{\prod_{i=1}^{q} \biggl\{
\frac{%
2\ell_{i}+1}{4\pi}C_{\mathbb{\ell}_{i}} \biggr\} }
\\
&&\hspace*{31pt}{}\times\int_{S^{2}\times\cdots \times S^{2}}P_{\mathbb{\ell}%
_{1}}\bigl( \langle
x,y_{1} \rangle\bigr)\cdots\\
&&\hspace*{88pt}{}\times  P_{\mathbb{\ell}%
_{q}}\bigl( \langle
x,y_{q} \rangle\bigr)\,dW(y_{1})\cdots  dW(y_{q})\,dx.
\end{eqnarray*}
Using the isometry property of stochastic integrals, it follows easily
that
\begin{eqnarray*}
&&
\mathbb{E} \bigl\{ g_{j;q}(z_{1})g_{j;q}(z_{2})
\bigr\}\\
&&\qquad =
\frac{q!}{\sigma_{\beta_{j}}^{2q}}\int_{S^{2}\times S^{2}}\sum
_{\ell
_{1}\ell_{2}}\frac{2\ell_{1}+1}{4\pi}\kappa(\ell_{1})
\frac{2\ell
_{2}+1%
}{4\pi}\kappa(\ell_{2})P_{\ell_{1}}\bigl( \langle
z_{1},x_{1} \rangle\bigr)P_{\ell_{2}}\bigl( \langle
z_{2},x_{2} \rangle\bigr)
\\
&&\hspace*{50pt}\qquad\quad{}\times\sum_{\mathbb{\ell}_{1}\cdots \mathbb{\ell}_{q}}b^{2}\biggl(
\frac{\ell
_{1}%
}{B^{j}}\biggr)\cdots  b^{2}\biggl(\frac{\ell_{q}}{B^{j}}\biggr)\sqrt
{\prod_{i=1}^{q}%
\biggl\{ \frac{2\ell_{i}+1}{4\pi}C_{\mathbb{\ell}_{i}} \biggr\} }\\
&&\hspace*{50pt}\qquad\quad{}\times P_{\mathbb
{\ell
}_{1}}\bigl(
\langle x_{1},x_{2} \rangle\bigr)\cdots  P_{\mathbb{\ell}%
_{q}}\bigl(
\langle x_{1},x_{2} \rangle\bigr)\,dx_{1}\,dx_{2}.
\end{eqnarray*}
Now write
\begin{eqnarray*}
&&\frac{(2\mathbb{\ell}_{1}+1)\cdots (2\mathbb{\ell}_{q}+1)}{(4\pi)^{q}}P_{
\mathbb{\ell}_{1}}\bigl( \langle x_{1},x_{2}
\rangle\bigr)\cdots  P_{\mathbb
{%
\ell}_{q}}\bigl( \langle x_{1},x_{2}
\rangle\bigr)
\\
&&\qquad=\sum_{m_{1}\cdots  m_{q}}Y_{\mathbb{\ell}_{1}m_{1}}(x_{1})\cdots  Y_{\mathbb
{\ell}%
_{q}m_{q}}(x_{1})
\overline{Y}_{\mathbb{\ell
}_{1}m_{1}}(x_{2})\cdots \overline{Y}%
_{\mathbb{\ell}_{q}m_{q}}(x_{2})
\end{eqnarray*}
so that
\begin{eqnarray*}
&&\frac{(2\mathbb{\ell}_{1}+1)\cdots (2\mathbb{\ell}_{q}+1)}{(4\pi)^{q}}%
\int_{S^{2}\times S^{2}}P_{\ell_{1}}
\bigl( \langle z_{1},x_{1} \rangle \bigr)P_{\ell_{2}}
\bigl( \langle z_{2},x_{2} \rangle\bigr)\\
&&\hspace*{110pt}\qquad{}\times P_{\mathbb{\ell}%
_{1}}
\bigl( \langle x_{1},x_{2} \rangle\bigr)\cdots  P_{\mathbb{\ell}%
_{q}}
\bigl( \langle x_{1},x_{2} \rangle\bigr)\,dx_{1}\,dx_{2}
\\
&&\qquad=\sum_{\mu_{1}\mu_{2}}\sum_{m_{1}\cdots  m_{q}}
\mathcal{G}(\mathbb{\ell} 
_{1},m_{1};\ldots \mathbb{
\ell}_{q},m_{q};\ell_{1},\mu_{1})
\\
&&\hspace*{46pt}\qquad\quad{}\times\mathcal{G}(\mathbb{\ell}_{1},m_{1};
\ldots \mathbb{
\ell}%
_{q},m_{q};\ell_{2},
\mu_{2}) \biggl\{ \frac{4\pi}{2\ell+1}Y_{\ell
_{1}\mu
_{1}}(z_{1})
\overline{Y}_{\ell_{2}\mu_{2}}(z_{2}) \biggr\}
\\
&&\qquad=\frac{4\pi}{2\ell+1}\sum_{\mu_{1}\mu_{2}}Y_{\ell_{1}\mu
_{1}}(z_{1})%
\overline{Y}_{\ell_{2}\mu_{2}}(z_{2})\delta_{\ell_{1}}^{\ell
_{2}}
\delta _{\mu_{1}}^{\mu_{2}}=P_{\ell_{1}}\bigl( \langle
z_{1},z_{2} \rangle\bigr)%
.
\end{eqnarray*}
The general case $q\geq3$ hence yields (see also \cite{marpecbook}, Theorem~7.5 for a related computation)
\begin{eqnarray*}
&&\mathbb{E}g_{j;q}^{2}(x)\\
&&\qquad=
\frac{q!}{\sigma_{\beta_{j}}^{2q}}\sum_{\mathbb{\ell}}\kappa ^{2}(
\mathbb{%
\ell})\sum_{\mathbb{\ell}_{1}\cdots \mathbb{\ell}_{q}}\mathcal{C}(\ell
_{1},\ldots,\ell_{q},\ell)b^{2}\biggl(
\frac{\mathbb{\ell
}_{1}}{B^{j}}\biggr)\cdots  b^{2}\biggl(%
\frac{\mathbb{\ell}_{q}}{B^{j}}
\biggr)\frac{2\mathbb{\ell}_{1}+1}{4\pi
}\cdots\\
&&\hspace*{116pt}{}\times  \frac{%
2\mathbb{\ell}_{q}+1}{4\pi}C_{\mathbb{\ell}_{1}}\cdots  C_{\mathbb{\ell
}_{q}}%
\end{eqnarray*}
and
\begin{eqnarray*}
&&\mathbb{E} \bigl\{ g_{j;q}(x)g_{j;q}(y) \bigr\}
\\
&&\qquad=\frac{q!}{\sigma_{\beta_{j}}^{2q}}\sum_{\mathbb{\ell}}\kappa^{2}(
\mathbb{\ell})\sum_{\mathbb{\ell}_{1}\cdots \mathbb{\ell}_{q}}\mathcal{C}
(\ell_{1},\ldots,\ell_{q},\ell)b^{2}
\biggl(\frac{\mathbb{\ell}_{1}}{B^{j}}%
\biggr)\cdots  b^{2}\biggl(
\frac{\mathbb{\ell}_{q}}{B^{j}}\biggr)
\\
&&\hspace*{84pt}\qquad\quad{}\times\frac{2\mathbb{\ell}_{1}+1}{4\pi}\cdots \frac{2\mathbb{\ell
}_{q}+1}{%
4\pi}C_{\mathbb{\ell}_{1}}\cdots  C_{\mathbb{\ell}_{q}}P_{\mathbb{\ell}%
}
\bigl( \langle x_{1},x_{2} \rangle\bigr),
\end{eqnarray*}
as claimed.
\end{pf}

\begin{remark}
It is immediately checked that the angular power spectrum of
$g_{j;q}(y)$ is
given by [see (\ref{colizzi2})]
%
\begin{eqnarray}\label{porteaperte2}
C_{\mathbb{\ell};j,q}&:=&\frac{q!}{\sigma_{\beta_{j}}^{2q}}\frac{4\pi
}{%
2\ell+1}\kappa^{2}(
\mathbb{\ell})
\nonumber
\\[-8pt]
\\[-8pt]
\nonumber
&&{}\times
\sum_{\mathbb{\ell}_{1}\cdots \mathbb
{\ell}%
_{q}}\mathcal{C}(
\ell_{1},\ldots,\ell_{q},\ell)\prod
_{k=1}^{q} \biggl[ b^{2}\biggl(
\frac{\mathbb{\ell}_{k}}{B^{j}}\biggr)\frac{2\mathbb{\ell
}_{k}+1}{4\pi}C_{%
\mathbb{\ell}_{k}} \biggr].
\end{eqnarray}
As a special case, for $q=2$ we recover the previous result (\ref
{colizzi})
%
\begin{eqnarray}\label{powspe}
C_{\mathbb{\ell};j,2}&=&\frac{2!}{\sigma_{\beta_{j}}^{4}}\kappa^{2}(%
\mathbb{
\ell})\sum_{\mathbb{\ell}_{1}\mathbb{\ell}_{2}}b^{2}\biggl(
\frac{%
\mathbb{\ell}_{1}}{B^{j}}\biggr)b^{2}\biggl(\frac{\mathbb{\ell}_{2}}{B^{j}}\biggr)
\frac{
(2\ell_{1}+1)(2\ell_{2}+1)}{4\pi}\nonumber\\
&&\hspace*{56pt}{}\times C_{\mathbb{\ell}_{1}}C_{\mathbb
{\ell}%
_{2}}\pmatrix{ \mathbb{\ell} & \mathbb{\ell}_{1} & \mathbb{
\ell}_{2}
\vspace*{2pt}\cr
0 & 0 & 0} ^{2}
\nonumber
\\[-8pt]
\\[-8pt]
\nonumber
&=&\frac{2!}{\sigma_{\beta_{j}}^{4}}\kappa^{2}(\mathbb{\ell})\frac
{4\pi}{%
2\ell+1}\sum
_{\mathbb{\ell}_{1}\mathbb{\ell}_{2}}\mathcal{C}(\ell _{1},
\ell_{2},\ell)b^{2}\biggl(\frac{\mathbb{\ell}_{1}}{B^{j}}
\biggr)b^{2}\biggl(\frac{
\mathbb{\ell}_{2}}{B^{j}}\biggr)\\
&&\hspace*{89pt}{}\times \frac{(2\ell_{1}+1)}{4\pi}
\frac{(2\ell
_{2}+1)}{%
4\pi}C_{\mathbb{\ell}_{1}}C_{\mathbb{\ell}_{2}},\nonumber
\end{eqnarray}
because
\[
\mathcal{C}(\ell_{1},\ell_{2},\ell)= \bigl\{
C_{\ell_{1}0\ell
_{2}0}^{\ell0} \bigr\} ^{2}=(2\ell+1)\pmatrix{ \mathbb{\ell} & \mathbb{\ell}_{1} &
\mathbb{\ell}_{2}
\vspace*{2pt}\cr
0 & 0 & 0} ^{2}.
\]
\end{remark}

\section{Weak convergence}
\label{sec:weak:cgs}

In this section, we provide our main convergence results. It must be
stressed that the convergence we study here is in some sense different from
the standard theory as presented, for instance, by \cite{Billingsley}, but
refers instead to the broader notion developed by \cite{davydov,aristotilediaconis}; see also \cite{dudley}, Chapter~11.

We start first from the following conditions (see, e.g., \cite
{bkmpAoS,marpecbook,mayeli}):

\begin{condition}\label{ConditionB} The angular power spectrum has the form
\[
C_{\ell}=G(\ell)\ell^{-\alpha},\qquad\ell=1,2,\ldots,
\]
where $\alpha>2$ and $G\dvtx \mathbb{R}^{+}\mathbb{\rightarrow R}^{+}$ is such
that, for all $u>0$,
\begin{eqnarray*}
0 &<&c_{0}\leq G(\cdot)\leq d_{0},
\\
\biggl\llvert \frac{d^{r}}{du^{r}}G(u)\biggr\rrvert &\leq&c_{r}u^{-r},\qquad
r=1,2,\ldots,M\in\mathbb{N}.
\end{eqnarray*}
\end{condition}

\begin{condition}\label{ConditionA} The Kernel $K(\cdot)$ and the field $
\{
\beta_{j}(\cdot) \} $ are such that, for all $j=1,2,3,\ldots $
\[
\operatorname{Var} \biggl\{ \int_{S^{2}}K\bigl( \langle x,y \rangle
\bigr)H_{q}\bigl(\widetilde {%
\beta}_{j}(y)
\bigr)\,dy \biggr\} =\sigma_{j}^{2}B^{-2j}\qquad\mbox{for all }j=1,2,\ldots
\]
and there exist positive constants $c_{1},c_{2}$ such that $c_{1}\leq
\sigma
_{j}^{2}\leq c_{2}$ (note that the right-hand side does not depend on
$x$ by
isotropy).
\end{condition}

These assumptions are mild and it is easy to find many physical examples
such that they are fulfilled. In particular, Condition \ref{ConditionB} is
fulfilled when $G(\ell)=P(\ell)/Q(\ell)$ and $P(\ell),Q(\ell)>0$ are
two positive polynomials of the same order. In the now dominant Bardeen's
potential model for the angular power spectrum of the cosmic microwave
background radiation (which is theoretically justified by the so-called
inflationary paradigm for the Big Bang Dynamics; see, e.g., \cite
{Durrer,dode2004}) one has $C_{\ell}\sim(\ell(\ell+1))^{-1}$ for the
observationally relevant range $\ell\leq5\times10^{3}$ (the decay becomes
faster at higher multipoles, in view of the so-called Silk damping effect,
but these multipoles are far beyond observational capacity). This is clearly
in good agreement with Condition \ref{ConditionB}. On the other hand,
assuming that Condition~\ref{ConditionB} holds and taking, for instance, $
K( \langle x,y \rangle)\equiv1$ [e.g., focusing on the integral
of the field $ \{ H_{q}(\widetilde{\beta}_{j}(y))\}$],
Condition \ref{ConditionA} has been shown to be satisfied by \cite{cammar}.
Indeed, it is readily checked that $ \{ H_{q}(\widetilde{\beta}%
_{j}(y)) \} $ is a polynomial of finite order (the integer part of
$%
B^{q(j+1)}$), and we can hence consider the following heuristic
argument: we
have
\begin{eqnarray*}
\int_{S^{2}}K\bigl( \langle x,y \rangle\bigr)H_{q}
\bigl(\widetilde{\beta}%
_{j}(y)\bigr)\,dy &=&\int
_{S^{2}}H_{q}\bigl(\widetilde{\beta}_{j}(y)
\bigr)\,dy
\\
&=&\sum_{k\in\mathcal{X}_{j}}H_{q}\bigl(\widetilde{
\beta}_{j}(\xi _{jk})\bigr)\lambda_{jk},
\end{eqnarray*}
where $ \{ \xi_{jk},\lambda_{jk} \} $ are a set of cubature
points and weights (see \cite{npw1,bkmpBer}); indeed, because
the $%
\beta_{j}(\cdot)$ are band-limited (polynomial) functions, this Riemann sum
approximations can be constructed to be exact (by the so-called cubature
formula established in \cite{npw1}; see also \cite{bkmpBer} for some
discussion), with weights $\lambda_{jk}$ of order $\simeq B^{-2j}$. It is
now known that under Condition \ref{ConditionB}, it is possible to establish
a fundamental decorrelation inequality which will play a crucial role
in our
proof below (see also \cite{bkmpAoS,spalan,mayeli}). Indeed,
exploiting (\ref{covhermite}) and (\ref{localineq}) we have that for
any $%
M\in\mathbb{N}$, there exists a constant $C_{M}$ such that
\[
\operatorname{Cov} \bigl\{ H_{q}\bigl(\widetilde{\beta}_{j}(
\xi_{jk_{1}})\bigr),H_{q}\bigl(\widetilde {%
\beta}_{j}(\xi_{jk_{2}})\bigr) \bigr\} \leq\frac{C_{M}q!}{ \{
1+B^{j}\,d(\xi
_{jk_{1}},\xi_{jk_{2}}) \} ^{qM}},
\]
entailing that the terms $H_{q}(\beta_{j}(\xi_{jk}))$ can be treated as
asymptotically uncorrelated, for large $j$. Hence, heuristically
\begin{eqnarray*}
\operatorname{Var} \biggl\{ \sum_{k\in\mathcal{X}_{j}}H_{q}\bigl(
\widetilde{\beta}_{j}(\xi _{jk})\bigr)\lambda_{jk}
\biggr\} &\simeq&\sum_{k\in\mathcal
{X}_{j}}\operatorname{Var} \bigl\{
H_{q}\bigl(\widetilde{\beta}_{j}(\xi_{jk})\bigr)
\bigr\} \lambda_{jk}^{2}
\\
&\simeq&C_{q}\sum_{k\in\mathcal{X}_{j}}
\lambda_{jk}^{2}\simeq C_{q}B^{-2j},
\end{eqnarray*}
because $\sum_{k\in\mathcal{X}_{j}}\lambda_{jk}=4\pi$.

\begin{example}
For $q=2$, we obtain
\begin{eqnarray*}
&&\operatorname{Var} \biggl\{ \int_{S^{2}}\bigl(\widetilde{
\beta}_{j}^{2}(y)-1\bigr)\,dy \biggr\} \\
&&\qquad=\operatorname{Var} \biggl\{ \int
_{S^{2}}\widetilde{\beta}_{j}^{2}(y)\,dy
\biggr\}\\
&&\qquad=\frac{1}{ \{ \sum_{\mathbb{\ell}}b^{2}({\mathbb{\ell
}}/{B^{j}}%
)(2\mathbb{\ell}+1)C_{\mathbb{\ell}} \} ^{2}}\operatorname{Var}
 \biggl\{ \int_{S^{2}}%
 \biggl[ \sum_{\ell m}b\biggl(\frac{\ell}{B^{j}}
\biggr)a_{\ell m}Y_{\ell m}(y) \biggr] ^{2}\,dy \biggr\}
\\
&&\qquad=\frac{1}{ \{ \sum_{\mathbb{\ell}}b^{2}({\mathbb{\ell
}}/{B^{j}}%
)(2\mathbb{\ell}+1)C_{\mathbb{\ell}} \} ^{2}}\operatorname{Var} \Biggl\{ \sum_{\ell
}b^{2}
\biggl(\frac{\ell}{B^{j}}\biggr)\sum_{m=-\ell}^{\ell}
\llvert a_{\ell
m}\rrvert ^{2} \Biggr\} ,
\end{eqnarray*}
where we have used \eqref{norcia}, \eqref{ptrf} and the ortho-normality
properties of spherical harmonics, that is,
\[
\int_{S^{2}}Y_{\ell m}(y)\overline{Y}_{\ell^{\prime}m^{\prime
}}(y)\,dy=
\delta_{\ell}^{\ell^{\prime}}\delta_{m}^{m^{\prime}}.
\]
Now write
\[
\widehat{C}_{\mathbb{\ell}}:=\frac{1}{2\ell+1}\sum_{m=-\ell}^{\ell
}
\llvert a_{\ell m}\rrvert ^{2},
\]
the so-called sample angular power spectrum; it is readily verified
that $%
\widehat{C}_{\mathbb{\ell}}/C_{\ell}$ obeys a chi-square law with
$(2\ell
+1)$ degrees of freedom, whence we obtain
\begin{eqnarray*}
&&\operatorname{Var} \biggl\{ \int_{S^{2}}\widetilde{\beta}_{j}^{2}(y)\,dy
\biggr\}
\\
&&\qquad=\frac{\operatorname{Var} \{ \sum_{\mathbb{\ell}}b^{2}({\mathbb{\ell
}}/{B^{j}}%
)(2\mathbb{\ell}+1)\widehat{C}_{\mathbb{\ell}} \} }{ \{ \sum_{%
\mathbb{\ell}}b^{2}({\mathbb{\ell}}/{B^{j}})(2\mathbb{\ell}+1)C_{
\mathbb{\ell}} \} ^{2}}=\frac{\sum_{\mathbb{\ell}}b^{4}(
{\mathbb{%
\ell}}/{B^{j}})(2\mathbb{\ell}+1)^{2}\operatorname{Var}(\widehat{C}_{\mathbb{\ell
}})}{%
 \{ \sum_{\mathbb{\ell}}b^{2}({\mathbb{\ell
}}/{B^{j}})(2\mathbb{%
\ell}+1)C_{\mathbb{\ell}} \} ^{2}}
\\
&&\qquad=\frac{2\sum_{\mathbb{\ell}=B^{j-1}}^{B^{j+1}}b^{4}({\mathbb
{\ell}}/{%
B^{j}})(2\mathbb{\ell}+1)C_{\mathbb{\ell}}^{2}}{ \{ \sum_{\mathbb
{\ell
}}b^{2}({\mathbb{\ell}}/{B^{j}})(2\mathbb{\ell}+1)C_{\mathbb{\ell
}%
} \} ^{2}}\simeq\frac{B^{j(2-2\alpha)}}{ \{ B^{j(2-\alpha
)} \} ^{2}}\simeq B^{-2j},
\end{eqnarray*}
as claimed.
\end{example}

\subsection{Finite-dimensional distributions}
\label{fdd}

The general technique we shall exploit to establish the central limit
theorem is based upon sharp bounds on normalized fourth-order cumulants.
Note that, in view of results from \cite{nourdinpeccati}, this will actually
entail a stronger form of convergence, more precisely in total variation
norm (see~\cite{nourdinpeccati}).

We start by recalling that the field $ \{ \widetilde{\beta}%
_{j}(\cdot) \} $ can be expressed in terms of the isonormal Gaussian
process, for example, as a stochastic integral
\begin{eqnarray*}
\widetilde{\beta}_{j}(y)&:=&\frac{1}{\sigma_{\beta_{j}}}\sum
_{\mathbb
{\ell}%
}b\biggl(\frac{\mathbb{\ell}}{B^{j}}\biggr)T_{\mathbb{\ell}}(y)\\
&=&
\frac{1}{\sigma
_{\beta_{j}}}\sum_{\mathbb{\ell}}b\biggl(\frac{\mathbb{\ell}}{B^{j}}
\biggr)\sqrt{ 
\frac{(2\mathbb{\ell}+1)C_{\ell}}{4\pi}}\int_{S^{2}}P_{\mathbb{\ell}
}
\bigl( \langle y,z \rangle\bigr)W(dz),
\end{eqnarray*}
where $W(A)$ is a white noise Gaussian measure on the sphere, which
satisfies
\[
\mathbb{E}W(A)=0, \qquad\mathbb{E} \bigl\{ W(A)W(B) \bigr\} =\int_{A\cap B}\,dz\qquad
\mbox{for all }A,B\in\mathcal{B}\bigl(S^{2}\bigr).
\]
It thus follows immediately that the transformed process $ \{ H_{q}(
\widetilde{\beta}_{j}(\cdot)) \} $ belongs to the $q$th order Wiener
chaos; see \cite{nourdinpeccati,noupebook} for more discussion and
detailed definitions. Let us now recall the definition of the \emph{total
variation} distance between the laws of two random variables $X$ and $Z$,
which is given by
\[
d_\mathrm{TV}(X,Z)=\sup_{A\in\mathcal{B}(\mathbb{R})}\bigl\llvert \Pr(W\in A)-\Pr
(X\in A)\bigr\rrvert .
\]
When $Z$ is a standard Gaussian and $X$ is a zero-mean, {unit variance}
random variable which belongs to the $q$th order Wiener chaos of a Gaussian
measure, the following remarkable inequality holds for the total variation
distance
\[
d_\mathrm{TV}(X,Z)\leq\sqrt{\frac{q-1}{3q}\mathrm{cum}_{4}(X)};
\]
see again \cite{nourdinpeccati,noupebook} for more discussion
and a
full proof.

From now on, we shall normalize the fields $ \{ g_{j;q} \} $ to
make them unit variance, that is, we shall define
\[
\widetilde{g}_{j;q}(x):=\frac{g_{j;q}(x)}{\sqrt{\mathbb
{E}g_{j;q}^{2}(x)}}%
;
\]
also, we introduce an isotropic zero-mean Gaussian process $f_{j;q}$, with
the same covariance function as that of $\widetilde{g}_{j;q}$. Our next
result will establish the asymptotic convergence of the finite-dimensional
distributions for $\widetilde{g}_{j;q}$ and $f_{j;q}$. In particular, we
have the following.

\begin{lemma}
For any fixed vector $(x_{1},\ldots,x_{p})$ in $S^{2}$, we have that
\[
d_\mathrm{TV} \bigl( \bigl( \widetilde{g}_{j;q}(x_{1}),
\ldots,\widetilde{g}%
_{j;q}(x_{p}) \bigr) , \bigl(
f_{j;q}(x_{1}),\ldots,f_{j;q}(x_{p})
\bigr) \bigr) =o(1),
\]
as $j\rightarrow\infty$.
\end{lemma}

\begin{pf}
For notational simplicity, we shall focus on the univariate case. In this
case, the Nourdin--Peccati inequality \cite{nourdinpeccati,noupebook} can be restated as
%
\begin{equation}
d_\mathrm{TV} \biggl( \frac{g_{j;q}(x)}{\sqrt{\mathbb{E}g_{j;q}^{2}(x)}}%
,N(0,1) \biggr) \leq
\sqrt{\frac{q-1}{3q}\mathrm{cum}_{4}\biggl(\frac{g_{j;q}(x)}{\sqrt{
\mathbb{E}g_{j;q}^{2}(x)}}
\biggr)}. \label{baglioni}
\end{equation}
In view of (\ref{baglioni}), for the central limit theorem to hold we shall
only need to study the limiting behaviour of the normalized fourth-order
cumulant of $g_{j;q}$. Let us then consider
\begin{eqnarray*}
&&\mathrm{cum}_{4} \bigl\{ g_{j;q}(x) \bigr\}\\
&&\qquad =
\int_{ \{ S^{2} \} ^{\otimes4}}K\bigl( \langle x,y_{1} \rangle
\bigr)\cdots  K\bigl( \langle x,y_{4} \rangle\bigr)\\
&&\hspace*{62pt}{}\times \mathrm{cum}_{4} \bigl
\{ H_{q}\bigl(\widetilde {%
\beta}_{j}(y_{1})
\bigr),\ldots,H_{q}\bigl(\widetilde{\beta}_{j}(y_{4})
\bigr) \bigr\} \,dy_{1}\cdots  dy_{4}.
\end{eqnarray*}
We now need to provide a bound on the cumulant inside the integral; to this
aim, we need to recall the \emph{diagram formula} (see, e.g.,
\cite%
{pectaq}, Chapter~7 or \cite{marpecbook}, Proposition~4.15 for further
details). In particular, fix a set of integers $\alpha_{1},\ldots
,\alpha
_{p}$; a~\emph{diagram} is a graph with $\alpha_{1}$ vertices labeled
by $%
1 $, $\alpha_{2}$ vertices labeled by $2, \ldots, \alpha_{p}$ vertices
labeled by $p$, such that each vertex has degree $1$. We can view the
vertices as belonging to $p$ different rows; the edges may connect only
vertices with different labels, that is, there are no
(``flat'')
edges connecting two vertices on the same row. The set of such diagrams that
are connected (i.e., such that it is not possible to partition the rows into
two subsets $A$ and $B$ such that no edge connect a vertex in $A$ with a
vertex in $B$) is denoted by $\Gamma_{c}(\alpha_{1},\ldots,\alpha_{p})$.
Given a diagram $\gamma\in\Gamma_{c}$, $\eta_{ik}(\gamma)$ is the
number of edges between the vertices labeled by $i$ and the vertices
labeled by $k$ in $\gamma$. The \emph{diagram formula} for Hermite
polynomials states the following; let $(Z_{1},\ldots,Z_{p})$ be a centered
Gaussian vector whose components have unit variance, and let $%
H_{l_{1}},\ldots,H_{l_{p}}$ be Hermite polynomials of degrees
$l_{1},\ldots
,l_{p}\ (\geq1)$, respectively. Then
\[
{\mathrm{cum}}\bigl(H_{l_{1}}(Z_{1}),\ldots,H_{l_{p}}(Z_{p})
\bigr)=\sum_{\gamma\in\Gamma
_{c}(l_{1},\ldots,l_{p})}\prod_{1\leq i\leq j\leq p}
\bigl\{\mathbb{E}%
[Z_{i}Z_{j}]\bigr
\}^{\eta_{ij}(\gamma)}.
\]
For a proof, see \cite{pectaq}, Section~7.3. A simple application in
our case
then yields
\begin{eqnarray}\label{chiarati}
&&\mathrm{cum}_{4} \bigl\{ H_{q}\bigl(\widetilde{
\beta}_{j}(y_{1})\bigr),\ldots,H_{q}\bigl(
\widetilde {%
\beta}_{j}(y_{4})\bigr) \bigr\}\nonumber
\\
&&\qquad=\sum_{\gamma\in\Gamma_{c}(q,q,q,q)}\prod_{1\leq s\leq t\leq4}
\bigl\{ \mathbb{E%
}\bigl[\widetilde{\beta}_{j}(y_{s})
\widetilde{\beta}_{j}(y_{t})\bigr]\bigr\}^{\eta
_{{st}}(\gamma)}
\nonumber
\\[-8pt]
\\[-8pt]
\nonumber
&&\qquad\leq\sum_{\gamma\in\Gamma_{c}(q,q,q,q)}\bigl\llvert \rho
_{j}(y_{1},y_{2})\bigr\rrvert ^{\eta_{12}(\gamma)}
\bigl\llvert \rho _{j}(y_{2},y_{3})\bigr\rrvert
^{\eta_{23}(\gamma)}\bigl\llvert \rho _{j}(y_{3},y_{4})
\bigr\rrvert ^{\eta_{34}(\gamma)}
\\
&&\hspace*{51pt}\qquad\quad{}\times\bigl\llvert \rho_{j}(y_{4},y_{1})\bigr
\rrvert ^{\eta
_{41}(\gamma
)}\bigl\llvert \rho_{j}(y_{1},y_{3})
\bigr\rrvert ^{\eta_{13}(\gamma
)}\bigl\llvert \rho_{j}(y_{2},y_{4})
\bigr\rrvert ^{\eta_{24}(\gamma)}, \nonumber
\end{eqnarray}
where
\[
\rho_{j}(y_{1},y_{2})=\frac{\sum_{\mathbb{\ell}}b^{2}({\mathbb
{\ell}%
}/{B^{j}})({(2\mathbb{\ell}+1)}/{(4\pi)})C_{\mathbb{\ell}}P_{\mathbb
{\ell}%
}(y_{1},y_{2})}{\sum_{\mathbb{\ell}}b^{2}({\mathbb{\ell}}/{B^{j}})
({(2\mathbb{\ell}+1)}/{(4\pi)})C_{\mathbb{\ell}}}\leq
\frac
{C_{M}}{ \{
1+B^{j}d(y_{1},y_{2}) \} ^{M}},
\]
in view of (\ref{ConditionB}) and the decorrelation inequality provided by
\cite{bkmpAoS}; see also \cite{spalan,mayeli}. Note that in our
circumstances, the total number of ``edges'' satisfies
\[
\sum_{t=1}^{4}\eta_{{st}}(
\gamma)=q\qquad \mbox{for all }s=1,\ldots,4\quad\mbox{and}\quad 
\sum
_{1<s<t\leq4}\eta_{{st}}(\gamma)=2q.
\]
It is simple to see that for any $\gamma\in\Gamma_{c}(q,q,q,q)$ and any
given $s$, there must exist two distinct indexes $t,t^{\prime}$ such
that $%
\eta_{{st}}(\gamma),\eta_{st^{\prime}}(\gamma)>0$. Indeed, assume by
contradiction that this is not the case for some $s$; then there must
exist $%
t\neq s$ such that such that $\eta_{{st}}(\gamma)=q$, and hence $\eta
_{s^{\prime}t}(\gamma)=0$ for all $s\neq s^{\prime}$. It follows
that $%
\gamma$ cannot be connected, yielding the desired contradiction.
Hence, up
to a relabeling of the indices there must necessarily exist a
``spanning cycle,'' that is, a sequence
\[
\eta_{12}(\gamma),\eta_{23}(\gamma),\eta_{34}(
\gamma),\eta _{41}(\gamma)>0,
\]
where the inequality is strict. Since the correlations are bounded by unity,
it follows that
\begin{eqnarray*}
&&\bigl\llvert \rho_{j}(y_{1},y_{2})\bigr\rrvert
^{\eta_{12}(\gamma
)}\bigl\llvert \rho_{j}(y_{2},y_{3})
\bigr\rrvert ^{\eta_{23}(\gamma
)}\bigl\llvert \rho_{j}(y_{3},y_{4})
\bigr\rrvert ^{\eta_{34}(\gamma)}
\\
&&\quad{}\times\bigl\llvert \rho_{j}(y_{4},y_{1})\bigr
\rrvert ^{\eta
_{41}(\gamma
)}\bigl\llvert \rho_{j}(y_{1},y_{3})
\bigr\rrvert ^{\eta_{13}(\gamma
)}\bigl\llvert \rho_{j}(y_{2},y_{4})
\bigr\rrvert ^{\eta_{24}(\gamma)}
\\
&&\qquad\leq\bigl\llvert \rho_{j}(y_{1},y_{2})\bigr
\rrvert \bigl\llvert \rho _{j}(y_{2},y_{3})\bigr
\rrvert \bigl\llvert \rho_{j}(y_{3},y_{4})\bigr
\rrvert \bigl\llvert \rho_{j}(y_{4},y_{1})\bigr
\rrvert .
\end{eqnarray*}
Therefore, writing $C(q)$ as the cardinality of set $\Gamma_{c}(q,q,q,q)$,
which is the set of all connected graphs of a given order, we get
\begin{eqnarray*}
&&\mathrm{cum}_{4} \bigl\{ H_{q}\bigl(\widetilde{\beta
}_{j}(y_{1})\bigr),\ldots,H_{q}\bigl(
\widetilde{%
\beta}_{j}(y_{4})\bigr) \bigr\}
\\
&&\qquad\leq\sharp \bigl\{ \Gamma_{c}(q,q,q,q) \bigr\} \times\bigl\llvert
\rho _{j}(y_{1},y_{2})\bigr\rrvert \bigl\llvert
\rho_{j}(y_{2},y_{3})\bigr\rrvert \bigl\llvert
\rho_{j}(y_{3},y_{4})\bigr\rrvert \bigl\llvert
\rho _{j}(y_{4},y_{1})\bigr\rrvert
\\
&&\qquad=C(q)\times\bigl\llvert \rho_{j}(y_{1},y_{2})
\bigr\rrvert \bigl\llvert \rho _{j}(y_{2},y_{3})
\bigr\rrvert \bigl\llvert \rho_{j}(y_{3},y_{4})
\bigr\rrvert \bigl\llvert \rho_{j}(y_{4},y_{1})
\bigr\rrvert ,
\end{eqnarray*}
%
Thus, we have
\begin{eqnarray*}
\mathrm{cum}_{4}\bigl\{g_{j;q}(x)\bigr\} &\leq&C(q)\int
_{ \{ S^{2} \} ^{\otimes
4}}\bigl\llvert K\bigl( \langle x,y_{1} \rangle
\bigr)\cdots  K\bigl( \langle x,y_{4} \rangle\bigr)\bigr\rrvert \bigl\llvert
\rho _{j}(y_{1},y_{2})\bigr\rrvert
\\
&&\hspace*{54pt}{}\times\bigl\llvert \rho_{j}(y_{2},y_{3})\bigr
\rrvert \bigl\llvert \rho _{j}(y_{3},y_{4})\bigr
\rrvert \bigl\llvert \rho_{j}(y_{4},y_{1})\bigr
\rrvert\, dy_{1}\cdots dy_{4}.
\end{eqnarray*}
Now standard computations yield
\begin{eqnarray*}
\int_{S^{2}}\bigl\llvert \rho(y_{1},y_{2})
\bigr\rrvert\, dy_{2}&\leq&\int_{S^{2}}%
\frac{C_{M}}{ \{ 1+B^{j}d(y_{1},y_{2}) \} ^{M}}\,dy_{2}
\\
&\leq&\int_{y_{2}\dvtx d(y_{1},y_{2})\leq B^{-j}}\frac{C_{M}}{ \{
1+B^{j}d(y_{1},y_{2}) \} ^{M}}\,dy_{2}\\
&&{}+\int
_{y_{2}\dvtx d(y_{1},y_{2})\geq
B^{-j}}\frac{C_{M}}{ \{ 1+B^{j}d(y_{1},y_{2}) \}
^{M}}\,dy_{2}\\
&\leq&
CB^{-2j}.
\end{eqnarray*}
Hence,
\begin{eqnarray*}
&&\int_{ \{ S^{2} \} ^{\otimes4}}\bigl\llvert \rho (y_{1},y_{2})
\bigr\rrvert \bigl\llvert \rho(y_{2},y_{3})\bigr\rrvert \bigl
\llvert \rho(y_{3},y_{4})\bigr\rrvert \bigl\llvert
\rho(y_{4},y_{1})\bigr\rrvert\, dy_{1}\cdots dy_{4}
\\
&&\qquad\leq\int_{ \{ S^{2} \} ^{\otimes4}}\bigl\llvert \rho (y_{1},y_{2})
\bigr\rrvert \bigl\llvert \rho(y_{2},y_{3})\bigr\rrvert \bigl
\llvert \rho(y_{3},y_{4})\bigr\rrvert\, dy_{1}\cdots dy_{4}
\leq CB^{-6j}
\end{eqnarray*}
and
\[
\mathrm{cum}_{4} \bigl\{ \widetilde{g}_{j;q}(x) \bigr\} =O
\bigl(B^{-2j}\bigr),
\]
entailing that for every fixed $x\in S^{2}$,
\[
d_\mathrm{TV} \bigl( \widetilde{g}_{j;q}(x),N(0,1) \bigr) =O
\bigl(B^{-2j}\bigr),
\]
and hence the univariate central limit theorem, as claimed. The proof
in the
multivariate case is analogous, and hence omitted for the sake of brevity.
\end{pf}

\subsection{Tightness}
\label{tightness}

We now focus on asymptotic tightness for both sequences $ \{
g_{j;q} \} $ and $ \{ f_{j;q} \}$. We shall exploit the
following criterion from \cite{Kallenberg}.

\begin{proposition}
[(\cite{Kallenberg})] Let $g_{j}\dvtx M\rightarrow D$ be a sequence of stochastic
processes, where $M$ is compact and $D$ is complete and separable. Assume
that the finite-dimensional distributions of $g_j$ converge to the
those of $%
g$, and that (tightness)
\[
\lim_{h\rightarrow0}\lim\sup_{j\rightarrow\infty}
\mathbb{E}%
\Bigl(\sup_{d(x,y)\leq h}\bigl\llvert
g_{j}(x)-g_{j}(y)\bigr\rrvert \wedge 1\Bigr)=0.
\]
Then $g_{j}\Rightarrow g$.
\end{proposition}

We are hence able to establish the following.

\begin{lemma}
For every $q\in\mathbb{N}$, the sequences $ \{ \widetilde{g}%
_{j;q} \} $ and $ \{ f_{j;q} \} $ are tight.
\end{lemma}

\begin{pf}
Write $ \{ a_{\mathbb{\ell}m}(f_{j;q}) \} $ for the spherical
harmonic coefficients of the fields $ \{ f_{j;q} \} $. For any
$%
x_{1},x_{2}\in S^{2}$, we have
\begin{eqnarray*}
&&\mathbb{E} \Bigl\{ \sup_{d(x_{1},x_{2})\leq\delta}\bigl\llvert
f_{j;q}(x_{1})-f_{j;q}(x_{2})\bigr
\rrvert \Bigr\} \\
&&\qquad=\mathbb{E} \biggl\{ \sup_{d(x_{1},x_{2})\leq\delta}\biggl\llvert
\sum_{\mathbb{\ell
}m}a_{\mathbb{%
\ell}m}(f_{j;q}) \bigl
\{ Y_{\mathbb{\ell}m}(x_{1})-Y_{\mathbb{\ell}%
m}(x_{2}) \bigr\}
\biggr\rrvert \biggr\}
\\
&&\qquad\leq\sum_{\mathbb{\ell}m} \bigl\{ \mathbb{E}\bigl\llvert
a_{\mathbb
{\ell}%
m}(f_{j;q})\bigr\rrvert \bigr\} \Bigl\{ \sup
_{d(x_{1},x_{2})\leq\delta
}\bigl\llvert \bigl\{ Y_{\mathbb{\ell}m}(x_{1})-Y_{\mathbb{\ell}%
m}(x_{2})
\bigr\} \bigr\rrvert \Bigr\}.
\end{eqnarray*}
Now
\[
\sup_{d(x_{1},x_{2})\leq\delta}\bigl\llvert \bigl\{ Y_{\mathbb{\ell}%
m}(x_{1})-Y_{\mathbb{\ell}m}(x_{2})
\bigr\} \bigr\rrvert \leq c\ell ^{2}\delta
\]
and
\[
\sum_{\mathbb{\ell}m} \bigl\{ \mathbb{E}\bigl\llvert
a_{\mathbb{\ell}%
m}(f_{j;q})\bigr\rrvert \bigr\} \leq\sum
_{\mathbb{\ell}m}\sqrt{ \bigl\{ \mathbb{E}\bigl\llvert
a_{\mathbb{\ell}m}(f_{j;q})\bigr\rrvert ^{2} \bigr\}
}%
=\sum_{\mathbb{\ell}}(2\mathbb{\ell}+1)
\sqrt{C_{\mathbb{\ell}}(f_{j;q})}
\]
and because $K(\cdot)$ is compactly supported in harmonic space (and hence,
again, a~finite-order polynomial)
\[
\leq \Biggl\{ \sum_{\mathbb{\ell}}^{L_{K}}(2\mathbb{
\ell}+1) \Biggr\} ^{1/2}%
\sqrt{\sum
_{\mathbb{\ell}}^{L_{K}}(2\mathbb{\ell}+1)C_{\mathbb{\ell}%
}(f_{j;q})}
\leq O(L_{K}),
\]
whence
\[
\mathbb{E} \Bigl\{ \sup_{d(x_{1},x_{2})\leq\delta}\bigl\llvert
f_{j;q}(x_{1})-f_{j;q}(x_{2})\bigr
\rrvert \Bigr\} \leq CL_{K}^{3}\delta ,
\]
for some $C>0$, uniformly over $j$, and thus the result follows [once again,
recall that $L_{K}$ is fixed by assumption (\ref{kernexp})]. The proof
for $%
 \{ \widetilde{g}_{j;q} \} $ is analogous.
\end{pf}

\subsection{Asymptotic proximity of distributions}

Our discussion above shows that the finite-dimensional distributions of the
non-Gaussian sequence of random fields $ \{ \widetilde
{g}_{j;q} \} $
converge to those of the Gaussian sequence $ \{ f_{j;q} \} $
as $j$
tends to infinity; moreover, both sequences are tight. However, the
finite-dimensional distributions of neither processes converge to a
well-defined limit. In view of this situation, we need a broader notion of
convergence than the one envisaged in standard treatment such as \cite%
{Billingsley}; this extended form of convergence is provided by the notion
of \emph{Asymptotic Proximity}, or \emph{Merging}, of distributions, as
discussed, for instance, by \cite{davydov,aristotilediaconis,dudley} and others.

\begin{definition}[(Asymptotic proximity of distribution \cite{davydov,aristotilediaconis,dudley})]
Let $g_{n},f_{n}$ be two sequences of
random elements in some metric space $(X,\rho)$, possibly defined on two
different probability spaces. We say that the laws of $g_{n},f_{n}$ are
\emph{asymptotically merging}, or $\emph{asymptotically}$ $\emph{proximal}$,
(denoted as $g_{n}\Rightarrow f_{n}$) if and only if as $n\rightarrow
\infty
$
\[
\bigl\llvert \mathbb{E}h(g_{n})-\mathbb{E}h(f_{n})\bigr
\rrvert \rightarrow 0%
,
\]
for all continuous and bounded functionals $h\in\mathcal
{C}_{b}(X,\mathbb{R}%
)$.
\end{definition}

In view of the results provided in the previous subsection, it is immediate
to establish that the sequences $ \{ \widetilde{g}_{j;q} \}
, \{ f_{j;q} \} $ are proximal. Indeed,

\begin{theorem}
As $j\rightarrow\infty$
\[
\widetilde{g}_{j;q}\Longrightarrow f_{j;q},
\]
that is, for all $h=h\dvtx \mathcal{C}(S^{2},\mathbb{R})\rightarrow\mathbb
{R}$, $h$
continuous and bounded, we have
\[
\bigl\llvert \mathbb{E}h(\widetilde{g}_{j;q})-\mathbb{E}h(f_{j;q})
\bigr\rrvert \rightarrow0.
\]
\end{theorem}

\begin{pf}
Applying to our circumstances the characterization of asymptotic proximity
provided by \cite{davydov}, we find that the sequences $ \{
\widetilde{g}%
_{j;q} \} , \{ f_{j;q} \} $ are asymptotically proximal
if and
only if they are both tight and their finite-dimensional distribution
converge, that is, for all $n\geq1$, $x_{1},\ldots,x_{n}\in K$, we have that
\[
\bigl\llvert \Pr \bigl\{ \bigl(\widetilde{g}_{j;q}(x_{1}),
\ldots,\widetilde{g}%
_{j;q}(x_{n})\bigr)\in A \bigr
\} -\Pr \bigl\{ \bigl(f_{j;q}(x_{1}),\ldots,f_{j;q}(x_{n})
\bigr)\in A \bigr\} \bigr\rrvert \rightarrow0
\]
for all $A\in\mathcal{B}(\mathbb{R}^{n})$. Now convergence of the
finite-dimensional distributions was established in Section~\ref{fdd},
while tightness was established in Section~\ref{tightness}; thus the
result follows immediately.
\end{pf}

As a simple application of the asymptotic proximity result, we have
\[
\mathbb{E} \biggl\{ \frac{\sup\widetilde{g}_{j;q}}{1+\sup\widetilde
{g}_{j;q}}%
 \biggr\} \rightarrow
\mathbb{E} \biggl\{ \frac{\sup f_{j;q}}{1+\sup
f_{j;q}}%
 \biggr\}.
\]
It should be noted that asymptotically proximal sequences do not enjoy all
the same properties as in the standard weak convergence case. For instance,
it is known that the Portmanteau lemma does not hold in general, that
is, it is
not true that, for every Borel set such that $\Pr \{ g_{n}\in
\partial
A \} =\Pr \{ f_{n}\in\partial A \} =0$, we have
\[
\bigl\llvert \Pr \{ g_{n}\in A \} -\Pr \{ f_{n}\in A \}
\bigr\rrvert \rightarrow0;
\]
as a counterexample, it is enough to consider the sequences $f_{n}=-n^{-1}$
and $g_{n}=n^{-1}$. However, it is indeed possible to obtain more stringent
characterizations when the subsequences are asymptotically Gaussian. We have
the following.

\begin{proposition}
\label{gastro} For every $A\in\mathcal{B(}\mathbb{R)}$, we have that
\[
\Bigl\llvert \Pr \Bigl\{ \sup_{x\in S^{2}}\widetilde{g}_{j;q}(x)
\in A \Bigr\} -\Pr \Bigl\{ \sup_{x\in S^{2}}f_{j;q}\in A
\Bigr\} \Bigr\rrvert \rightarrow0%
.
\]
\end{proposition}

\begin{pf}
We shall argue again by contradiction. Assume that there exists a
subsequence $j_{n}^{\prime}$ such that for some $\varepsilon>0$
%
\begin{equation}
\Bigl\llvert \Pr \Bigl\{ \sup_{x\in S^{2}}\widetilde{g}_{j_{n}^{\prime
};q}(x)
\in A \Bigr\} -\Pr \Bigl\{ \sup_{x\in S^{2}}f_{j_{n}^{\prime
};q}\in A
\Bigr\} \Bigr\rrvert >\varepsilon. \label{igor}
\end{equation}
By relative compactness, there exists a subsequence $j_{n}^{\prime
\prime}$
and a limiting process $g_{\infty;q}$ such that
\[
\Bigl\llvert \Pr \Bigl\{ \sup_{x\in S^{2}}\widetilde{g}_{j_{n}^{\prime
\prime
};q}(x)
\in A \Bigr\} -\Pr \Bigl\{ \sup_{x\in S^{2}}\widetilde{g}_{\infty
;q}
\in A \Bigr\} \Bigr\rrvert \rightarrow0.
\]
Likewise, consider $ \{ j_{n}^{\prime\prime\prime} \}
\subset
 \{ j_{n}^{\prime\prime} \} $; again by relative compactness
there exist $f_{\infty;q}$ such that $f_{j_{n}^{\prime\prime\prime
};q}\Rightarrow f_{\infty;q}$, and hence
\[
\Bigl\llvert \Pr \Bigl\{ \sup_{x\in S^{2}}f_{j_{n}^{\prime\prime\prime
};q}(x)\in A
\Bigr\} -\Pr \Bigl\{ \sup_{x\in S^{2}}f_{\infty;q}\in A \Bigr\}
\Bigr\rrvert \rightarrow0.
\]
Note that $f_{\infty;q},\widetilde{g}_{\infty;q}$ are isotropic and
continuous Gaussian random fields; indeed for $\widetilde{g}_{\infty
;q}$ it
suffices to recall that the finite-dimensional distributions of $ \{
\widetilde{g}_{j;q} \} $ are asymptotically Gaussian (Section~\ref{fdd}), so
if a weak limit exists it must be Gaussian as well. Hence, the
supremum is necessarily a continuous random variable, and no problems with
nonzero boundary probabilities can arise. Also, the finite-dimensional
distributions are a determining class, whence the two Gaussian
processes $%
f_{\infty;q},\widetilde{g}_{\infty;q}$ must necessarily have the same
distribution. Hence,
\[
\Bigl\llvert \Pr \Bigl\{ \sup_{x\in S^{2}}f_{j_{n}^{\prime\prime\prime
};q}(x)\in A
\Bigr\} -\Pr \Bigl\{ \sup_{x\in S^{2}}\widetilde{g}%
_{j_{n}^{\prime\prime\prime};q}(x)
\in A \Bigr\} \Bigr\rrvert \rightarrow0%
,
\]
yielding a contradiction with \eqref{igor}.
\end{pf}

This result immediately suggests two alternative ways to achieve the
ultimate goal of this paper, for example, the evaluation of excursion
probabilities
on the non-Gaussian sequence of random fields $ \{ g_{j;q} \}
$. On
one hand, it follows immediately that these probabilities may be evaluated
by simulations, by simply sampling realizations of a Gaussian field with
known angular power spectrum; for $q=2$, for example, $f_{j;q}$ is
simply a
Gaussian process with angular power spectrum given by~(\ref{powspe}). There
exist now very efficient techniques, based on packages such as HealPix
\cite%
{HEALPIX}, for the numerical simulation of Gaussian fields with a given
power spectra; here the only burdensome step can be the numerical evaluation
of expressions like~(\ref{powspe}), but this is in any case much faster and
simpler than the Monte Carlo evaluation of smoothed non-Gaussian fields.
Therefore, our result has an immediate applied relevance.

One can try, however, to be more ambitious than this, and verify whether
these excursion probabilities can indeed be evaluated analytically, rather
than by Gaussian simulations. This is in fact the purpose of the next, and
final, section.

\section{Asymptotics for the excursion probabilities}
\label{sec:application}

The purpose of this final section is to show how the previous weak
convergence results allow for very neat characterizations of excursion
probabilities, even in non-Gaussian circumstances. In particular, our main
result is the following.

\begin{theorem}
\label{thm:exc:prob} There exists constants $\alpha>1$ and $\mu^{+}>0$
such that, for \mbox{$u>\mu^{+}$}
\begin{eqnarray*}
&&\limsup_{j\rightarrow\infty}\Bigl\llvert \Pr \Bigl\{ \sup
_{x\in
S^{2}}\tilde{%
g}_{j;q}(x)>u \Bigr\} - \bigl
\{ 2\bigl(1-\Phi(u)\bigr)+2u\phi(u)\lambda _{j;q} \bigr\} \Bigr\rrvert
\\
&&\qquad\leq\exp \biggl( -\frac{\alpha
u^{2}}{2} \biggr) %
,
\end{eqnarray*}
where [see (\ref{porteaperte2})]
%
\begin{equation}
\lambda_{j;q}=\frac{\sum_{\mathbb{\ell}=1}^{L}({(2\mathbb{\ell}+1)}/{(4\pi)})C_{\mathbb{\ell};j,q}P_{\mathbb{\ell}}^{\prime}(1)}{\sum_{\mathbb{%
\ell}=1}^{L}({(2\mathbb{\ell}+1)}/{(4\pi)})C_{\mathbb{\ell};j,q}}. \label{porteaperte}
\end{equation}
\end{theorem}

\begin{pf}
Note that
%
\begin{eqnarray}
\label{eqn:1} && \Bigl\llvert \Pr \Bigl\{ \sup_{x\in S^{2}}
\tilde{g}_{j;q}(x)>u \Bigr\} - \bigl\{ 2\bigl(1-\Phi(u)\bigr)+2u\phi(u)
\lambda_{j;q} \bigr\} \Bigr\rrvert
\nonumber
\\
&&\qquad\leq\Bigl\llvert \Pr \Bigl\{ \sup_{x\in S^{2}}\tilde
{g}_{j;q}(x)>u \Bigr\} -\Pr \Bigl\{ \sup_{x\in S^{2}}
\tilde{f}_{j;q}(x)>u \Bigr\} \Bigr\rrvert
\\
&&\qquad\quad{} + \Bigl\llvert \Pr \Bigl\{ \sup_{x\in S^{2}}\tilde{f}_{j;q}(x)>u
\Bigr\} - \bigl\{ 2\bigl(1-\Phi(u)\bigr)+2u\phi(u)\lambda_{j;q} \bigr\}
\Bigr\rrvert ,\nonumber
\end{eqnarray}
where $\tilde{f}_{j;q}$ is as defined in the previous section. Observe that
by Proposition~\ref{gastro} the first part of the right-hand side of the
above inequality converges to $0$, therefore, we need only prove the
required estimate for the second part of the right-hand side.

We shall mainly exploit Theorem~14.3.3 of \cite{RFG}, with some
modifications to adapt it to our needs. For each $x_{0}\in S^{2}$, let us
define the corresponding pivoted random field as
%
\begin{eqnarray}
\label{eqn:f-hat} \widehat{f}_{j;q}^{x_{0}}(x)
&=&
\frac{1}{1-\rho(x,x_{0})} \biggl\{%
f_{j;q}(x)-\rho(x,x_{0})f_{j;q}(x_{0})
\nonumber
\\
&&\hspace*{28pt}\qquad\quad{}-\operatorname{Cov}\biggl(f_{j;q}(x),\frac{\partial}{\partial\vartheta
}f_{j;q}(x_{0})
\biggr)\nonumber\\
&&\hspace*{72pt}{}\times \operatorname{Var}\biggl(%
\frac{\partial}{\partial\vartheta}f_{j;q}(x)\biggr)
\frac{\partial
}{\partial
\vartheta}f_{j;q}(x)
\\
&&\hspace*{28pt}\qquad\quad{}-\operatorname{Cov}\biggl(f_{j;q}(x),\frac{\partial}{\sin\vartheta\partial\phi}%
f_{j;q}(x_{0})
\biggr)\nonumber\\
&&\hspace*{72pt}{}\times \operatorname{Var}\biggl(\frac{\partial}{\sin\vartheta\partial\phi
}f_{j;q}(x)\biggr)%
\frac{\partial}{\sin\vartheta\partial\phi}f_{j;q}(x) \biggr\} ,
\nonumber
\end{eqnarray}
where $\rho(x,x_{0})=\mathbb{E} ( f_{j;q}(x)f_{j;q}(x_{0}) ) $.
Next define
\[
\mu_{j}^{+}=\sup_{x_{0}}\mathbb{E} \Bigl(
\sup_{x\neq x_{0}}\widehat{f}%
_{j;q}^{x_{0}}(x)
\Bigr)
\]
and
\[
\sigma_{j}^{2}=\sup_{x_{0}}\sup
_{x\neq x_{0}}\operatorname{Var}\bigl(\widehat{f}%
_{j;q}^{x_{0}}(x)
\bigr).
\]
Then from page 371 of \cite{RFG}, we know that for $u\geq\mu_{j}^{+}$
%
\begin{eqnarray}\label{eqn:proposition}
&&\Bigl\llvert \Pr \Bigl\{ \sup_{x\in S^{2}}f_{j;q}(x)>u
\Bigr\} -\mathbb {E}%
\mathcal{L}_{0}\bigl(A_{u}
\bigl(f_{j;q},S^{2}\bigr)\bigr)\Bigr\rrvert
\nonumber
\\[-8pt]
\\[-8pt]
\nonumber
&&\qquad\leq K u e^{-{(u-\mu_{j}^{+})^{2}}/{2} ( 1+{1}/{(2\sigma_{j}^{2})} ) }\sum_{i=0}^{2}
\Bigl\{ \mathbb{E}\Bigl\llvert \det _{i} \Bigl( -\nabla^{2}f_{j;q}-f_{j;q}I_{2}
\bigr) \Bigr\rrvert ^{2} \Bigr\} ^{1/2},
\end{eqnarray}
where $I_{2}$ is the $2\times2$ identity matrix, $\det_{i}$ of a
matrix is
the sum over all the $i$-minors of the matrix under consideration, and $K$
is a constant not depending on~$j$. Note that the expression on page
371 of
\cite{RFG} also involves an integral over the parameter space with the
metric induced by the second-order spectral moment. However, under (\ref
{kernexp}) this integral is easily seen to be uniformly bounded with respect
to $j$, so that we can get rid of it by invoking the isotropy of the
field $%
f_{j;q}$, and absorbing the arising constant into $K$ upfront.

Our goal is to get a uniform bound for the right-hand side of %
\eqref{eqn:proposition}. Clearly, $\sum_{i=0}^{2}\mathbb{E}\llvert
\det_{i} ( -\nabla^{2}f_{j;q}-f_{j;q}I_{2} ) \rrvert
^{2}$ is
bounded above by a universal constant, largely because of the finite
expansion for the kernel $K(\cdot,\cdot)$ used to define the field $%
g_{j;q} $. Next, to get a uniform bound for $\mu_{j}^{+}$, we shall resort
to a Slepian inequality type of argument, and use the standard
techniques of
estimating the expected value of supremum of a Gaussian random field using
\emph{metric entropy}.

In particular, we shall prove Proposition~\ref{claim:lip} in the \hyperref[app]{Appendix}
that the assumed regularity conditions on the kernel $K$ ensure the
following:
%
\begin{equation}
\mathbb{E}\bigl(\widehat{f}_{j;q}^{x_{0}}(x_{2})-
\widehat{f}%
_{j;q}^{x_{0}}(x_{1})
\bigr)^{2}\leq c(L_{K},q)|x_{2}-x_{1}|.
\label{maccibound}
\end{equation}
Then using this uniform bound and a Slepian type of comparison
argument, we
get a uniform (over $j$) bound on the metric entropy corresponding to
various $\widehat{f}_{j;q}^{x_{0}}$, which in turn ensures that there exist
finite constants $\alpha>1$ and $\mu^{+}=\sup_{j}\mu_{j}^{+}<\infty$,
such that, for $u>\mu^{+}$,
%
\begin{equation}
\Bigl\llvert \Pr \Bigl\{ \sup_{x\in S^{2}}f_{j;q}(x)>u
\Bigr\} -\mathbb{E} 
\mathcal{L}_{0}\bigl(A_{u}
\bigl(f_{j;q},S^{2}\bigr)\bigr)\Bigr\rrvert \leq\exp\biggl(-
\frac{\alpha
u^{2}}{2}\biggr),\label{unif:exc:prob}
\end{equation}
uniformly over $j$, where
\[
\mathbb{E}\mathcal{L}_{0}\bigl(A_{u}
\bigl(f_{j;q},S^{2}\bigr)\bigr)=2\bigl(1-\Phi(u)\bigr)+2u\phi
(u)\lambda_{j;q},
\]
which proves the result.
\end{pf}

\begin{appendix}\label{app}
\section*{Appendix}

All of this section is devoted to the proof of the following proposition.

\begin{proposition}
\label{claim:lip} Under the assumption that the kernel $K$ appearing in the
definition of $\tilde{g}_{j;q}$ is of the form (\ref{kernexp}), the
field $%
\widehat{f}_{j;q}^{x_{0}}$ satisfies the following:
%
\begin{equation}
\mathbb{E}\bigl(\widehat{f}_{j;q}^{x_{0}}(x_{2})-
\widehat{f}%
_{j;q}^{x_{0}}(x_{1})
\bigr)^{2}\leq c(L_{K},q)|x_{2}-x_{1}|,
\label{eqn:claim:lip}
\end{equation}
where the constant $c(L_{K},q)$ depends on $q$ and $\mathbb{\ell}$, but
does not depend on $j$.
\end{proposition}

As a by-product of the proof, we shall also obtain a uniform upper
bound on $%
\sigma_{j}^{2}$. For notational simplicity and without loss of generality,
we take the coefficients $ \{ k_{i}\frac{2i+1}{4\pi} \} $ in %
\eqref{kernexp} to be identically equal to one.

Writing $\rho(x,y)=\operatorname{cov}(f_{j;q}(x),f_{j;q}(y))$, and $\partial
_{\phi
_{x}}$, $\partial_{\theta_{x}}$ as directional derivatives at $x$ in the
normalized spherical coordinate directions, we have
\begin{eqnarray*}
&&\operatorname{cov} \bigl( \widehat{f}_{j;q}^{x_{0}}(x_{1}),
\widehat{f}%
_{j;q}^{x_{0}}(x_{2}) \bigr)
\\
&&\qquad=\frac{1}{(1-\rho(x_{0},x_{1}))(1-\rho(x_{0},x_{2}))}\\
&&\qquad\quad{}\times \bigl(\rho (x_{1},x_{2})-
\rho(x_{0},x_{1})\rho(x_{0},x_{2})
\\
&&\hspace*{16pt}\quad\qquad{}-\operatorname{cov}\bigl(f_{j;q}(x_{1}),\partial_{\theta_{x_{0}}}f_{j;q}(x_{0})
\bigr)%
\operatorname{cov}\bigl(f_{j;q}(x_{2}),
\partial_{\theta_{x_{1}}}f_{j;q}(x_{1})\bigr)\\
&&\hspace*{24pt}\quad\qquad{}\times\operatorname{cov}\bigl(\partial_{\theta_{x_{1}}}f_{j;q}(x_{1}),
\partial_{\theta
_{x_{1}}}f_{j;q}(x_{1})\bigr)
\\
&&\hspace*{16pt}\quad\qquad{}-\operatorname{cov}\bigl(f_{j;q}(x_{1}),\partial_{\phi_{x_{0}}}f_{j;q}(x_{0})
\bigr)%
\operatorname{cov}\bigl(f_{j;q}(x_{2}),
\partial_{\phi_{x_{1}}}f_{j;q}(x_{1})\bigr)\\
&&\hspace*{24pt}\quad\qquad{}\times\operatorname
{cov}%
\bigl(\partial_{\phi_{x_{1}}}f_{j;q}(x_{1}),
\partial_{\phi
_{x_{1}}}f_{j;q}(x_{1})\bigr)
\\
&&\hspace*{16pt}\quad\qquad{}-\rho(x_{0},x_{1})\rho(x_{0},x_{2})+
\rho(x_{0},x_{1})\rho (x_{0},x_{2})
\rho(x_{0},x_{0})
\\
&&\hspace*{16pt}\qquad\quad{}+\rho(x_{0},x_{2})\operatorname{cov}\bigl(f_{j;q}(x_{1}),
\partial_{\theta
_{x_{0}}}f_{j;q}(x_{0})\bigr)\operatorname{cov}
\bigl(f_{j;q}(x_{0}),\partial_{\theta
_{x_{1}}}f_{j;q}(x_{1})
\bigr)\\
&&\hspace*{24pt}\quad\qquad{}\times\operatorname{cov}\bigl(\partial_{\theta
_{x_{1}}}f_{j;q}(x_{1}),
\partial_{\theta_{x_{1}}}f_{j;q}(x_{1})\bigr)
\\
&&\hspace*{16pt}\qquad\quad{}+\rho(x_{0},x_{2})\operatorname{cov}\bigl(f_{j;q}(x_{1}),
\partial_{\phi
_{x_{0}}}f_{j;q}(x_{0})\bigr)\operatorname{cov}
\bigl(f_{j;q}(x_{0}),\partial_{\phi
_{x_{1}}}f_{j;q}(x_{1})
\bigr)\\
&&\hspace*{24pt}\quad\qquad{}\times\operatorname{cov}\bigl(\partial_{\phi
_{x_{1}}}f_{j;q}(x_{1}),
\partial_{\phi_{x_{1}}}f_{j;q}(x_{1})\bigr)
\\
&&\hspace*{16pt}\qquad\quad{}-\operatorname{cov}\bigl(f_{j;q}(x_{2}),\partial_{\theta_{x_{0}}}f_{j;q}(x_{0})
\bigr)%
\operatorname{cov}\bigl(\partial_{\theta_{x_{2}}}f_{j;q}(x_{2}),f_{j;q}(x_{1})
\bigr)\\
&&\hspace*{24pt}\quad\qquad{}\times\operatorname{cov}\bigl(\partial_{\theta_{x_{2}}}f_{j;q}(x_{2}),
\partial_{\theta
_{x_{2}}}f_{j;q}(x_{2})\bigr)
\\
&&\hspace*{16pt}\qquad\quad{}+\rho(x_{0},x_{1})\operatorname{cov}\bigl(f_{j;q}(x_{2}),
\partial_{\theta
_{x_{0}}}f_{j;q}(x_{0})\bigr)\operatorname{cov}\bigl(
\partial_{\theta
_{x_{2}}}f_{j;q}(x_{2}),f_{j;q}(x_{0})
\bigr)\\
&&\hspace*{24pt}\quad\qquad{}\times\operatorname{cov}\bigl(\partial_{\theta
_{x_{2}}}f_{j;q}(x_{2}),
\partial_{\theta_{x_{2}}}f_{j;q}(x_{2})\bigr)
\\
&&\hspace*{16pt}\qquad\quad{}+ \bigl( \operatorname{var}\bigl(\partial_{\theta_{x_{1}}}f_{j;q}(x_{1})
\bigr) \bigr) ^{2}%
\operatorname{cov}\bigl(f_{j;q}(x_{1}),
\partial_{\theta_{x_{0}}}f_{j;q}(x_{0})\bigr)%
\\
&&\hspace*{24pt}\quad\qquad{}\times\operatorname{cov}\bigl(f_{j;q}(x_{2}),\partial_{\theta_{x_{0}}}f_{j;q}(x_{0})
\bigr)
\\
&&\hspace*{24pt}\quad\qquad{}\times\operatorname{cov}\bigl(\partial_{\theta_{x_{1}}}f_{j;q}(x_{1}),
\partial _{\theta_{x_{2}}}f_{j;q}(x_{2})\bigr)
\\
&&\hspace*{16pt}\qquad\quad{}+\operatorname{var}\bigl(\partial_{\theta_{x_{1}}}f_{j;q}(x_{1})
\bigr)\operatorname {var}\bigl(\partial _{\phi_{x_{2}}}f_{j;q}(x_{2})
\bigr)\operatorname{cov}\bigl(f_{j;q}(x_{1}),\partial_{\phi
_{x_{0}}}f_{j;q}(x_{0})
\bigr)
\\
&&\hspace*{24pt}\quad\qquad{}\times\operatorname{cov}\bigl(f_{j;q}(x_{2}),\partial_{\theta
_{x_{0}}}f_{j;q}(x_{0})
\bigr)%
\operatorname{cov}\bigl(\partial_{\theta_{x_{1}}}f_{j;q}(x_{1}),
\partial_{\phi
_{x_{2}}}f_{j;q}(x_{2})\bigr)
\\
&&\hspace*{16pt}\qquad\quad{}-\operatorname{cov}\bigl(f_{j;q}(x_{2}),\partial_{\phi_{x_{0}}}f_{j;q}(x_{0})
\bigr)%
\operatorname{var}\bigl(\partial_{\phi_{x_{2}}}f_{j;q}(x_{2})
\bigr)\\
&&\hspace*{24pt}\quad\qquad{}\times\operatorname{cov}\bigl(\partial _{\phi_{x_{2}}}f_{j;q}(x_{2}),f_{j;q}(x_{1})
\bigr)
\\
&&\hspace*{16pt}\qquad\quad{}+\rho(x_{0},x_{1})\operatorname{cov}\bigl(f_{j;q}(x_{2}),
\partial_{\phi
_{x_{0}}}f_{j;q}(x_{0})\bigr)\operatorname{var}\bigl(
\partial_{\phi
_{x_{2}}}f_{j;q}(x_{2})\bigr)\\
&&\hspace*{24pt}\quad\qquad{}\times
\operatorname{cov}\bigl(\partial_{\phi_{x_{2}}}f_{j;q}(x_{2})f_{j;q}(x_{0})
\bigr)
\\
&&\hspace*{16pt}\qquad\quad{}+\operatorname{var}\bigl(\partial_{\theta_{x_{1}}}f_{j;q}(x_{1})
\bigr)\operatorname {var}\bigl(\partial _{\phi_{x_{2}}}f_{j;q}(x_{2})
\bigr)\\
&&\hspace*{24pt}\quad\qquad{}\times\operatorname{cov}\bigl(f_{j;q}(x_{1}),\partial
_{\theta
_{x_{0}}}f_{j;q}(x_{0})\bigr)
\\
&&\hspace*{24pt}\quad\qquad{}\times\operatorname{cov}\bigl(f_{j;q}(x_{2}),\partial_{\phi
_{x_{0}}}f_{j;q}(x_{0})
\bigr)%
\operatorname{cov}\bigl(\partial_{\theta_{x_{1}}}f_{j;q}(x_{1}),
\partial_{\phi
_{x_{2}}}f_{j;q}(x_{2})\bigr)
\\
&&\hspace*{16pt}\qquad\quad{}+ \bigl( \operatorname{var}\bigl(\partial_{\phi_{x_{1}}}f_{j;q}(x_{1})
\bigr) \bigr) ^{2}%
\operatorname{cov}\bigl(f_{j;q}(x_{1}),
\partial_{\phi_{x_{0}}}f_{j;q}(x_{0})\bigr)\\
&&\hspace*{24pt}\quad\qquad{}\times\operatorname
{cov}%
\bigl(f_{j;q}(x_{2}),\partial_{\phi_{x_{0}}}f_{j;q}(x_{0})
\bigr)
\operatorname{cov}\bigl(\partial_{\phi_{x_{1}}}f_{j;q}(x_{1}),
\partial _{\phi
_{x_{2}}}f_{j;q}(x_{2})\bigr) \bigr).
\end{eqnarray*}

Note that $\rho(x_{1},x_{2})$ can be assumed to have $P_{l}(\langle
x_{1},x_{2}\rangle)$ as the leading polynomial (uniform over all $j$).
Then, taking $x_{1}=x_{2}$ in the above computation, and going through some
more (but simple) calculations, one can show that there exists a
constant $%
M>0$ such that $\operatorname{Var}(\widehat{f}_{j;q}^{x_{0}}(x))\leq M$ uniformly over
all $%
j$, which in turn, together with the assumption of isotropy, proves
that $%
\sigma^{2}_j\leq M^{\prime}$, for some $M^{\prime} <\infty$.

Next, to prove Proposition~\ref{claim:lip} we begin with
\begin{eqnarray*}
&&\mathbb{E} \bigl( \widehat{f}_{j;q}^{x_{0}}(x_{2})-
\widehat{f}%
_{j;q}^{x_{0}}(x_{1}) \bigr)
^{2}\\
&&\qquad=\operatorname{var}\bigl(\widehat{f}%
_{j;q}^{x_{0}}(x_{1})
\bigr)+\operatorname{var}\bigl(\widehat{f}_{j;q}^{x_{0}}(x_{2})
\bigr)-2%
\operatorname{cov} \bigl( \widehat{f}_{j;q}^{x_{0}}(x_{1}),
\widehat{f}%
_{j;q}^{x_{0}}(x_{2}) \bigr).
\end{eqnarray*}

We shall analyze each pair of the terms in the above expression separately.
Let us, for instance, consider (together) one of the, seemingly, more
involved term of the expression which is the last term of the
covariance and
the corresponding term in $\operatorname{var}(\widehat
{f}_{j;q}^{x_{0}}(x_{1}))$.\vspace*{1pt} At
the expense of introducing more notation, let us write $C_{\mathbb{\ell
}%
;\phi\phi}=\operatorname{var}(\partial_{\phi_{x}}f_{j;q}(x))$ (note that
due to
isotropy, the variance does not depend on the spatial point $x$), then the
difference between the last term of $\operatorname{Var}(\widehat{f}^{x_0}_{j;q}(x_1))$ and
the last term of $\operatorname{Cov}(\widehat{f}^{x_0}_{j;q}(x_1),\widehat{f}%
^{x_0}_{j;q}(x_2))$, can be written as, for all $x_{1},x_{2}\in (
B(x_{0},\varepsilon) ) ^{c}$ that is, outside a ball of size
$\varepsilon$
around the point $x_{0}$, we shall have
\begin{eqnarray*}
&&\frac{1}{(1-\rho(x_{0},x_{1}))^{2}(1-\rho(x_{0},x_{2}))}
\\
&&\qquad{}\times \bigl(C_{\mathbb{\ell};\phi\phi}^{3} \bigl( \operatorname{cov}%
\bigl(f_{j;q}(x_{1}),\partial_{\phi_{x_{0}}}f_{j;q}(x_{0})
\bigr) \bigr) ^{2}\bigl(1-\rho (x_{0},x_{2})
\bigr)\\
&&\hspace*{38pt}{}-C_{\mathbb{\ell};\phi\phi}^{2}\operatorname{cov}%
\bigl(f_{j;q}(x_{1}),
\partial_{\phi_{x_{0}}}f_{j;q}(x_{0})\bigr)
\\
&&\hspace*{48pt}{}\times\operatorname{cov}\bigl(f_{j;q}(x_{2}),\partial_{\phi_{x_{0}}}f_{j;q}(x_{0})
\bigr)\\
&&\hspace*{57pt}{}\times
\operatorname{cov}\bigl(\partial_{\phi_{x_{1}}}f(x_{1}),
\partial_{\phi
_{x_{2}}}f_{j;q}(x_{2})\bigr) \bigl(1-
\rho(x_{0},x_{1})\bigr) \bigr)
\\
&&\qquad=\frac{C_{\mathbb{\ell};\phi\phi}^{2}\partial_{\phi_{x_{0}}}\rho
(\langle x_{1},x_{0}\rangle)}{(1-\rho(x_{0},x_{1}))^{2}(1-\rho
(x_{0},x_{2}))}
\\
&&\qquad\quad{}\times \bigl(C_{\mathbb{\ell};\phi\phi}\bigl(1-\rho (x_{0},x_{2})
\bigr)\partial_{\phi_{x_{0}}}\rho(x_{1},x_{0})\\
&&\hspace*{48pt}{}-\bigl(1-
\rho (x_{0},x_{1})\bigr)\partial_{\phi_{x_{0}}}
\rho(x_{2},x_{0})\partial_{\phi
_{x_{1}}}
\partial_{\phi_{x_{2}}}\rho(x_{1},x_{2}) \bigr)
\\
&&\qquad=\frac{C_{\mathbb{\ell};\phi\phi}^{2}\partial_{\phi_{x_{0}}}\rho
(\langle x_{1},x_{0}\rangle)}{(1-\rho(x_{0},x_{1}))^{2}(1-\rho
(x_{0},x_{2}))}
\\
&&\qquad\quad{}\times \bigl( \bigl( \partial_{\phi_{x_{0}}}\rho (x_{1},x_{0})-
\partial_{\phi_{x_{0}}}\rho(x_{2},x_{0}) \bigr)
C_{\mathbb{%
\ell};\phi\phi}\bigl(1-\rho(x_{0},x_{2})\bigr)
\\
&&\hspace*{14pt}\qquad\quad{}+\partial_{\phi_{x_{0}}}\rho(x_{2},x_{0}) \bigl(
C_{\mathbb
{\ell}%
;\phi\phi}\bigl(1-\rho(x_{0},x_{2})\bigr)\\
&&\hspace*{121pt}{}-\bigl(1-
\rho(x_{0},x_{1})\bigr)\partial_{\phi
_{x_{1}}}
\partial_{\phi_{x_{2}}}\rho(x_{1},x_{2}) \bigr) \bigr).
\end{eqnarray*}

Recall that the covariance function $\rho$ does depend on $j$, but
since we
are assuming the kernel $K(x,y)$ to have finite expansion, thus the
corresponding Legendre polynomial expansion of $\rho(x_{1},x_{2})$ can be
assumed to have a $P_{\mathbb{\ell}}(\langle x_{1},x_{2}\rangle)$ (uniform
over $j$) which is the leading polynomial. Then, taking the modulus of the
above expression, and considering all $x_{1},x_{2}\in (
B(x_{0},\varepsilon) ) ^{c}$ that is, outside a ball of size
$\varepsilon$
around the point $x_{0}$, we shall have
\begin{eqnarray*}
&&\biggl\llvert \frac{C_{\mathbb{\ell};\phi\phi}^{2}\partial_{\phi
_{x_{0}}}P_{\mathbb{\ell}}(\langle x_{1},x_{0}\rangle)}{[1-P_{\mathbb
{\ell
}}(\langle x_{0},x_{1}\rangle)]^{2}[1-P_{\mathbb{\ell}}(\langle
x_{0},x_{2}\rangle)]}\biggr\rrvert
\\
&&\quad{}\times \bigl| \bigl( \bigl\{ \partial_{\phi_{x_{0}}}P_{\mathbb
{\ell}%
}\bigl(\langle
x_{1},x_{0}\rangle\bigr)-\partial_{\phi_{x_{0}}}P_{\mathbb{\ell
}%
}
\bigl(\langle x_{2},x_{0}\rangle\bigr) \bigr\}
C_{\mathbb{\ell};\phi\phi
}\bigl[1-P_{%
\mathbb{\ell}}\bigl(\langle x_{0},x_{2}
\rangle\bigr)\bigr]
\\
&&\hspace*{21pt}\quad{}+\partial_{\phi_{x_{0}}}P_{\mathbb{\ell}}\bigl(\langle x_{2},x_{0}
\rangle\bigr) \bigl\{ C_{\mathbb{\ell};\phi\phi}\bigl[1-P_{\mathbb
{\ell}%
}\bigl(\langle
x_{0},x_{2}\rangle\bigr)\bigr]\\
&&\hspace*{107pt}\quad{}-\bigl[1-P_{\mathbb{\ell}}
\bigl(\langle x_{0},x_{1}\rangle\bigr)\bigr]
\partial_{\phi_{x_{1}}}\partial_{\phi
_{x_{2}}}P_{%
\mathbb{\ell}}\bigl(\langle
x_{1},x_{2}\rangle\bigr) \bigr\} \bigr) \bigr|
\\
&&\qquad\leq\biggl\llvert \frac{C_{\mathbb{\ell};\phi\phi}^{2}P_{\mathbb
{\ell}%
}^{\prime}(\langle x_{1},x_{0}\rangle)}{(1-P_{\mathbb{\ell}}(\langle
x_{0},x_{1}\rangle))^{2}(1-P_{\mathbb{\ell}}(\langle
x_{0},x_{2}\rangle))}%
\biggr\rrvert
\\
&&\quad\qquad{}\times \bigl( \bigl| \bigl( P_{\mathbb{\ell}}^{\prime}
\bigl(\langle
x_{1},x_{0}\rangle\bigr) \bigl(-\sin\theta_{x_{1}}
\sin(\phi_{x_{1}}-\phi _{x_{0}})\bigr)\\
&&\hspace*{20pt}\qquad\quad{}-P_{\mathbb{\ell}}^{\prime}
\bigl(\langle x_{2},x_{0}\rangle \bigr) \bigl(-\sin
\theta_{x_{2}}\sin(\phi_{x_{2}}-\phi_{x_{0}})\bigr) \bigr)
\bigr|\cdot \varepsilon C_{\mathbb{\ell};\phi\phi}
\\
&&\hspace*{16pt}\qquad\quad{}+ \bigl|P_{\mathbb{\ell}}^{\prime}\bigl(\langle x_{2},x_{0}
\rangle \bigr) \bigl(-\sin\theta_{x_{2}}\sin(\phi_{x_{2}}-
\phi_{x_{0}})\bigr)\bigr |\\
&&\hspace*{61pt}{}\times  \bigl|%
 \bigl(C_{\mathbb{\ell};\phi\phi}
\bigl[1-P_{\mathbb{\ell}}\bigl(\langle x_{0},x_{2}\rangle
\bigr)\bigr]
-\bigl[1-P_{\mathbb{\ell}}\bigl(\langle x_{0},x_{1}
\rangle\bigr)\bigr]\\
&&\hspace*{80pt}{}\times\bigl\{P_{\mathbb{\ell
}%
}^{\prime\prime}\bigl(\langle
x_{1},x_{2}\rangle\bigr)\sin\theta_{x_{1}}\sin
\theta_{x_{2}}\sin^{2}(\phi_{x_{1}}-
\phi_{x_{2}})\\
&&\hspace*{114pt}\qquad\quad{}+P_{\mathbb{\ell}%
}^{\prime}\bigl(\langle
x_{1},x_{2}\rangle\bigr)\cos(\phi_{x_{1}}-\phi
_{x_{2}})\bigr\}%
 \bigr) \bigr| \bigr)
\\
&&\qquad\leq C_{\mathbb{\ell};\phi\phi}^{2}M(\varepsilon,\mathbb{\ell}) \\
&&\qquad\quad{}\times\bigl(
 \bigl\{\bigl|P_{\mathbb{\ell}}^{\prime}\bigl(\langle
x_{1},x_{0}\rangle\bigr)\bigr|\\
&&\hspace*{20pt}\quad\qquad{}\times \bigl|\bigl(\sin
\theta_{x_{2}}\sin(\phi_{x_{2}}-\phi_{x_{0}})-\sin\theta
_{x_{1}}\sin(\phi_{x_{1}}-\phi_{x_{0}})\bigr)\bigr|
\\
&&\hspace*{20pt}\qquad\quad{}+\bigl|\bigl(P_{\mathbb{\ell}}^{\prime}\bigl(\langle x_{2},x_{0}
\rangle \bigr)-P_{\mathbb{%
\ell}}^{\prime}\bigl(\langle x_{1},x_{0}
\rangle\bigr)\bigr)\bigr|\cdot\bigl|\sin\theta _{x_{2}}\sin(\phi_{x_{2}}-
\phi_{x_{0}})\bigr| \bigr\}\times\varepsilon C_{%
\mathbb{\ell};\phi\phi}
\\
&&\hspace*{20pt}\qquad\quad{}+M_{1}(\varepsilon,\mathbb{\ell})C_{\mathbb{\ell};\phi\phi
}\bigl|P_{\mathbb{%
\ell}}
\bigl(\langle x_{0},x_{1}\rangle\bigr)-P_{\mathbb{\ell}}
\bigl(\langle x_{0},x_{2}\rangle\bigr)\bigr|\\
&&\hspace*{20pt}\qquad\quad{}+M_{1}(
\varepsilon,\mathbb{\ell})\bigl|1-P_{\mathbb{\ell
}%
}\bigl(\langle x_{0},x_{1}
\rangle\bigr)\bigr|
\\
&&\hspace*{20pt}\qquad\quad{}\times\bigl|C_{\mathbb{\ell},\phi\phi}-P_{\mathbb{\ell}}^{\prime
\prime
}\bigl(\langle
x_{1},x_{2}\rangle\bigr)\sin\theta_{x_{1}}\sin\theta
_{x_{2}}\sin ^{2}(\phi_{x_{1}}-\phi_{x_{2}})\\
&&\hspace*{183pt}\qquad\quad{}-P_{\mathbb{\ell}}^{\prime}
\bigl(\langle x_{1},x_{2}\rangle\bigr)\cos(
\phi_{x_{1}}-\phi_{x_{2}})\bigr| \bigr)
\\
&&\qquad\leq C_{\mathbb{\ell},\phi\phi}^{2}M(\varepsilon,\mathbb{\ell})\\
&&\qquad\quad{}\times \bigl(
\varepsilon C_{\mathbb{\ell},\phi\phi}M_{2}(\mathbb{\ell},\varepsilon )
\\
&&\hspace*{14pt}\qquad\quad{}\times\bigl(%
|\sin\theta_{x_{2}}|\cdot\bigl|\sin(\phi_{x_{2}}-
\phi_{x_{0}})-\sin (\phi _{x_{1}}-\phi_{x_{0}})\bigr|
\\
&&\hspace*{30pt}\qquad\quad{}+\bigl|\sin(\phi_{x_{1}}-\phi_{x_{0}})\bigr|\cdot|\sin
\theta_{x_{2}}-\sin \theta_{x_{1}}|+M_{3}(\mathbb{
\ell},\varepsilon)|x_{2}-x_{1}| \bigr)
\\
&&\hspace*{16pt}\qquad\quad{}+M_{1}^{\prime}(\varepsilon,\mathbb{\ell})|x_{2}-x_{1}|+M_{1}^{\prime
\prime}(
\varepsilon,\mathbb{\ell})\cdot|\sin\theta_{x_{1}}\sin\theta
_{x_{2}}|\cdot\sin^{2}(\phi_{x_{1}}-
\phi_{x_{2}})
\\
&&\hspace*{76pt}\qquad\quad{}+M_{1}^{\prime\prime}(\varepsilon,\mathbb{\ell})\times\bigl|C_{\mathbb
{\ell}%
,\phi\phi}-P_{\mathbb{\ell}}^{\prime}
\bigl(\langle x_{1},x_{2}\rangle \bigr)\cos (
\phi_{x_{1}}-\phi_{x_{2}})\bigr| \bigr)
\\
&& \qquad\leq C_{l\phi\phi}^{2}M(\varepsilon,\mathbb{\ell})\\
&&\qquad\quad{}\times \bigl(\varepsilon
C_{%
\mathbb{\ell},\phi\phi}M_{2}(\varepsilon,\mathbb{\ell})M_{4}(\varepsilon
,%
\mathbb{\ell})\cdot\bigl|\sin(\phi_{x_{2}}-\phi_{x_{1}})-
\sin(\phi _{x_{1}}-\phi_{x_{1}})\bigr|
\\
&&\hspace*{16pt}\qquad\quad{}+M_{4}(\varepsilon,\mathbb{\ell})\cdot|\sin\theta_{x_{2}}-\sin
\theta _{x_{1}}|+M_{3}(\varepsilon,\mathbb{\ell})|x_{2}-x_{1}|\\
&&\qquad\quad\hspace*{16pt}{}+M_{1}^{\prime
}(
\varepsilon,\mathbb{\ell})|x_{2}-x_{1}|
\\
&&\hspace*{16pt}\qquad\quad{}+M_{1}^{\prime\prime\prime}(\varepsilon,\mathbb{\ell})\sin ^{2}(
\theta _{x_{2}}-\theta_{x_{1}})+M_{1}^{\prime\prime}(
\varepsilon,\mathbb{\ell }%
)\cdot\bigl|C_{\mathbb{\ell},\phi\phi}-P_{\mathbb{\ell}}^{\prime
}
\bigl(\langle x_{1},x_{2}\rangle\bigr)\bigr|
\\
&&\hspace*{144pt}\qquad\quad{}+M_{1}^{(iv)}(\varepsilon,\mathbb{\ell})\bigl|1-\cos(
\phi_{x_{2}}-\phi _{x_{1}})\bigr| \bigr).
\end{eqnarray*}

Now note that $C_{\mathbb{\ell},\phi\phi}$ is precisely equal to $P_{
\mathbb{\ell}}^{\prime}(1)$, which can be rewritten as $P_{\mathbb
{\ell}%
}^{\prime}(\langle x_{1},x_{1}\rangle)$. Replacing this in the last part
of the above expression, we get the following:
\begin{eqnarray*}
&&\biggl\llvert \frac{C_{\mathbb{\ell};\phi\phi}^{2}\partial_{\phi
_{x_{0}}}P_{\mathbb{\ell}}(\langle x_{1},x_{0}\rangle)}{[1-P_{\mathbb
{\ell
}}(\langle x_{0},x_{1}\rangle)]^{2}[1-P_{\mathbb{\ell}}(\langle
x_{0},x_{2}\rangle)]}\biggr\rrvert
\\
&&\quad{}\times \bigl| \bigl( \bigl\{ \partial_{\phi_{x_{0}}}P_{\mathbb
{\ell}%
}\bigl(\langle
x_{1},x_{0}\rangle\bigr)-\partial_{\phi_{x_{0}}}P_{\mathbb{\ell
}%
}
\bigl(\langle x_{2},x_{0}\rangle\bigr) \bigr\}
C_{\mathbb{\ell};\phi\phi
}\bigl[1-P_{%
\mathbb{\ell}}\bigl(\langle x_{0},x_{2}
\rangle\bigr)\bigr]
\\
&&\hspace*{19pt}\quad{}+\partial_{\phi_{x_{0}}}P_{\mathbb{\ell}}\bigl(\langle x_{2},x_{0}
\rangle\bigr) \bigl\{ C_{\mathbb{\ell};\phi\phi}\bigl[1-P_{\mathbb
{\ell}%
}\bigl(\langle
x_{0},x_{2}\rangle\bigr)\bigr]\\
&&\hspace*{117pt}{}-\bigl[1-P_{\mathbb{\ell}}
\bigl(\langle x_{0},x_{1}\rangle\bigr)\bigr]
\partial_{\phi_{x_{1}}}\partial_{\phi
_{x_{2}}}P_{%
\mathbb{\ell}}\bigl(\langle
x_{1},x_{2}\rangle\bigr) \bigr\} \bigr) \bigr|
\\
&&\qquad\leq C_{l\phi\phi}^{2}M(\varepsilon,\mathbb{\ell})\\
&&\qquad\quad{}\times \bigl(\varepsilon
C_{%
\mathbb{\ell},\phi\phi}M_{2}(\varepsilon,\mathbb{\ell})M_{4}(\varepsilon
,%
\mathbb{\ell})\cdot\bigl|\sin(\phi_{x_{2}}-\phi_{x_{1}})-
\sin(\phi _{x_{1}}-\phi_{x_{1}})\bigr|
\\
&&\hspace*{14pt}\qquad\quad{}+M_{4}(\varepsilon,\mathbb{\ell})\cdot|\sin\theta_{x_{2}}-\sin
\theta _{x_{1}}|+M_{3}(\varepsilon,\mathbb{\ell})|x_{2}-x_{1}|\\
&&\hspace*{14pt}\qquad\quad{}+M_{1}^{\prime
}(
\varepsilon,\mathbb{\ell})|x_{2}-x_{1}|
\\
&&\hspace*{14pt}\qquad\quad{}+M_{1}^{\prime\prime\prime}(\varepsilon,\mathbb{\ell})\sin ^{2}(
\theta _{x_{2}}-\theta_{x_{1}})+M_{1}^{\prime\prime}(
\varepsilon,\mathbb{\ell }%
)\cdot\bigl|P_{\mathbb{\ell}}^{\prime}\bigl(
\langle x_{1},x_{1}\rangle\bigr) -P_{\mathbb{%
\ell}}^{\prime}
\bigl(\langle x_{1},x_{2}\rangle\bigr)\bigr|
\\
&&\hspace*{166pt}\qquad\quad{}+M_{1}^{(iv)}(\varepsilon,\mathbb{\ell})\bigl|1-\cos(
\phi_{x_{2}}-\phi _{x_{1}})\bigr| \bigr)
\\
&&\qquad\leq c(\varepsilon,L_{K})|x_{1}-x_{2}|.
\end{eqnarray*}

By replicating these set of calculations for each pair of terms in
$\mathbb{E%
}(\widehat{f}_{j;q}^{x_{0}}(x_{2})-\widehat{f}_{j;q}^{x_{0}}(x_{1}))^{2}$,
we conclude that for every $x_{1},x_{2}\in B(x_{0},\varepsilon)$,
\[
\mathbb{E}\bigl(\widehat{f}_{j;q}^{x_{0}}(x_{2})-
\widehat{f}%
_{j;q}^{x_{0}}(x_{1})
\bigr)^{2}\leq c(\varepsilon,L_{K})|x_{2}-x_{1}|.
\]
Next, we wish to extend this to points inside the set $B(x_{0},\varepsilon
)\setminus\{x_{0}\}$, but the Lipschitz coefficient $c(\varepsilon,L_{K})$
needs to be controlled. Observing that $c(\varepsilon,L_{K})$ depends on $
\varepsilon$ through the distance of points $x_{1},x_{2}$ to $x_{0}$, note
that $\operatorname{cov}(\widehat{f}_{j;q}^{x_{0}}(x_{1}),\widehat{f}%
_{j;q}^{x_{0}}(x_{2}))$ grows rapidly as either of $x_{1}$ or $x_{2}$
approach $x_{0}$, whereas when $x_{1}$ and $x_{2}$ simultaneously
approach $%
x_{0}$, then the expression assumes the form of an indeterminate form, for
which one can use the standard l'H\^{o}pital's rule to get a precise
form of
the expression. Thus, let us first examine the following:
\begin{eqnarray*}
&&\lim_{x\rightarrow x_{0}}\operatorname{var}\bigl(\widehat{f}_{j;q}^{x_{0}}(x)
\bigr)
\\
&&\qquad=\lim_{x\rightarrow x_{0}}\frac{1}{(1-\rho(x_{0},x))^{2}}\\
&&\qquad\quad{}\times \bigl(1-\rho
^{2}(x_{0},x)+2\rho(x_{0},x)
\partial_{\theta_{x_{0}}}\rho (x_{0},x)\partial_{\theta_{x}}
\rho(x_{0},x)\partial_{\theta
_{x}}^{2}\rho(x,x)
\\
&&\hspace*{16pt}\qquad\quad{}+2\rho(x_{0},x)\partial_{\phi_{x_{0}}}\rho(x_{0},x)
\partial_{\phi
_{x}}\rho(x_{0},x)\partial_{\phi_{x}}^{2}
\rho(x,x)\\
&&\qquad\quad\hspace*{16pt}{}+\bigl\{\partial _{\theta
_{x_{0}}}\rho(x_{0},x)\bigr
\}^{2}\bigl\{\partial_{\theta_{x}}^{2}\rho(x,x)\bigr\}
^{3}
\\
&&\hspace*{123pt}\hspace*{16pt}\qquad\quad{}+\bigl\{\partial_{\phi_{x_{0}}}\rho(x_{0},x)\bigr\}^{2}
\bigl\{\partial_{\phi
_{x}}^{2}\rho(x,x)\bigr\}^{3}
\bigr).
\end{eqnarray*}

Let us do the limit computations for just the first term of the variance
expression:
\begin{eqnarray*}
&&\lim_{x\rightarrow x_{0}}\frac{1-\rho^{2}(x_{0},x)}{(1-\rho
(x_{0},x))^{2}%
}
\\
&&\qquad=\lim_{x\rightarrow x_{0}}\frac{-2\rho(x_{0},x)\partial_{\theta
_{x}}\rho(x_{0},x)}{(-2)(1-\rho(x_{0},x))\partial_{\theta_{x}}\rho
(x_{0},x)}
\\
&&\qquad=\lim_{x\rightarrow x_{0}}\frac{\rho(x_{0},x)\partial_{\theta
_{x}}^{2}\rho(x_{0},x)+ ( \partial_{\theta_{x}}\rho
(x_{0},x) )
^{2}}{ ( 1-\rho(x_{0},x) ) \partial_{\theta_{x}}^{2}\rho
(x_{0},x)- ( \partial_{\theta_{x}}\rho(x_{0},x) ) ^{2}}
\\
&&\qquad=\lim_{x\rightarrow x_{0}}\frac{\rho(x_{0},x)\partial_{\theta
_{x}}^{3}\rho(x_{0},x)+3\partial_{\theta_{x}}\rho(x_{0},x)\partial
_{\theta_{x}}^{2}\rho(x_{0},x)}{ ( 1-\rho(x_{0},x) )
\partial
_{\theta_{x}}^{3}\rho(x_{0},x)-3\partial_{\theta_{x}}\rho
(x_{0},x)\partial_{\theta_{x}}^{2}\rho(x_{0},x)}
\\
&&\qquad=\lim_{x\rightarrow x_{0}}\bigl(\rho(x_{0},x)\partial_{\theta
_{x}}^{4}\rho(x_{0},x)+\partial_{\theta_{x}}\rho(x_{0},x)\partial
_{\theta_{x}}^{3}\rho(x_{0},x)\\
&&\hspace*{55pt}{}+3\partial_{\theta_{x}}^{2}\rho
(x_{0},x)\partial_{\theta_{x}}^{2}\rho(x_{0},x)+3\partial_{\theta
_{x}}\rho(x_{0},x)\partial_{\theta_{x}}^{3}\rho(x_{0},x)\bigr)\\
&&\hspace*{23pt}\qquad\quad/\bigl( \bigl(
1-\rho
(x_{0},x) \bigr) \partial_{\theta_{x}}^{4}\rho(x_{0},x)-4\partial
_{\theta_{x}}\rho(x_{0},x)\partial_{\theta_{x}}^{3}\rho
(x_{0},x)\\
&&\hspace*{217pt}{}-3 \bigl( \partial_{\theta_{x}}^{2}\rho(x_{0},x) \bigr) ^{2}\bigr),
\end{eqnarray*}
where we have applied l'H\^{o}pital's rule at each step (four times),
and we
note that the final expression is indeed a nontrivial, determinate limit.

We note that we have assumed $\rho(x_{0},x)=P_{\mathbb{\ell}}(\langle
x_{0},x\rangle)$, and hence the derivatives above have the following form:
\begin{eqnarray*}
&&\partial_{\theta_{x}}P_{\mathbb{\ell}}\bigl(\langle x_{0},x\rangle
\bigr)=P_{%
\mathbb{\ell}}^{\prime}(\cdot) \bigl( \cos\theta_{x}
\sin\theta _{x_{0}}\cos(\phi_{x}-\phi_{x_{0}})-\sin
\theta_{x}\cos\theta _{x_{0}} \bigr),
\\
&&\partial_{\theta_{x}}^{2}P_{\mathbb{\ell}}\bigl(\langle
x_{0},x\rangle\bigr) =P_{%
\mathbb{\ell}}^{\prime\prime}(\cdot)
\bigl( \cos\theta_{x}\sin \theta _{x_{0}}\cos(
\phi_{x}-\phi_{x_{0}})-\sin\theta_{x}\cos\theta
_{x_{0}} \bigr) ^{2}
\\
&&\hspace*{32pt}\qquad\qquad{}+P^{\prime}(\cdot) \bigl( -\sin\theta_{x}\sin
\theta_{x_{0}}\cos (\phi _{x}-\phi_{x_{0}})-\cos
\theta_{x}\cos\theta_{x_{0}} \bigr),
\\
&&\partial_{\theta_{x}}^{3}P_{\mathbb{\ell}}\bigl(\langle
x_{0},x\rangle \bigr)
\\
&&\qquad=P_{\mathbb{\ell}}^{\prime\prime\prime}(\cdot) \bigl( \cos\theta _{x}\sin
\theta_{x_{0}}\cos(\phi_{x}-\phi_{x_{0}})-\sin\theta
_{x}\cos \theta_{x_{0}} \bigr) ^{3}
\\
&&\quad\qquad{}+2P_{\mathbb{\ell}}^{\prime\prime}(\cdot) \bigl( \cos\theta _{x}
\sin \theta_{x_{0}}\cos(\phi_{x}-\phi_{x_{0}})-\sin
\theta_{x}\cos \theta _{x_{0}} \bigr)\\
&&\qquad\qquad{}\times \bigl( -\sin
\theta_{x}\sin\theta_{x_{0}}\cos(\phi _{x}-
\phi_{x_{0}})-\cos\theta_{x}\cos\theta_{x_{0}} \bigr)
\\
&&\quad\qquad{}+P^{\prime}(\cdot) \bigl( -\cos\theta_{x}\sin
\theta_{x_{0}}\cos (\phi _{x}-\phi_{x_{0}})+\sin
\theta_{x}\cos\theta_{x_{0}} \bigr),
\\
&&\partial_{\theta_{x}}^{4}P_{\mathbb{\ell}}\bigl(\langle
x_{0},x\rangle \bigr)=P_{%
\mathbb{\ell}}^{(iv)}(\cdot) \bigl(
\cos\theta_{x}\sin\theta _{x_{0}}\cos(\phi_{x}-
\phi_{x_{0}})-\sin\theta_{x}\cos\theta _{x_{0}} \bigr)
^{4}
\\
&&\hspace*{52pt}\qquad{}+3P_{\mathbb{\ell}}^{\prime\prime\prime}(\cdot) \bigl( \cos\theta _{x}
\sin\theta_{x_{0}}\cos(\phi_{x}-\phi_{x_{0}})-\sin
\theta _{x}\cos \theta_{x_{0}} \bigr) \\
&&\hspace*{52pt}\quad\qquad{}\times\bigl( -\sin
\theta_{x}\sin\theta_{x_{0}}\cos (\phi_{x}-
\phi_{x_{0}})-\cos\theta_{x}\cos\theta_{x_{0}} \bigr)
\\
&&\hspace*{52pt}\qquad{}+2P_{\mathbb{\ell}}^{\prime\prime\prime}(\cdot) \bigl( \cos\theta _{x}
\sin\theta_{x_{0}}\cos(\phi_{x}-\phi_{x_{0}})-\sin
\theta _{x}\cos \theta_{x_{0}} \bigr) ^{2} \\
&&\hspace*{52pt}\quad\qquad{}\times\bigl( -
\sin\theta_{x}\sin\theta _{x_{0}}\cos (\phi_{x}-
\phi_{x_{0}})-\cos\theta_{x}\cos\theta_{x_{0}} \bigr)
\\
&&\hspace*{52pt}\qquad{}+2P_{\mathbb{\ell}}^{\prime\prime}(\cdot) \bigl( -\sin\theta _{x}
\sin \theta_{x_{0}}\cos(\phi_{x}-\phi_{x_{0}})-\cos
\theta_{x}\cos \theta _{x_{0}} \bigr) ^{2}
\\
&&\hspace*{52pt}\qquad{}-2P_{\mathbb{\ell}}^{\prime\prime}(\cdot) \bigl( \cos\theta _{x}
\sin \theta_{x_{0}}\cos(\phi_{x}-\phi_{x_{0}})-\sin
\theta_{x}\cos \theta _{x_{0}} \bigr) ^{2}
\\
&&\hspace*{52pt}\qquad{}+P^{\prime\prime}(\cdot) \bigl( \cos\theta_{x}\sin\theta
_{x_{0}}\cos (\phi_{x}-\phi_{x_{0}})-\sin
\theta_{x}\cos\theta_{x_{0}} \bigr)\\
&&\hspace*{52pt}\quad\qquad{}\times \bigl( -\cos
\theta_{x}\sin\theta_{x_{0}}\cos(\phi_{x}-
\phi_{x_{0}})+\sin \theta_{x}\cos\theta_{x_{0}} \bigr)
\\
&&\hspace*{52pt}\qquad{}+P^{\prime}(\cdot) \bigl( \sin\theta_{x}\sin
\theta_{x_{0}}\cos (\phi _{x}-\phi_{x_{0}})+\cos
\theta_{x}\cos\theta_{x_{0}} \bigr).
\end{eqnarray*}

Thus, we conclude that
\begin{eqnarray*}
P_{\mathbb{\ell}}\bigl(\langle x_{0},x\rangle\bigr) |_{x=x_{0}}
&=&1,
\\
\partial_{\theta_{x}}P_{\mathbb{\ell}}\bigl(\langle x_{0},x\rangle
\bigr) |%
_{x=x_{0}} &=&0,
\\
\partial_{\theta_{x}}^{2}P_{\mathbb{\ell}}\bigl(\langle
x_{0},x\rangle \bigr) |%
_{x=x_{0}}
&=&-P^{\prime}(1),
\\
\partial_{\theta_{x}}^{3}P_{\mathbb{\ell}}\bigl(\langle
x_{0},x\rangle \bigr) |%
_{x=x_{0}} &=&0,
\\
\partial_{\theta_{x}}^{4}P_{\mathbb{\ell}}\bigl(\langle
x_{0},x\rangle \bigr) |%
_{x=x_{0}},
&=&2P_{\mathbb{\ell}}^{\prime\prime}(1)+P_{\mathbb{\ell}%
}^{\prime}(1).
\end{eqnarray*}

Subsequently, we shall argue that by continuity, and the fact the field
$%
\widehat{f}^{x_0}_{j;q}$ appears to be singular at $x_0$, we conclude that
for $x_1,x_2\in B(x_0,\varepsilon)$ and a small enough $\varepsilon$,
\begin{eqnarray*}
&&\sup_{x_1,x_2\in B(x_0,\varepsilon)}\mathbb{E} \bigl( \widehat{f}%
_{j;q}^{x_0}(x_2)
- \widehat{f}_{j;q}^{x_0}(x_1)
\bigr)^2\\
&&\qquad = \lim_{(x_1,x_2)\to(x_0,x_0)}\mathbb{E} \bigl( \widehat
{f}_{j;q}^{x_0}(x_2) - \widehat{f}_{j;q}^{x_0}(x_1)
\bigr)^2.
\end{eqnarray*}

The limit on the right-hand side can again be evaluated by applying l'H\^{o}pital's
rule, and thus, the (uniform) Lipschitz behaviour is justified. Thereafter,
we note that by the isotropy of the underlying field $f_{j;q}$, the
$\mathbb{%
E} (\sup_{x\in S^{2}\setminus\{x_{0}\}}\widehat
{f}_{j;q}^{x_{0}}(x)%
 )$ does not depend on $x_{0}$, and thus we get a uniform (over
$j$ and
$x_{0}$) Lipschitz bound, as claimed.
\end{appendix}

\section*{Acknowledgements} We are grateful to the two referees for
their constructive comments, which helped us improve the readability of the\break 
paper.


%




\printaddresses
\end{document}